\documentclass[a4]{article}
\usepackage{amsthm}
\usepackage{amsmath}
\usepackage{amssymb}
\usepackage{color}

\title{Fast growth of the number of periodic points arising from
 heterodimensional connections}
\author{Masayuki Asaoka, Katsutoshi Shinohara, and Dmitry Turaev}

\def\RR{\mathbb{R}}
\def\ZZ{\mathbb{Z}}

\def\dd{{\mathbf{d}}}

\def\cC{{\mathcal C}}
\def\cD{{\mathcal D}}

\def\cF{{\mathcal F}}

\def\cN{{\mathcal N}}
\def\cP{{\mathcal P}}
\def\cU{{\mathcal U}}
\def\cV{{\mathcal V}}
\def\cW{{\mathcal W}}

\def\cO{{\mathcal O}}

\def\Id{{\rm Id}}

\def\ra{{\rightarrow}}

\def\vphi{\varphi}

\def\loc{{\rm{loc}}}
\def\bl{{\rm{bl}}}

\def\hsp{{\hspace{3mm}}}

\DeclareMathOperator{\Diff}{Diff}

\DeclareMathOperator{\GL}{GL}
\DeclareMathOperator{\Mat}{{Mat}}

\DeclareMathOperator{\Fix}{Fix}

\DeclareMathOperator{\Per}{Per}

\DeclareMathOperator{\Int}{Int}

\DeclareMathOperator{\supp}{supp}

\DeclareMathOperator{\sgn}{sgn}

\DeclareMathOperator{\Ker}{Ker}

\newcommand\cl[1]{{\overline{#1}}}

\theoremstyle{plain}
\newtheorem{thm}{Theorem}[section]
\newtheorem{prop}[thm]{Proposition}
\newtheorem{lemma}[thm]{Lemma}
\newtheorem{cor}[thm]{Corollary}

\newtheorem*{Mthm}{Main Theorem}

\theoremstyle{definition}
\newtheorem{dfn}[thm]{Definition}

\theoremstyle{remark}
\newtheorem{rmk}[thm]{Remark}

\begin{document}
\maketitle

\begin{abstract}
We consider $C^r$-diffeomorphisms of a compact smooth manifold 
having a pair of robust heterodimensional cycles where $1 \leq r \leq +\infty$. 
We prove that if certain conditions about the signatures of 
non-linearities and Schwarzian derivatives of the transition maps 
are satisfied, then by giving $C^r$ arbitrarily small perturbation, 
we can produce a periodic point at which the first return map 
in the center direction is $C^r$-flat. 
As a consequence, we will prove that 
$C^r$-generic diffeomorphisms in the neighborhood 
of the initial diffeomorphism exhibit 
super-exponential growth of number of periodic points. 
We also give examples which show the necessity 
of the conditions on non-linearities and the Schwarzian derivatives. 

{ \medskip

}
\end{abstract}

%
%

\section{Introduction}
In this paper 
we prove the $C^\infty$-generic super-exponential growth
 of the number of periodic points
 for a class of partially hyperbolic diffeomorphisms.
The growth rate of the number of periodic points
 as a function of their period appears to be determined
 by an interplay between dynamical properties and the regularity of the map.
The classical result by Artin and Mazur \cite{AM} is
 that for a dense set of smooth diffeomorphisms
 the growth is at most exponential with the period,
 independently of the type of the dynamics of the map.
On the other hand, for any axiom A diffeomorphism,
 the existence of a finite Markov partition also implies
 the at most exponential growth,
 independently of the regularity class of the diffeomorphism \cite{Bow}.
For smooth maps of an interval,
 the condition for the exponential growth is the non-flatness of
 all critical points \cite{MMS, KK}, 
 i.e., it is a regularity type condition.

In dimension two, however, 
when a diffeomorphism is not axiom A
 and belongs to the so-called Newhouse domain
 in the space of smooth diffeomorphisms
  where maps with homoclinic tangencies are dense,
 $C^\infty$-generic diffeomorphisms exhibit super-exponential
 growth of the number of periodic points.
Namely, it was discovered by Kaloshin \cite{Ka}
 that given any candidate upper bound for the growth rate of
 the number of periodic points
 this bound will be exceeded by a generic diffeomorphism
 from the Newhouse domain.
Thus, for a generic smooth non-hyperbolic map in dimension two and higher
 the uncontrollable growth of the number of periodic points appears to
 be a dynamical property,
 almost independent of the regularity of the map.
In the case of real-analytic non-hyperbolic maps the situation is less clear.
It is shown in \cite{Asa} that
 for a class of real-analytic area-preserving maps
 with an invariant KAM-circle
 an uncontrollable super-exponential growth of the number
 of periodic points is a dense phenomenon,
 but it is not known what will happen in the real-analytic case
 if the area-preservation property is dropped.

Here we focus on a different class of non-hyperbolic systems,
 namely, we consider partially hyperbolic diffeomorphisms
 with one-dimensional central direction.
Such maps cannot have homoclinic tangencies,
 so the above described results are not applicable.
In fact, for this case, the balance between dynamics and regularity
 in the question of the growth of the number of periodic points is
 shifted in a peculiar way.
Dynamically, the super-exponential growth here
 is related to persistent heterodimensional cycles:
 they cause the super-exponential growth of the number of periodic orbits 
 $C^1$-generically, see \cite{BDF}.
For a subclass where the heterodimensional cycle is
 embedded into a certain normally-hyperbolic invariant fibration by circles,
 the super-exponential growth is shown to be $C^\infty$-generic \cite{Be}.
For free semi-group actions on an interval,
which is a simplified model for partially hyperbolic diffeomorphisms we discuss in this paper,
 open conditions for a $C^\infty$-generic super-exponential growth are
 given in \cite{AST}.
However, in the same paper, we constructed  $C^2$-open and
 $C^3$-open classes of semi-group actions with heterodimensional cycles
 where the growth of the number of periodic orbits is at most exponential,
 even though these classes lie in the $C^1$-interior of the set of systems
 where the super-exponential growth is $C^1$-generic.  

The present paper is a sequel of \cite{AST}.
We show that main constructions can be transferred
 to the general case of partially hyperbolic diffeomorphisms.
Such generalization is non-trivial in several respects.
In particular, we do not require a large spectral 
gap assumption,
 so we do not have a smooth center foliation,
 {\it i.e.}, there is no reduction to a smooth skew-product system.  
Now, let us give rough description of our 
setting. The precise statement is given 
in Section~\ref{sec:outline}.

\medskip
{\bf Heteroclinic pairs.}
Let $n \geq 3$ and $M$ be a closed 
smooth $n$-dimensional manifold,.
Let $r \geq 2$, $\mathrm{Diff}^r(M)$ be the space of 
$C^r$-diffeomorphisms with $C^r$-topology and $f\in  \mathrm{Diff}^r(M)$.
Let $(p_1, p_2)$ and $(p_3, p_4)$ be two pairs of hyperbolic periodic points of $f$
(we do not exclude the case where 
$p_1 = p_3$ or $p_2 =p_4$). 
We assume that the unstable manifolds of $p_1$ and $p_3$
 have the same dimension 
$\mathrm{u}\text{-}\mathrm{ind}(p_1)= 
\mathrm{u}\text{-}\mathrm{ind}(p_3)=d+1$ and 
the unstable manifolds of $p_2$ and $p_4$ have the same dimension 
$\mathrm{u}\text{-}\mathrm{ind}(p_2) = 
\mathrm{u}\text{-}\mathrm{ind}(p_4) = d$.
We assume the existence of two heteroclinic points
$q_i \in W^u(p_{2i-1})\cap W^s(p_{2i})$, $i=1, 2$,
 so the orbit of $q_1$ tends to $p_2$ at forward iterations of $f$
 and to $p_1$ at backward iterations,
 while the orbit of $q_2$ tends to $p_4$ at forward iterations 
 and to $p_3$ at backward iterations.

{\bf Partial hyperbolicity.}
We assume that each periodic point $p_i$, $i=1,\dots,4$ admits
 a partially hyperbolic splitting $E^{uu} \oplus E^c \oplus E^{ss}$
 where $\dim E^{uu} =d$ and $\dim E^c =1$
 (note that the subspace $E^c$ corresponds to
 the direction of the weakest expansion for $p_1$ and $p_3$
 and the weakest contraction for $p_2$ and $p_4$). 
This splitting is assumed to be extended to
 the neighborhood of the heteroclinic orbits $\cO(q_1)$ and $\cO(q_2)$.
We assume that the center direction has an orientation 
 and it is preserved under the iterations of $f$.  

{\bf Signatures of the heteroclinic orbits.}
The partial hyperbolicity guarantees that the intersection 
 $W^u(p_{2i-1}) \cap W^s(p_{2i})$ is transverse near $q_i$ and
 is locally a one-dimensional curve $\ell_i$ tangent to $E^c$, $i=1,2$. 
We consider the restriction of $f$ to $\ell_i$,
 which gives  a one-dimensional $C^r$-map.
Then, following \cite{AST},  
 we can introduce the notion of
 the ``signature of the heteroclinic orbits''
 -- it is a pair of signatures of certain combinations
 of the derivatives up to order 3 of some iteration of the map $f|_{\ell_i}$,
 see Section \ref{ss:signature}.
We assume  that $q_1$ and $q_2$ have {\em opposite signatures}.

{\bf Blender.}
We assume that $f$ has a {\em blender}.
 Blender is a dynamical structure which produces
 robust connections by pseudo-orbits
 between invariant manifolds
 for which the sum of dimensions is lower than
 the dimension of the ambient space.
We provide a precise description in Section \ref{ss:blender}.
A blender is called $C^r$-robust
 if every diffeomorphism $g$ from a $C^r$-neighborhood of $f$
 has a blender that, in a certain sense, depends continuously on $g$.
We assume that there is a $C^r$-robust blender
 which ``links'' the points $\{p_i\}_{i=1, \ldots, 4}$.
Thus the set of points $p_i, q_i$ are
 all in the same chain-recurrence class $C^r$-robustly.

Furthermore, we assume that the partial hyperbolicity around $p_i$ and $q_i$
 is also extended to the neighborhood of the blender
 and the center direction has an orientation
 which is compatible with the iterations of $f$.  

\medskip

The following is a rough description of our main result.

\begin{Mthm}
Take $r=1, \ldots, +\infty$
 and let $\cW^r \subset \mathrm{Diff}^r(M)$ be
 the open set of diffeomorphisms satisfying above conditions. 
Then for every sequence $(a_i) \subset \mathbb{N}$
 there exists a $C^r$-residual set 
 $\mathcal{R}^r=\mathcal{R}^r_{(a_i)} \subset \cW^r$
 such that for every $f \in \mathcal{R}^r$ we have 
\[
\limsup_{n \to +\infty} 
 \frac{ \#\{x \in M\mid f^n(x) =x \}}{a_n} = +\infty.  \qquad  (\star)  
\]
\end{Mthm}
We say that a diffeomorphism is \emph{super-exponential}
 if the condition $(\star)$ holds for some sequence $(a_n)$
 which grows more rapidly than any exponential function. 
Our result establishes the generic super-exponential growth.
 We point out that one cannot improve the statement (say, to open and dense)
 because of the above mentioned result of Artin and Mazur \cite{AM},
 which implies that every diffeomorphism can be
 $C^r$-approximated by one with at most exponential growth.

Notice that the case $r=1$ is dealt in a paper
 by Bonatti, D\'{i}az, and Fisher \cite{BDF}. 
The case $r \geq 2$ is quite different in nature,
as it employs the information about the higher order derivatives.
As we will see later in Section~\ref{sec:examples}, 
the condition on the signatures of heteroclinic orbits
are essential: we will give examples
which show that this condition can be indeed necessary.
\bigskip

Let us explain the idea of the proof. 
We say that a periodic point $p$ of $f$ is $r$-flat in the central direction
 if it has one-dimensional center manifold
 and the $r$-jet of the first return map restricted to the center manifold
 is identity, see Section~\ref{sec:outline} for the detail.
It is easy to see that
 if $f \in \mathrm{Diff}^r(M)$ has an $r$-flat periodic point of period $\pi$,
 then by adding a perturbation, arbitrarily small in  $C^r$,
 we can create as many points of period $\pi$ as we want.
 Thus, our theorem is the consequence of the following perturbation result.
\begin{prop}
\label{p.flatproduction}
Let $f\in \cW^r$ be a $C^{\infty}$ diffeomorphism. 
For any $n_0\geq 1$ and $r \geq 1$,
 arbitrarily close to $f$ in the $C^{\infty}$ topology,
 there exists  $g\in \cW^r$ which has  a periodic point,
 $r$-flat in the central direction,
 whose least period is greater than $n_0$. 
\end{prop}
Once this statement is proven,
 a standard genericity argument as in \cite{Ka, AST}
 leads to the conclusion of the theorem.

We prove Proposition~\ref{p.flatproduction}
by carefully investigating the local behavior of 
points around \emph{heterodimensional cycles}
together with our construction from \cite{AST}. 
The analysis we carry out has similar flavor as
done in the serial researches by D\'{i}az 
and D\'{i}az-Rocha (see for instance \cite{D, DR}).
We begin with the initial heteroclinic cycles. 
The saddle $p_{2i-1}$ is connected to the saddle $p_{2i}$ by
the orbit of the heteroclinic point $q_i$, and $p_{2i}$ is connected to $p_{2i-1}$ via blender, $i=1,2$. 
We show that by adding an arbitrarily $C^{r}$-small perturbation to $f$ one can create 
a 1-flat periodic point near any of these cycles.
This 1-flatness is obtained by ``cross breeding'' 
the expanding behavior of $p_{2i-1}$ and the contracting 
behavior of $p_{2i}$ in the center direction. 
As the cycles are robust, 
we can repeat the procedure as many times as we want
and accordingly we can create 
as many 1-flat periodic points as we want. We also prove that 
these newly created periodic points are all connected through blenders.  

Next, we create $2$-flat periodic points 
 by adding perturbations near heteroclinic connections
 of the $1$-flat periodic points we have obtained.
Since the number of $1$-flat periodic points can be as large as we want,
 we also have as many $2$-flat periodic points as we want.
Then we create heteroclinic connections with many $3$-flat points
 and after that we continue inductively
 until we obtain an $r$-flat periodic point.
Note that $1$-flat, $2$-flat, $3$-flat, and $4$-flat periodic points are
 obtained by different procedures,
 while the procedure is the same for $k$-flat points starting with $k\geq 4$.

In the course of this induction, we need to take care of the following matters.
\begin{itemize}
\item We need to have control of the support of the perturbation
 so as to guarantee that each perturbation does not destroy
 previously constructed flat points. 
\item For the construction of $k$-flat points
 from the heteroclinic network of $(k-1)$-flat points with $k=2$ and $k=3$,
 we need to control the sign of the second derivative and,
 respectively, the Schwarzian derivative of the 
 first return map for the $(k-1)$-flat points.
Thus, while creating $(k-1)$-flat points on the previous step,
 we also need to control these derivatives.
\end{itemize}
These two difficulties are already appeared
 in the one-dimensional semi-group action case,
 and the way we solve this problem is similar to what was done in \cite{AST}.
In the multi-dimensional case we consider here,
 we have another substantially new problem:
\begin{itemize}
\item the holonomy between center leaves can be non-smooth,
 so there is no direct reduction of the dynamics
 (projected to the center direction)
 to those generated by iterations of a system of smooth maps of the interval. 
\end{itemize}
Overcoming this problem is one of the most demanding topics of this paper.

Note also that we have here a problem of creating connections
 between stable manifolds and unstable manifolds of periodic points
 by $C^r$-small perturbations
 (this is an integral part of our induction argument).
By the condition of the theorem,
 we know they are connected by pseudo-orbit through the blender. 
To transform the weak connection to a true heteroclinic intersection, 
 the known technique is the connecting lemma by Hayashi \cite{H},
 which is valid only for $C^1$-topology and cannot be used for our purpose. 
We circumvent this problem
 by reviewing the properties of the blender.
The weak connections which the blender produces are local in a sense,
 i.e., it does not have any intermediate orbits outside the blender.
This makes enough room for the perturbation,
 enabling us to obtain an appropriate $C^r$-connecting result. 

Finally, we explain the organization of this paper. 
In Section 2 we discuss basic properties of the objects employed,
 that is, dynamics around partially hyperbolic periodic points,
signatures of heteroclinic points, and blenders. 
In Section 3 we give the precise statement of the main theorem.
We also state several propositions
 which lead to the proof of the main theorem
 and discuss the general scheme of the proof. 
In Section 4 we prove a perturbation result
 which is used to produce flat periodic points. 
In Section 5 we discuss the construction of the one-flat periodic points.
In Section 6 we prove a perturbation result
 which produces $(r+1)$-flat periodic point
 from a heteroclinic network of $r$-flat periodic points.   
In Section 7 we discuss several examples
 which elucidate the importance of the assumptions about 
 the signatures in the Main Theorem.

{\bf Acknowledgments.} This work was supported by the grants
 14-41-00044 of RSF, EPSRC grant EP/P026001/1,
 the Royal Society grant IE141468, 
 JSPS KAKENHI Grant Numbers 26400085, 16k17609, and 18k03357.
The authors also acknowledge the support
 by JSPS Bilateral Open Partnership Joint Research Projects.


%
%

\section{Preliminaries}\label{s:prelimi}

In this section, we discuss basic definitions and notation necessary for giving the precise statement of the Main Theorem.

\subsection{Basic notation}\label{ss:basic}
Let $(M,p)$ denote a smooth manifold $M$ with a point $p\in M$.
By a \emph{local} map $f:(M_1,p_1) \ra (M_2,p_2)$
 we will mean a map defined in a neighborhood of $p_1\in M_1$,
 such that $f(p_1)=p_2$.
In the rest of this subsection,
 we consider the case where $M_1 = M_2 = \RR$ and $p_1 =p_2 =0$.

Let $F:(\RR,0) \ra (\RR,0)$ be a local $C^r$-map with $r \geq 1$.
By $F^{(s)}$, we denote the $s$-th derivative of $F$ at the origin.
We will also write $F'$, $F''$, and $F'''$ for the first, second,
 and third derivatives respectively.

For $1 \leq r \leq \infty$,
 let $\Diff_\loc^r(\RR,0)$ be the group of local $C^r$-maps
 $F:(\RR,0) \ra (\RR,0)$ with $F' \neq 0$.
We say that $F \in \Diff^r(\RR,0)$ is \emph{$1$-flat} if $F'(0)=1$.
For $2 \leq s \leq r$, we also say that $F$ is 
\emph{$s$-flat}
 if it is $1$-flat and $F^{(j)}=0$ for $2\leq j \leq s$.
The \emph{non-linearity} $A(F)$ and the \emph{Schwarzian derivative} $S(F)$
 for $F \in \Diff_\loc^r(\RR,0)$ are defined as follows:
\begin{equation}
\label{asdef}
A(F) = \frac{F''}{F'}, 
 \quad
S(F) = \frac{F'''}{F'}
  -\frac{3}{2}\left(\frac{F''}{F'}\right)^2.
\end{equation}
For $F,G \in \Diff_\loc^r(\RR,0)$,
 we have the following {\it cocycle property} of the non-linearity
 and the Schwarzian derivative (by direct computation):
\begin{align}
\label{cocpr}
 A(G \circ F)  =A(G) \cdot F' + A(F), \quad
 S(G \circ F)  =S(F) \cdot (F')^2 + S(F).
\end{align}
This implies that
\begin{align*}
 A(F^{-1})  = -A(F) \cdot (F')^{-1}, \quad 
 S(F^{-1})  = -S(F) \cdot (F')^{-2}.
\end{align*}
Applying these formulas,
 we obtain the following result (we leave the proof to the reader):
\begin{lemma}
\label{lemma:A,S conjugacy}
Let $F, H \in \Diff_\loc^3(\RR,0)$.
If $F$ is $1$-flat, then
\begin{align*}
 A(H^{-1} \circ F \circ H) = A(F) \cdot H', \quad 
 S(H^{-1} \circ F \circ H) = S(F) \cdot (H')^2.
\end{align*}
In particular, the signs $\sgn A(F)$ and $\sgn S(F)$ are not changed
 by a conjugacy by an orientation-preserving local $C^3$-diffeomorphism.
\end{lemma}
\begin{rmk}
\label{rmk:AS}
Note that if $F$ is not $1$-flat,
 then $\sgn A(F)$ and $\sgn S(F)$ are not invariants of the conjugacy.
Indeed, in such case $F$ can be linearized by a smooth conjugacy,
 and both $A(F)$ and $S(F)$ vanish for linear maps.
\end{rmk}

For $r \geq 1$,
 let $\cP^r(\RR,0)$ be the set of real polynomials $P(t)$
 in one variable with $P(0)=0$ and $\deg(P) \leq r$.
We define the norm $\|P\|_r$ for $P=a_1t+\dots a_r t^r \in \cP^r(\RR,0)$ as
\begin{equation*}
\|P\|_r=|a_1|+\dots+|a_r|. 
\end{equation*}
We equip $\cP^r(\RR,0)$ with the topology induced by this norm.

\subsection{$c$-oriented transverse pairs of invariant cones}
Next, we give the definition of the pairs of invariant cones
 which describes the partial hyperbolicity of our dynamics
 and discuss the notion of their center orientation ($c$-orientation).

We call the triple $\dd=(d_c,d_s,d_u)$ of positive integers the {\it index}.
Throughout this paper, we fix the index $\dd=(d_c, d_s ,d_u)$ with $d_c=1$
 and denote $|\dd|=1+d_s+d_u$.
Let $M$ be a compact $|\dd|$-dimensional manifold
 and $U$ be an open subset of $M$.
Let $\|\cdot\|$ be a continuous (in $x$) 
metric of $T_xM$ for $x\in U$.
Let $\tilde{E}^c$, $\tilde{E}^s$, and 
 $\tilde{E}^u$ be continuous subbundles of $TM|_U$
 of dimensions $d_c$, $d_s$, and $d_u$,
 respectively such that
 $TM|_{U}=\tilde{E}^c \oplus \tilde{E}^s \oplus \tilde{E}^u$.
We say that the triplet $(\tilde{E}^c, \tilde{E}^s, \tilde{E}^u)$ gives
 a continuous splitting of $TM|_U$.
Given a positive constant $\alpha<1$,
 define the {\it transverse pair of cone fields} $(\cC^{cs},\cC^{cu})$
 {\it of index} $\dd$  as follows:
 for every $x\in U$
\begin{equation}
\label{ccsccu}
\begin{array}{ll}
\cC^{cs}(x) & =
 \{v=v^c+v^s+v^u \in T_x M \mid \|v^u\|\leq \alpha(\|v^c\|+\|v^s\|)\},\\
\cC^{cu}(x) & =
 \{v=v^c+v^s+v^u \in T_x M \mid \|v^s\|\leq \alpha(\|v^c\|+\|v^u\|)\},
\end{array}
\end{equation}
 where $v^c$, $v^s$, and $v^u$ denote the $\tilde{E}^c$
 (resp., $\tilde{E}^s$ and $\tilde{E}^u$) components of $v \in T_x M$.
For such a pair, we define {\it the center cone} $\cC^{c}(x)$ at $x$ as
\begin{equation}
\label{ccc}
 \cC^{c}(x)=\cC^{cs}(x) \cap C^{cu}(x). 
\end{equation}
Obviously,
 $\cC^{c}(x) \cap (\tilde{E}^s \oplus \tilde{E}^u)(x)=\{0\}$.
A {\it $c$-orientation} of the pair $(\cC^{cs},\cC^{cu})$
 is a continuous $1$-form $\omega$ over $U$
 such that $\Ker \omega \cap \cC^c(x) =\{0\}$.
Remark that $\cC^{c}(x)\setminus \{0\}$ consists of
 two connected components $\cC^c_+(x)$ and $\cC^c_-(x)$,
 where $\omega(v)>0$ for any $v\in \cC^c_+(x)$
 and $\omega(v)<0$ for any $v\in \cC^c_-(x)$.
We call the set $\cC^c_+(x)$ {\it the positive half} of $\cC^c(x)$.

Let $f$ be a $C^1$ diffeomorphism of $M$.
We say that a transverse pair $(\cC^{cs}, \cC^{cu})$ of cone fields
 is \emph{$f$-invariant} if
 $Df(\cC^{cu}(x)) \subset \Int \cC^{cu}(f(x)) \cup \{0\}$,
 $Df^{-1}(\cC^{cs}(f(x)) \subset \Int \cC^{cs}(x) \cup \{0\}$
 for any $x \in U \cap f^{-1}(U)$.
We also say that a $c$-orientation $\omega$
 of an $f$-invariant pair $(\cC^{cs}, \cC^{cu})$ is $f$-invariant
 if $\omega((Df)_x (v_c))>0$
 for all $x \in U \cap f^{-1}(U)$ and all
 $v_c \in \cC^c(x) \cap Df^{-1}(\cC^c(f(x)))$  with $\omega(v_c)>0$.

It is obvious that
 the invariance of $(\cC^{cs}, \cC^{cu})$ and $\omega$ persists
under $C^1$-small perturbation of $f$ in the following sense.
\begin{lemma}
\label{lemma:cone persistence}
Let $f$ be a $C^1$ diffeomorphism of $M$,
 $(\cC^{cs},\cC^{cu})$ be an $f$-invariant transverse pair of cone fields defined in
 an open set $U \subset M$, and $\omega$ be its $f$-invariant $c$-orientation.
Then, for every open set $U'$ with $\cl{U'} \subset U$, 
there exists a $C^1$-neighborhood $\mathcal{U}$ of $f$ in $\mathrm{Diff}^1(M)$ such that
$(\cC^{cs}, \cC^{cu})$ and $\omega$ restricted to $U'$ are $\bar{f}$-invariant for any $\bar{f} \in \cU$.
\end{lemma}

For $k \geq 1$,
 we say that an $f$-invariant compact subset $\Lambda$ of $M$
 admits {\it an $k$-strongly partially hyperbolic splitting}
 $TM|_\Lambda=E^c \oplus E^s \oplus E^u$ of index $\dd$
 if it is a $Df$-invariant continuous splitting with
 $\dim E^\tau=d_\tau$ for $\tau=c,s,u$
 and there exist a metric $\|\cdot\|$ 
 of $T_xM$ which is continuous in $x$
 and a constant $0<\lambda<1$ such that for every $x \in \Lambda$ we have
\begin{align*}
 \|(Df)_x|_{E^s}\| \cdot \max\{1,\|(Df)_{f(x)}^{-1}|_{E^c}\|^k\} & <\lambda, \\
 \|(Df)^{-1}_{f(x)}|_{E^u}\| \cdot \max\{1,\|(Df)_x|_{E^c}\|^k\} & <\lambda.
\end{align*}
In this case, we call the set $\Lambda$
 {\it a $k$-strongly partially hyperbolic invariant set of index $\dd$}.
A $1$-strongly partially hyperbolic set
 will be called simply a strongly partially hyperbolic set.

On a strongly partially hyperbolic invariant set,
 an $f$-invariant transverse pair of cone fields is compatible
 with the partially hyperbolic splitting.
\begin{prop}
\label{prop:cone PH}
Let $f$ be a $C^1$ diffeomorphism of $M$
 and $(\cC^{cs},\cC^{cu})$ be
 an $f$-invariant transverse pair of cone fields defined on $U$.
Suppose that $f$ has a compact,
 strongly partially hyperbolic invariant set $\Lambda$ in $U$.
Then, the partially hyperbolic splitting
 $TM|_\Lambda=E^c \oplus E^s \oplus E^u$ satisfies
\begin{enumerate}
 \item  $(E^c\oplus E^s)(x) \subset \Int \cC^{cs}(x) \cup \{0\}$,
 $(E^c\oplus E^u)(x) \subset \Int \cC^{cu}(x) \cup \{0\}$, and
 \item $E^u(x) \cap \cC^{cs}(x)=E^s(x) \cap \cC^{cu}(x)=\{0\}$
\end{enumerate}
 for all $x \in \Lambda$.
\end{prop}
\begin{proof}
Let $TM|_\Lambda=\tilde{E^c} \oplus \tilde{E^s} \oplus \tilde{E^u}$
be the splitting in the definition of the transverse pair of cones,
 while $E^c$, $E^s$, and $E^u$ be the subbundles of $TM|_\Lambda$
 from the definition of strong partial hyperbolicity.
Denote $E^{cu}=E^c\oplus E^u$ and $E^{cs}=E^c\oplus E^s$.

Take any $x'\in\Lambda$ and let $E(x')$ be
 any $(1+d_u)$-dimensional subspace of $\Int \cC^{cu}(x')\cup\{0\}$
 which is transverse to $E^s(x')$
 (such subspace $E$ exists from the dimension count).
The strong partial hyperbolicity on $\Lambda$ implies that
 the distance between $Df^n(E^{cu}(x'))$ and $Df^n(E(x'))$ converges to zero
 in the Grassmanian bundle as $n \to \infty$.
Thus, by the invariance of $E^{cu}$ and $\cC^{cu}$,
 for each $x' \in  \Lambda$, and all $n$ sufficiently large, 
we have $E^{cu}(f^n(x'))) \subset \Int \cC^{cu}(f^n(x')) \cup \{0\}$.
By the compactness of $\Lambda$
 and the continuity of $E^{cu}(x')$ and $\cC^{cu}(x')$,
 this holds true for some $n$ independent of $x'\in\Lambda$.
Now, by taking $x=f^{-n}(x')$,
 we obtain that $E^{cu}(x) \subset \Int\cC^{cu}(x) \cup \{0\}$
 for each $x \in \Lambda$.

In the same way
 (by iterating $f^{-1}$ instead of $f$)
 one proves that $E^{cs}(x) \subset \Int \cC^{cs}(x) \cup \{0\}$
 for each $x\in\Lambda$, thus finishing the proof of assertion 1.
Note that this implies also that $E^{cs}(x)$ is transverse to $\tilde E^u(x)$,
 as follows from the definition of $\cC^{cs}$, see (\ref{ccsccu}).
Therefore, taking any point $x'\in\Lambda$,
 we infer from the definition of strong partial hyperbolicity
 that the distance between $Df^n(E^{u}(x'))$ and $Df^n(\tilde E^u(x'))$
 converges to zero in the Grassmanian bundle as $n \to \infty$.
Since the complement to $\cC^{cs}\setminus\{0\}$ is $Df$-invariant
 and $\tilde E^u(x')\cap \cC^{cs}(x')=\{0\}$,
 it follows that $Df^n(\tilde E^u(x'))$ lies in the complement to
 $\cC^{cs}(f^n(x'))\setminus\{0\}$.
Hence, the same is true for $E^{u}(f^n(x'))$ for all sufficiently large $n$.
As $n$ can be chosen the same for all $x'\in \Lambda$,
 we conclude that $E^u(x) \cap \cC^{cs}(x)=\{0\}$ for all $x\in \Lambda$.
By iterating $f^{-1}$ instead of $f$,
 we prove $E^s(x) \cap \cC^{cu}(x)=\{0\}$ in the same way,
 thus finishing assertion 2.
\end{proof}

As an immediate corollary, we have the following
\begin{cor}
\label{cor:cone PH}
Let $f$ be a $C^1$ diffeomorphism of $M$ and
 $(\cC^{cs},\cC^{cu})$ be an $f$-invariant transverse pair
 of cone fields on $U$.
Suppose that $f$ admits a compact, strongly partially hyperbolic invariant set
 $\Lambda$ in $U$ with the partially hyperbolic splitting
 $TM|_\Lambda=E^c \oplus E^s \oplus E^u$.
Then, $E^c(x) \subset \cC^c(x)=\cC^{cu}(x)\cap \cC^{cs}(x)$
 for all $x \in \Lambda$.
\end{cor}

\subsection{Invariant subbundles on the stable set of a periodic points}
\label{ss:invari}
Let $M$ be a $|\dd|$-dimensional manifold, 
$f \in \mathrm{Diff}^r(M)$, and $k\geq 1$ be an integer.
We say that a point $p\in M$ is
 \emph{a $k$-strongly partially hyperbolic periodic point of index $\dd$}
 if it is a periodic point of $f$
 and its orbit $\cO(p,f) = \{ f^k(p)\}_{k \in \ZZ} $ is
 a $k$-strongly partially hyperbolic set of index $\dd$ for $f$.
When $k=1$, we just call $p$ a strongly  partially hyperbolic periodic point.
Let $\Per_\dd^k(f)$ be the set of
 all $k$-strongly partially hyperbolic periodic points of index $\dd$
 and put $\Per_\dd(f)=\Per_\dd^1(f)$.

Let $p \in \Per_\dd(f)$.
We denote the period of $p$ by $\pi(p)$.
Recall that the space $E^c(p)$ is one-dimensional, so
\begin{equation}
\label{lamc}
(Df^{\pi (p)})_p (v) = \lambda_c(p) v \;\text{ for }\; v\in E^c(p),
\end{equation}
 with a non-zero constant $\lambda_c(p)$.
 We will call $\lambda_c$ the {\em central multiplier} of $p$.

We define the \emph{stable set} $W^s(p)$ and
 the \emph{unstable set} $W^u(p)$ of the periodic point $p$ by 
\begin{align*}
 W^s(p) & =
 \left\{q \in M \mid d(f^n(p),f^n(q)) \ra 0 \, (n \ra +\infty)
 \right\},\\
 W^u(p) & =
 \left\{q \in M \mid d(f^n(p),f^{-n}(q)) \ra 0 \, (n \ra +\infty)
 \right\},
\end{align*}
where $d$ is a metric on $M$.
Put
\begin{equation*}
 W^s(\cO(p))  = \bigcup_{j=0}^{\pi-1} W^s(f^j(p)), \quad
 W^u(\cO(p))  = \bigcup_{j=0}^{\pi-1} W^u(f^j(p)),
\end{equation*}
 where $\pi$ is the period of $p$.
For $\delta>0$,
  we also define the \emph{local stable and unstable sets $W^s_\delta(p)$ and $W^u_\delta(p)$} by
\begin{align*}
 W^s_\delta(p) & = \left\{q \in W^s(p)
 \mid \sup_{n \geq 0}d(f^n(p),f^n(q)) \leq \delta \right\},
 \\
 W^u_\delta(p) & = \left\{q \in W^u(p)
  \mid \sup_{n \leq 0}d(f^n(p),f^n(q)) \leq \delta\right\}.
\end{align*}
Remark that we have the following:
\begin{equation*}
 W^s(\cO(p))=\bigcup_{n \leq 0}f^{-n}(W^s_\delta(p)), \quad
 W^u(\cO(p))=\bigcup_{n \leq 0}f^n(W^u_\delta(p)).
\end{equation*}

Below we consider the case where $|\lambda_c(p)|\neq 1$,
 i.e., the periodic point is hyperbolic.
Then the stable and unstable sets are $C^r$-smooth manifolds;
 more precisely, they are smooth embeddings of $\pi(p)$ disjoint balls.
The dimension of $W^s(\cO(p))$ is $d_s$
 if $|\lambda_c|>1$ and $d_s+1$ if $|\lambda_c|<1$.
It is tangent at $p$ to $E^s_p$
 in the former case
 and to $E^s_p\oplus E^c_p$ in the latter one
 (here $T_p M=E^c_p \oplus E^s_p \oplus E^u_p$ is
 the partially hyperbolic splitting at $p$).
The dimension of $W^u(\cO(p))$ is $d_u$ if $|\lambda_c|<1$
 (then $T_pW^u(p)=E^u(p)$)
 and $d_u+1$ if $|\lambda_c|>1$
 (then $T_p W^u(p)=E^u(p)\oplus E^c(p)$).

The next lemma says that we can extend
 the center-stable splitting over the whole of the stable manifold
 when the center direction of the periodic point is contracting.
\begin{lemma}
\label{lemma:invariant subbundles}
Let $p\in\mathrm{Per}_{\dd}(f)$
 and $T_p M=E^c_p \oplus E^s_p \oplus E^u_p$ be the partially hyperbolic
 splitting at $p$.
Then, there exist unique subbundles $E^s$ and $E^{cs}$ of $TM|_{W^s(\cO(p))}$
 such that the restrictions of $E^s$ and $E^{cs}$ to $W^s_\delta(p)$
 are continuous for some $\delta>0$,
 $E^s(p)=E^s_p$, $E^{cs}(p)=E^c_p \oplus E^s_p$,
 and they are $f$-invariant:
 $Df(E^s(q))=E^s(f(q))$ and $Df(E^{cs}(q))=E^{cs}(f(q))$
 for any $q \in W^s(\cO(p))$.
\end{lemma}
\begin{proof}
Let $\pi$ be the period of $p$ and take small $\delta>0$.
Then, using the fact that $f^{\pi}(W^s_\delta(p)) \subset W^s_\delta(p)$
 and that $E^{cs}_p$ is an attracting fixed point 
 in the Grassmanian bundle of $TM|_{W^s_\delta(p)}$ 
 in regard to the dynamics induced by $Df^{-\pi}$, 
 we have the lemma for $W^s_{\delta}(p)$
 by the standard argument of the $C^r$-section theorem
 if we choose sufficiently small $\delta>0$
 (see for instance \cite[Proposition 7.6]{Shu}).
Next, by taking backward images, we extend the bundle $E^{c} \oplus E^s$
 to the whole $TM|_{W^s(\mathcal{O}(p))}$.
\end{proof}
We call $E^s$ and $E^{cs}$ the \emph{strong stable subbundle}
 and, respectively, the \emph{center-stable subbundle} on $W^s(\cO(p))$.
By the (strong) stable manifold theorem,
 if $|\lambda_c(p)|<1$ then $W^s(\cO(p))$ is an injectively immersed
 manifold of dimension $d_s+1$ and
 $E^s(x)=T_x W^s(\cO(p))$ for any $x \in W^s(\cO(p))$,
 and if $|\lambda_c(p)|>1$ then
 $W^s(\cO(p))$ is an injectively immersed of dimension $d_s$
 and $E^{s}(x)=T_x W^s(\cO(p))$ for any $x \in W^s(\cO(p))$.
\begin{rmk}\label{rmk:invsub}
Similar to the lemma, we can define two vector bundles 
$E^{u}$ and $E^{cu}$ over $W^u(\cO(p))$ satisfying similar properties. 
We call them \emph{the strong unstable subbundle}
 and \emph{the center-unstable subbundle}, respectively.
\end{rmk}

The following proposition summarizes the well-known result
 about the existence and uniqueness
 of the strong stable foliation in $W^{s}(p)$
 (the foliation whose fibers are tangent to $E^s$).
The foliation is of codimension 1,
 so the quotient of $f^{\pi(p)}$ is linearized near its 
hyperbolic fixed point (see (\ref{linpsi})).
\begin{prop}
\label{prop:c linearization}
Let $p \in \Per_\dd(f)$.
Suppose that $f$ is $C^r$ for $r \geq 2$ and $|\lambda_c(p)|<1$.
Then, there exists a $C^r$-function $\psi^s_p$ on $W^s(\cO(p))$
 such that for all $q \in W^s(\cO(p))$
 the kernel of $(D\psi^s_p)_q$ coincides with $E^s(q)$
 (where $E^s(q)$ is the vector bundle of
 Lemma \ref{lemma:invariant subbundles}) and
\begin{equation}
\label{linpsi}
 \psi^s_p \circ f^{\pi (p)} =\lambda^c_p \cdot \psi^s_p.
\end{equation}
Moreover,
\begin{itemize}
 \item $\psi^s_p$ is uniquely determined up to a multiplication
 by a non-zero constant, and
 \item if a sequence of diffeomorphisms $\{f_i\} \subset \mathrm{Diff}^r(M)$ 
 converges to $f_\infty$ in the $C^r$-topology
 and points $p_i \in \Fix_\dd(f_i^n)$ converge to
 $p_\infty \in \Fix_\dd(f_\infty^n)$,
 then we can choose $\psi^s_{p_i}$ such that
 $\psi^s_{p_i}$ converges to $\psi^s_{p_\infty}$ in the $C^r$-topology.
\end{itemize}
\end{prop}
The proof for the existence of $\psi^s_p$ on $W^s_\delta(p)$
 for a sufficiently small $\delta$
 can be found e.g. in \cite{GST8} (see Lemma~6 there).
The function $\psi^s_p$ is uniquely defined
 (when scaled so that $(D\psi^s_p)_p|_{E^c(p)}=id$)
 in terms of a convergent power series (see formula A.6 in \cite{GST8}).
As each term depends continuously on $f$ with respect to the $C^r$-topology,
 and the convergence rate is uniform for all maps $f_i$
 which are $C^r$-close to $f$,
 we obtain the required continuity of $\psi^s_p$ with respect to $f$.
The function $\psi^s_p$ is uniquely extended to the whole of $W^s(\cO(p))$
 from $W^s_\delta(p)$ by iterating relation (\ref{linpsi}).

We will call the function $\psi^s_p$ 
 {\it a central linearization} (a $c$-linearization) on $W^s(\cO(p))$.
In the same way, we can define a $c$-linearization $\psi^u_p$
 on $W^u(\cO(p))$ if $|\lambda_c(p)|>1$.
\medskip

Let $U$ be an open subset of $M$.
We denote
\begin{equation*}
\Per_\dd(f, U) = \{ p \in U \mid \cO(p) \subset U\}.  
\end{equation*}
For $p\in \Per_\dd(f, U)$, we put
\begin{align*}
 W^s(\cO(p),U) & = W^s(\cO(p)) \cap \bigcap_{n \geq 0} f^{-n}(U),\\
 W^u(\cO(p),U) & = W^u(\cO(p)) \cap \bigcap_{n \geq 0} f^n(U).
\end{align*}
Since $\cO(p)$ is contained in $\mathrm{Int}(U)$
(where $\mathrm{Int}(U)$ denotes the topological 
interior of $U$), 
we have
\begin{equation*}
 W^s_\delta(p) \subset W^s(\cO(p),U), \quad
 W^u_\delta(p) \subset W^u(\cO(p),U)
\end{equation*}
 if $\delta>0$ is sufficiently small.
\medskip

Let  $(\cC^{cs},\cC^{cu})$ be an $f$-invariant transverse pair
 of cone fields of index $\dd$ on an open set $U_\cC$ and
 ${\omega_\cC}$ be an $f$-invariant $c$-orientation of $(\cC^{cs},\cC^{cu})$.
 Lemma~\ref{lemma:subbundle cone} and 
 Lemma~\ref{lemma:c-orientation} below state that the partially hyperbolic 
 splittings and the orientation on $U_{\cC}$ are compatible with the  
 structures introduced in Lemma~\ref{lemma:invariant subbundles} 
 and Proposition~\ref{prop:c linearization} for points inside $U_{\cC}$. Note that since $f$ preserves the $c$-orientation, we have 
$$\lambda_c(p)>0$$
for any point $p\in \Per_\dd(f, U)$.
\begin{lemma}
\label{lemma:subbundle cone} 
Let $p \in \Per_\dd(f,{U_\cC})$ and $\lambda_c(p)\neq 1$.
Then, $E^{cs}(q) \subset \cC^{cs}(q)$ and $E^s(q) \cap \cC^{cu}(q)=\{0\}$ 
 for $q \in W^s(\cO(p),U_\cC)$.
Similarly,
 $E^{cu}(q) \subset \cC^{cu}(q)$ and $E^u(q) \cap \cC^{cs}(q)=\{0\}$
 for $q \in W^u(\cO(p),U_\cC)$.
\end{lemma}
\begin{proof}
By Proposition \ref{prop:cone PH}, we have
 $E^{cs}(p) \subset \cC^{cs}(p)$
 and $E^s(p) \cap \cC^{cu}(p) =\{0\}$.
The continuity of cones and subbundles $E^{cs}$ and $E^s$ on $W^s_\delta(p)$
(for some $\delta>0$) implies that
 $E^{cs}(f^n(q)) \subset \cC^{cs}(f^n(q))$
 and $E^s(f^n(q)) \cap \cC^{cu}(f^n(q))=\{0\}$
 for some sufficiently large $n$.
By the invariance of $E^{cs}$ and $(\cC^{cs},\cC^{cu})$,
 this implies that $E^{cs}(q) \subset \cC^{cs}(q)$.
The proof for the claims about $E^s$ and $W^u(\cO(p),U_\cC)$
 is done in a similar way.
\end{proof}

Let $p$ be a point in $\Per_\dd(f,U_\cC)$ with $\lambda_c(p)<1$
 and $\psi^s_p$ be the $c$-linearization on $W^s(\cO(p))$.
By Lemma~\ref{lemma:subbundle cone},
 we have $\cC^c(p) \cap E^{cs}(p)=\cC^{cu}(p) \cap E^{cs}(p)$
 and $\cC^c(p) \cap E^s(p)=\{0\}$.
Therefore, since the kernel of $(D\psi^s_p)_p$ coincides with $E^s(p)$,
 we may choose $\psi^s_p$ so that
 $(D\psi^s_p)_p(v)$ and $\omega_\cC(v)$ have the same sign
 for all $v \in \cC^{cu}(p) \cap E^{cs}(p)$.
Such $\psi^s_p$ is said to be {\it compatible} with $\omega_\cC$.
Remark that a compatible $c$-linearization on $W^s(\cO(p))$
 is unique up to a multiplication by a positive constant. 
 
The next lemma states that a compatible $c$-linearization 
 respects the orientation $\omega_\cC$ everywhere
 on $W^s(\mathcal{O}(p), U_\cC)$.
\begin{lemma}
\label{lemma:c-orientation}
Let $p$ be a point in $\Per_\dd(f,U_\cC)$ with $0<\lambda_c(p)<1$
 and $\psi^s_p$ be a $c$-linearization on $W^s(\cO(p))$,
 compatible with $\omega_\cC$.
Then, for any $q \in W^s(\cO(p), U_\cC)$
 and $v \in \cC^{cu}(q) \cap E^{cs}(q)$,
 we have $\omega_\cC(v)>0$ if and only if $(D\psi^s_p)_q(v)>0$.
The same holds true for $\psi_p^u$ and
 all $q \in W^u(\cO(p), U_\cC)$ if $\lambda_c(p)>1$.
\end{lemma}
\begin{proof}
By the continuity of $\omega_\cC$ and $D\psi^s_p$,
 there exists $\delta>0$ such that
 $(D\psi^s_p)_{\tilde q}(v)$ and $\omega_\cC(v)$ have the same sign
 for any $\tilde{q} \in W^s_\delta(p)$
 and $v \in \cC^{cu}(\tilde{q}) \cap E^{cs}(\tilde{q})$.
Take any $q \in W^s(\cO(p))$ and $v \in \cC^{cu}(q) \cap E^{cs}(q)$.
The forward invariance of $\cC^{cu}$ and $E^{cs}$ implies that
 $Df^n(v)$ is contained in $\cC^{cu}(f^n(q)) \cap E^{cs}(f^n(q))$
 for all $n\geq 0$.
By Lemma \ref{lemma:subbundle cone},
 $Df^n(v)$ belongs to $\cC^c(f^n(q))$ for any $n \geq 0$.
By the invariance of the orientation over $U_{\cC}$,
 we have $\omega_\cC(v)$ and $\omega_\cC(Df^n(v))$ have the same sign
 for $q \in W^s(\cO(p), U_{\cC})$ and $v \in \cC^{cu}(q) \cap E^{cs}(q)$.
Then, we choose a large $n$ such that $f^n(q) \in W^s_\delta(p)$.
Since $(D\psi^s_p)(Df^n(v))=\lambda_c(p)^n(D\psi^s_p)(v)$,
 we see that $(D\psi^s_p)(Df^n(v))$ and $(D\psi^s_p)(v)$ have the same sign.
This implies that $\omega_\cC(v)$ and $(D\psi^s_p)(v)$ have the same sign.
\end{proof}

\subsection{The signature of a heteroclinic point}
\label{ss:signature}
In this subsection, we analyze local dynamics 
 along a heterodimensional heteroclinic orbit
 between two hyperbolic periodic points.
More precisely, we define the notion of transition map
 along a heteroclinic orbit.
It enables us to establish the notion of {\em signatures}
 of heteroclinic points.  

We keep using the notation of Section \ref{ss:invari}.
We also assume that $f$ has an invariant transverse pair of cone fields
 $(\cC^{cs},\cC^{cu})$ of index $\dd$ on an open set $U_\cC$
 and ${\omega_\cC}$ is an $f$-invariant $c$-orientation
 of $(\cC^{cs},\cC^{cu})$.
We call a pair of periodic points $(p_1,p_2)$
 a {\it heteroclinic pair in $U_\cC$}
 if $p_1$ and $p_2$ are in $\Per_\dd(f,U_\cC)$,
 $\lambda_c(p_1)>1>\lambda_c(p_2) >0$,
 and $W^u(\cO(p_1),U_\cC) \cap W^s(\cO(p_2),U_\cC) \neq \emptyset$.
A point of $W^u(\cO(p_1),U_\cC) \cap W^s(\cO(p_2),U_\cC)$ 
 is called {\it a $U_\cC$-heteroclinic point} for the pair $(p_1,p_2)$.
Note that $\cO(p_1) \cup \cO(p_2) \cup \cO(q)$
 is an $f$-invariant compact subset in $U_\cC$.

\begin{lemma}
\label{lemma:heteroclinic PH} 
Let $(p_1,p_2)$ be a heteroclinic pair in $U_\cC$
 and $q$ be the corresponding ${U_\cC}$-heteroclinic point.
The $f$-invariant set $\cO(p_1) \cup \cO(p_2) \cup \cO(q)$
 admits a (continuous) partially hyperbolic splitting
 $E^c \oplus E^s \oplus E^u$
 which is compatible with the one for $\cO(p_1)$ and $\cO(p_2)$. 
\end{lemma}
\begin{proof}
First, let $\bar{E^{s}}$, $\bar{E}^{cs}$, $\bar{E}^{u}$, $\bar{E}^{cu}$
 be the vector bundles defined by Lemma~\ref{lemma:invariant subbundles}
 on $W^u(\cO(p_1),U_\cC) \cap W^s(\cO(p_2),U_\cC)$.
Put $E^c(q) = \bar{E}^{cs}(q) \cap \bar{E}^{cu}(q)$.
Lemma~\ref{lemma:subbundle cone} implies that $E^c(q)$ is one-dimensional.  
Furthermore, $E^c(q)$ is complementary to $\bar{E}^s(q)$ in $\bar{E}^{cs}(q)$
 and $\bar{E}^u(q)$ in $\bar{E}^{cu}(q)$ respectively.
Thus, we see that $T_q M=E^c(q) \oplus \bar{E}^{s}(q) \oplus \bar{E}^u(q)$.
Similar splitting exists for any point in $\cO(q)$.
Along with the index-$\dd$ partially hyperbolic splittings
 for the periodic orbits $\cO(p_1)$ and $\cO(p_2)$,
 this gives a $Df$-invariant splitting
 $TM|_\Xi=E^c \oplus E^s \oplus E^u$ of index $\dd$ 
 where $\Xi=\cO(p_1) \cup \cO(p_2) \cup \cO(q)$.
By a standard argument (see e.g. \cite{T96}),
 this splitting is continuous and partially hyperbolic.
\end{proof}
\begin{rmk}\label{rmk:hetero PH}
In the assumption of Lemma~\ref{lemma:heteroclinic PH},
 let us furthermore assume that
 there exists a point $q' \in W^s(p_1, U_{\cC}) \cap W^u(p_2, U_{\cC})$.
Then, by Remark~\ref{rmk:invsub}
 and the same argument as in Lemma~\ref{lemma:heteroclinic PH},
 we also have the existence of a partially hyperbolic splitting
 $E^c \oplus E^s \oplus E^u$ over $\cO(p_1) \cup \cO(q') \cup \cO(p_2)$
 which is compatible with the one for $\cO(p_1)\cup \cO(p_2)$.
\end{rmk}

The above lemma implies that
 the invariant set $\cO(p_1) \cup \cO(p_2) \cup \cO(q)$ is
 strongly partially hyperbolic in ${U_\cC}$.
The sets $W^u(\cO(p_1))$ and $W^s(\cO(p_2))$
 are injectively immersed submanifolds tangent to $E^{cu}$
 and $E^{cs}$ respectively.
Hence, they intersect transversely at $q$
 and the intersection near $q$ is
 a $C^r$-embedded curve $I_q$ tangent to $E^c(q)$ at $q$.

Consider $c$-linearizations $\psi^u_{p_1}$ on $W^u(\cO(p_1))$
 and $\psi^s_{p_2}$ on $W^s(\cO(p_2))$
 which are compatible with $\omega_\cC$.
Since $E^c(q) \cap E^u(q)=\{0\}$
 and the kernel of $D(\psi^u_{p_1})_q$ is  $E^u(q)$,
 the restriction of $\psi^u_{p_1}$ on $I_q$ induces
 a local $C^r$-diffeomorphism from $(I_q,q)$ to $(\RR,\psi^u_{p_1}(q))$.
In the same way, 
 the restriction of $\psi^s_{p_2}$ on $I_q$ induces
 a local $C^r$-diffeomorphism from $(I_q,q)$ to $(\RR,\psi^s_{p_2}(q))$.
Now we can define a local $C^r$-diffeomorphism 
 $\psi_q \in \Diff_\loc^r(\RR,0)$ determined by the formula
\begin{equation}
\label{transmap}
 \psi^s_{p_2}(q')-\psi^s_{p_2}(q)=\psi_q(\psi^u_{p_1}(q')-\psi^u_{p_1}(q))
\end{equation}
 for all $q' \in I_q$ close to $q$. 

We call such defined $\psi_q$ the {\em transition map}.
By Lemma \ref{lemma:c-orientation},
 both $D(\psi^u_{p_1})_q(v)$ and $D(\psi^s_{p_2})_q(v)$ have
 the same sign as $\omega_\cC(v)$ for any $v \in E^c(q)$.
This implies that the transition map $\psi_q$ 
preserves orientation.
Recall that $\psi^u_{p_1}$ and $\psi^s_{p_2}$ are uniquely defined
 up to a multiplication by positive constants. 
This guarantees that
 the signs of $A(\psi_q)$ and $S(\psi_q)$ (see (\ref{asdef}))
 are uniquely defined quantities,
since a multiplication by a positive constant does not affect the signs. 
Now, we define the \emph{signature of $q$} as the pair of signs
\begin{equation}
\label{sgnhet}
\tau_A(q;f)=\sgn(A(\psi_q)),\quad  \tau_S(q;f)=\sgn(S(\psi_q))).
\end{equation}

\begin{rmk}\label{rmk:signa}
Since the $c$-linearizations $\psi^s$, $\psi^u$ vary continuously
 with respect to the $C^r$-topology, 
 if $f$ is $C^2$ (or $C^3$)
 then $A(\psi_q)$ (resp., $S(\psi_q)$) depends continuously on $f$.
In particular, if $A(\psi_q) \neq 0$ 
(resp., $S(\psi_q) \neq 0$),
 then its sign is  locally constant in $\Diff^2(M)$
 (resp., $\Diff^3(M)$).
\end{rmk}

\subsection{Central germs and flatness at a periodic point}
\label{ss:central germs}
In this subsection, we give a definition of the ``center dynamics''
 for a periodic point $p\in \Per_{\dd}(f)$.
Using this, we introduce the notion of flatness of the periodic point 
 and define signatures
 (the signs of the non-linearity and the Schwarzian derivative) 
 for flat periodic points. 
For a manifold $M$, a point $p \in M$,
 and local $C^k$ maps $l_1,l_2:(\RR,0) \ra (M,p)$, 
We write
\begin{equation*}
 l_1(t)=l_2(t)+o(|t|^k)
\end{equation*}
 if $\vphi \circ l_1(t)=\vphi \circ l_2(t)+o(|t|^k)$
 (that is, each coordinate component of 
 $\vphi \circ l_1(t)$ and $\vphi \circ l_2(t)$ are 
 equal up to degree $k$)
 for a $C^k$ local 
 chart $(\vphi,V)$ of $M$ with $p \in V$.
Remark that if $l_1(t)=l_2(t)+(|t|^k)$,
 the same holds for any $C^k$ local chart at $p$.
\begin{dfn}
Let $k \leq r$, $f \in \mathrm{Diff}^r(M)$, 
 $p\in \mathrm{Per}_{\dd}(f)$, $\pi$ be the period of $p$.
We say that a local $C^k$-map $l:(\RR,0) \ra (M,p)$ is
 a {\it $k$-central curve} of $p$ if $(dl/dt)(0) \in E^c(p)$
 and there exists a local $C^k$-diffeomorphism $F_l:(\RR,0) \ra (\RR,0)$
 such that 
\begin{equation}
\label{phfp}
  f^{\pi} \circ l(t)= l \circ F_l(t)+o(|t|^k),
\end{equation}
 in other words, the $k$-central curve has a tangency of order $k$
 to its image by $f^{\pi}$ at the point $p$.
The local diffeomorphism $F_l$ will be called
{\it the central germ} of $f$ associated with the central curve $l$.
\end{dfn}

For two $k$-central curves $l_1$ and $l_2$ of $p$,
 we say that $l_1$ and $l_2$ have the same orientation
 if $(dl_1/dt)(0)$ and $(dl_2/dt)(0)$ are contained
 in the same connected component of $E^c \setminus \{0\}$.

By the center manifold theorem (see e.g. Theorem 5.20 of \cite{book1}), 
 for a $k$-strongly partially hyperbolic periodic point
 there exists a $C^k$-smooth invariant curve of $f^\pi$ tangent to $E^c(p)$,
 i.e., at least one $k$-central curve exists for such points.
By definition, the central curves are not unique.
However, the next lemma shows that
 for a $k$-strongly partially hyperbolic periodic point
 the $k$-central curves and the associated central germs
 are uniquely defined up to order $k$
up to a conjugacy by a local diffeomorphism.

\begin{lemma}
\label{lemma:central curve}
Let $f \in \mathrm{Diff}^r(M)$ and $k \leq r$.
Let $p \in \Per_\dd^k(f)$ and
 let $l_1,l_2:(\RR,0) \ra (M,p)$ be $k$-central curves of $p$,
 with the same orientation,
 and $F_1,F_2$ be the central germs associated with them.
Then, there exists a polynomial (hence $C^{\infty}$) orientation-preserving 
 local diffeomorphism $H:(\RR,0) \ra (\RR,0)$ such that 
 $l_1(t) = l_2\circ H (t) + o(|t|^k)$ and
 $F_1(t)=H^{-1} \circ F_2 \circ H(t)+o(t^k)$. 
\end{lemma}
\begin{proof}
Let $\pi$ be the period of $p$. 
We take a $C^k$-coordinate chart $\vphi$ such that
\begin{equation}
\label{vph}
\vphi\circ l_1(t)=(t,0,0)+o(|t|^k)
\end{equation}
 and the differential $D\vphi$ sends
 the splitting $T_p M = E^c \oplus E^s \oplus E^u$ 
 to $T_0 \RR^{\dd}=\RR \oplus \RR^{d_s} \oplus \RR^{d_u}$.
Denote $f_\vphi=\vphi \circ f^\pi \circ \vphi^{-1}$.
By the $k$-strong partial hyperbolicity of $p$, we have 
\begin{equation}
\label{fpxyz}
 f_\vphi(x,y,z)
 =(\lambda_c x, Ay,Bz)+o(\|(x,y,z)\|), \quad  (x,y,z) \in \RR^{|\dd|}, 
\end{equation}
 where $\lambda^c\neq 0$ is the central multiplier 
 (see (\ref{lamc}))
 and $A$ and $B$ are square matrices such that
 the absolute values of the eigenvalues of $A$ are
 strictly smaller than $\min\{1, |\lambda_c|^k\}$
 and the absolute values of the eigenvalues of $B$
 are strictly larger than $\max\{1, |\lambda_c|^k\}$.

Since $l_1$ and $l_2$ are $k$-central curves with the same orientation,
 there exists an orientation-preserving local $C^k$-diffeomorphism
 $\tilde{H}:(\RR,0) \ra (\RR,0)$ such that
 the $x$-coordinate of $\vphi \circ l_2 \circ \tilde{H}(t)$ is $t$.
Let $H$ be the Taylor polynomial of $\tilde{H}$ up to degree $k$.
The Taylor expansion of $\vphi \circ (l_2 \circ H)$ up to order $k$
 has the form
\begin{equation}
\label{phl2}
 \vphi \circ (l_2 \circ H)(t)=(t,P_-(t),P_+(t))+o(|t|^k)
\end{equation}
 and the curve $l_2 \circ H$ is
 a $k$-central curve with the central germ $H^{-1} \circ F_2 \circ H$.
We claim that $P_-(t)=0$ and $P_+(t)=0$.
Once it is shown,
 we will immediately obtain that $l_1(t)=l_2 \circ H(t)+o(|t|^k)$
 and $F_1(t)=H^{-1} \circ F_2 \circ H(t)+o(|t|^k)$,
 i.e., this will prove the lemma.

Denote
\begin{equation*}
 P_-(t)=t^2v_2+\dots +t^k v_k, 
 \quad P_+(t)=t^2w_2+\dots +t^k w_k,
\end{equation*} 
 where the coefficients $v_2,\dots,v_k$ lie in $\RR^{d_s}$
 and the coefficients $w_2,\dots,w_k$ lie in $\RR^{d_u}$.
Given $j=2,\dots,k$, suppose that $v_i=0$ and $w_i=0$ for $i<j$.

By (\ref{vph}) and (\ref{phfp}), we have
\begin{align*}
f_\vphi(t,0,0)
 & = f_\vphi \circ \vphi \circ l_1(t) + o(|t|^k)
   = \vphi \circ f^\pi  \circ l_1(t) + o(|t|^k) \\
 & = \vphi \circ l_1(F_1(t))+ o(|t|^k) =(F_1(t),0,0)+o(|t|^{k}).
\end{align*}
Thus, by (\ref{fpxyz}),
\begin{equation*}
 f_\vphi \circ \vphi \circ (l_2 \circ H)(t)
 = f_\vphi(t,t^jv_j,t^j w_j)+o(|t|^j)
 = (F_1(t), t^j A v_j, t^j B w_j)+o(|t|^j).
\end{equation*}
Since
 $(H^{-1} \circ F_2 \circ H)(t)
 = F_2'(0)\cdot t+o(t)=\lambda_ct+o(|t|)$,
 we also have, by (\ref{phfp}) and (\ref{phl2}), the 
 following:
\begin{align*}
 f_\vphi \circ \vphi \circ (l_2 \circ H)(t)
 & = \vphi \circ (l_2 \circ H) \circ (H^{-1} \circ F_2 \circ H)(t)+o(|t|^k)\\
 & = (H^{-1} \circ F_2 \circ H(t),
 (\lambda_c t)^j v_j, (\lambda_c t)^j w_j)+o(|t|^j).
\end{align*}
Therefore, $A v_j= (\lambda_c)^j v_j$
 and $B w_j= (\lambda_c)^j w_j$.
By the $k$-strong partial hyperbolicity of $p$,
 no eigenvalues of $A$ and $B$ can be equal to $(\lambda_c)^j$
 as long as $j\leq k$.
This implies that $v_j=w_j=0$.
Thus, by induction, we obtain that
 $v_j=w_j=0$ for all $j=2,\dots,k$,
 i.e., we have proved that $P_-(t)=0$ and $P_+(t)=0$.
\end{proof}

Below we assume that the periodic point $p$ is non-hyperbolic
 and the central multiplier $\lambda_c$ equals to $1$.
We will call such a point $1$-{\em flat}.
The $1$-flat periodic points are $k$-strong partially hyperbolic.
Thus we can Lemma~\ref{lemma:central curve} apply to such points.
This implies that the following notion of ``$k$-flatness'' is well-defined.
\begin{dfn}
\label{rflatdef}
Let $f \in \mathrm{Diff}^r(M)$, $p\in \Per_{\dd}(f)$, $k \leq r$,
 and $F_c\in \Diff^k(\RR, 0)$ be the central germ associated
 with a $k$-central curve.
We say that $p$ is \emph{$k$-flat}
 if $F_c$ is $k$-flat, i.e., $F_c(t)=t+o(|t|^k)$.  
\end{dfn} 

Lemma~\ref{lemma:central curve} also enables us
 to define the notion of the signatures of flat periodic points.
 Let us consider a diffeomorphism $f$ which admits
 an $f$-invariant transverse pair of cone fields in an open set $U_{\cC}$
 with an $f$-invariant orientation $\omega_{\cC}$. 
Let $p \in \Per_\dd(f, U_{\cC})$ and assume $p$ is $1$-flat.
The $1$-flatness implies that $p \in \Per_\dd^3(f,U_\cC)$.
We say that a $3$-central curve $l$ of $p$ is \emph{adapted}
 if $\omega_\cC((dl/dt)(0))>0$.
Choose any such curve.
By Lemma \ref{lemma:A,S conjugacy} and Lemma~\ref{lemma:central curve},
\begin{align}
\label{pers}
\tau^{\Per}_A(p;f) & = \sgn(A(F)),
 &\tau^{\Per}_S(p;f) & = \sgn(S(F))
\end{align}
 are well-defined,
 where $F$ is the $3$-central germ of $l$ at $p$.
As mentioned in Remark \ref{rmk:AS}, 
 for periodic points which are not 1-flat, 
 the signs of the non-linearity and the Schwarzian derivative
 are not well-defined in general.
\begin{rmk}
In the definition (\ref{pers}) of the signature of the flat periodic point,
 the fact that the diffeomorphism preserves
 the central orientation plays an important role.
In the orientation-reversing setting,
 only $\tau^{\Per}_S$ will remain a well-defined quantity. 
\end{rmk}

\subsection{Blenders and a connecting lemma}
\label{ss:blender}

In this section, we introduce the notion of \emph{blenders} 
 and discuss some elementary properties of them.  
There are two kinds of existing definitions of blenders. 
The first one is to give a precise description of the dynamics
 which generate the ``dimension mixing'' property of blenders
 (for instance see \cite[p. 365]{BDblender} or\cite[section 3.2]{BDbhorse}).
The other way is to state the dimension mixing property in an abstract way
 (for instance see \cite[Definition 3.1]{BDbhorse})
 and define blenders axiomatically.
Our definition follows the second one with a minor modification, 
 which is an abstraction of the local maximality of the dynamics. 

For $k \geq 1$, let $D^k$ be a closed unit disk of dimension $k$
 in the Euclidean space
 and $\cD^k(M)$ be the set of $C^1$ embeddings of disks
 of dimension $k$ in $M$.
We equip $\cD^k(M)$ with the $C^1$ topology.
We will not distinguish between the embedding and its image. 
Let $f \in \mathrm{Diff}^k(M)$ and $U_\bl$ be an open subset of $M$.
We say that $\sigma^u \in \cD^{d_u}(M)$
 is \emph{tangled with $\sigma^s \in \cD^{d_s}(M)$ in $U_\bl$}
 if there exists $N \geq 0$ and $x \in \sigma^u$
 such that $f^n(x) \in {U_\bl}$ for every $n=0, 1, \dots,N$
 and $f^N(x) \in \sigma^s$.

\begin{dfn}
A blender of index $\dd$ in ${U_\bl}$ for the map $f$
 is a pair $(\cD^{s},\cD^{u})$ of open subsets in
 $\cD^{d_s}(U_{\bl})$ and $\cD^{d_u}(U_{\bl})$ respectively
 satisfying the following: for any $\sigma^s \in \cD^s$
 the set of disks $\sigma^u$ in $\cD^u$
 which are tangled with $\sigma^s$ in ${U_\bl}$ is dense in $\cD^{u}$.
We refer to $(\cD^{s},\cD^{u})$ as the \emph{disk system} of the blender.
\end{dfn}

The difference between this definition and existing ones is that 
 for the tangled disks
 we require that a certain part of the orbit of the intersection point
 stays in $U_{\bl}$.
This information is useful when we make a perturbation
  to obtain heteroclinic connections,
 as we will see in Lemma~\ref{lemma:connecting}.
One can check that most examples of blenders, for instance  
 the one in \cite[Section 6.2]{BDV}, satisfies above condition
 for an adequate choice of $(\cD^{s},\cD^{u})$.
We briefly discuss the construction in Section~\ref{sec:examples}.
A blender in ${U_\bl}$ with the disk system $(\cD^s,\cD^u)$
 is \emph{$C^r$-persistent}
 if there exists a $C^r$-neighborhood of $f$ such that
 $(\cD^s,\cD^u)$ is the disk system of a blender in ${U_{\bl}}$
 for any diffeomorphism in the neighborhood.

Let $U$ be an open subset of $M$.
For $p \in \Per_\dd(f, U)$,
 let $W^{ss}(p)$ and $W^{uu}(p)$ be
 the strong stable and unstable manifolds
 associated with the partially hyperbolic splitting
 $T_p M=E^c(p) \oplus E^s(p) \oplus E^u$ of index $\dd$,
 i.e., $W^{ss}(p)$ is the unique injectively immersed
 $f^\pi$-invariant submanifold of dimension $d_s$
 which contains $p$ and is tangent to $E^s(p)$ at $p$,
 where $\pi$ is th period of $\pi$,
 while $W^{uu}(p)$ is the unique injectively immersed
 $f^\pi$-invariant submanifold of dimension $d_u$
 which contains $p$ and is tangent to $E^u(p)$ at $p$.

We put
\begin{align*}
W^{ss}(\cO(p),U) & =\bigcup_{j=0}^{\pi(p)-1}W^{ss}(f^j(p))
 \cap \bigcap_{n \geq 0} f^{-n}(U),\\
W^{uu}(\cO(p),U) & =\bigcup_{j=0}^{\pi(p)-1}W^{uu}(f^j(p))
 \cap \bigcap_{n \geq 0} f^n(U).
\end{align*}
If $g$ is a diffeomorphism close to $f$ and $p\in \Per_\dd(g)$,
 we will write $W^{ss}(\cO(p),U; g)$ and $W^{uu}(\cO(p),U; g)$
 for the strong stable and, resp., strong unstable manifolds of $p$
 for the map $g$.

The following ``connecting lemma'' will be used
 throughout the proof of the main theorem.
For a diffeomorphism $h \in \mathrm{Diff}^1(M)$, 
we call the set $\mathrm{supp}(h) = \cl{\{ x \mid h(x) \neq x \}}$ 
 the \emph{support} of $h$.  
\begin{lemma}
\label{lemma:connecting}
Let $U$ and $U_{bl}$ be open subsets of $M$.
Assume that $f$ admits a blender of index $\dd$ in $U_\bl$
 with disk system $(\cD^u,\cD^s)$ such that $\cl{U_{\bl}}\subset U$.
Let $p_1, p_2$ be periodic points in $\Per_\dd(f, U) \setminus \cl{U_\bl}$.
Suppose that $W^{uu}(\cO(p_1), U)$ contains a disk in $\cD^u$
 and $W^{ss}(\cO(p_2),U)$ contains a disk in $\cD^s$.
Then, for any given neighborhood $V$ of $p_1$,
 there exists $x_{\ast} \in V \cap (W^{uu}(p_1,U)) \setminus \{p_1\})$
 such that for any neighborhood $V_*$ of $x_{\ast}$
 and any neighborhood $\cU$ of the identity map in $\Diff^\infty(M)$,
 there exists $h \in \cU$ such that
 the support of $h$ is contained in $V \cap V_*$
 and $W^{uu}(p_1, U; h\circ f) \cap W^{ss}(p_2, U; h\circ f) \neq \emptyset$.
\end{lemma}
\begin{proof}
If $W^{uu}(p_1, U; f)$ intersects with $W^{ss}(p_2, U; f)$,
 then the lemma holds trivially
 for the intersection point letting $h$ the identity map.

Suppose that $W^{uu}(p_1, U; f)$ does not intersects with $W^{ss}(p_2, U; f)$.
Take disks $\sigma^s \in \cD^s$
  and $\sigma^u \in \cD^u$
 such that $\sigma^s \subset W^{ss}(\cO(p_2),U)$
 and $\sigma^s \subset W^{ss}(\cO(p_1),U)$.
By definition of the blender,
 there exists a sequence $(\sigma^u_m)_{m \geq 1}$ in $\cD^u$,
 which converges to $\sigma^u$ as $m \ra \infty$,
 such that each $\sigma^u_m$ is tangled with $\sigma^s$.
Thus, for each $m \geq 1$,
 we can find $N_m \geq 1$ and $x_m \in \sigma^u_m \cap f^{-N_m}(\sigma^s)$
 such that $f^n(x_m) \in {U_\bl}$ for all $n=0,\dots,N_m$.
Since $\sigma^u_m$ converges to a compact disk $\sigma^u$,
 we may assume that $x_m$ converges to $x_\infty \in \sigma^u$
 as $m \to \infty$.
 Since $\sigma^u\subset W^{uu}(\cO(p_1), U)$
 and $\sigma^s\subset W^{ss}(\cO(p_2), U)$,
 we have that $f^{-k}(\sigma^u)$ gets
 into an arbitrarily small neighborhood of $\cO(p_1)$
 and $f^k(\sigma^s)$ gets
 into an arbitrarily small neighborhood of $\cO(p_2)$ as $k\to \infty$.

Fix a neighborhood $V$ of $p_1$.
As $p_1\not\in\cl{U_{\bl}}$,
 there exists an integer $K>0$ such that
 $f^{-K}(\sigma^u) \subset V \setminus \cl{U_{\bl}}$. 
Put $x_{\ast} = f^{-K}(x_\infty) \in V$,
\begin{align*}
\Lambda_1 & = \cO(p_1) \cup \left\{f^n(x_{\ast}) \mid n \leq K \right\},&
\Lambda_2 & =
 \cO(p_2) \cup  \left( \bigcup_{n \geq 0} f^{n}(\sigma^s) \right),
\end{align*}
 and $\Lambda=\Lambda_1 \cup \Lambda_2$.
Remark that $\Lambda_1$ is a compact subset of $W^{uu}(\cO(p_1),U)$
 and $\Lambda_2$ is a compact subset of  $W^{ss}(\cO(p_2),U)$.
In particular, $\Lambda$ is a compact set.
Since $W^{uu}(\cO(p_1),U)$ does not intersect with $W^{ss}(\cO(p_2),U)$ 
 as we assumed at the beginning of the proof,
 the point $x_*$ is isolated in $\Lambda$.
Let a neighborhood $V$ of $x_{\ast}$ be given. 
Since $\Lambda$ is compact and $x_*$ is its isolated point,
 one can take a neighborhood $V_*$ of $x_{\ast}$
 which is contained in $(U\cap V) \setminus \cl{U_{bl}}$
 and satisfies $\Lambda \cap V_*=\{x_*\}$.
For all sufficiently large $m$, we have
 $f^{-k}(x_m) \in U$ for $k=0,\dots, K$,
 the point $f^{-K}(x_m)$ is sufficiently close to $x_\ast$, 
 and $V_* \cap \{f^{-k}(x_m)\mid 0 \leq k \leq K\}
 =\{f^{-K}(x_m)\}$.
Thus, given a $C^{\infty}$
neighborhood $\cU$ of the identity map, 
we can take $h_m \in \cU$
 such that the support
  of $h_m$ is contained in $V_* \subset V$ and 
 $h_m(x_\ast)=f^{-K}(x_m)$.
Now, we have
\begin{equation*}
(h_m  \circ f)^k(f^{-1}(x_\ast))=
\begin{cases}
f^{k-1}(x_\ast) & (k \leq 0), \\
f^{(k-1)-K}(x_m) & (k \geq 1).
\end{cases}
\end{equation*}
It is easy to see that
$f^{-1}(x_\ast)$ is contained in
 $W^{uu}(\cO(p_1), U; h_m \circ f)$ and
 $(h_m \circ f)^{N_m+1+K}(f^{-1}(x_\ast))=f^{N_m}(x_m)$
 is contained in $W^{ss}(\cO(p_2), U; h_m \circ f)$.
In particular,
 $f^{-1}(x_\ast)$ is a point of 
 $W^{uu}(\cO(p_1), U;h_m \circ f) \cap W^{ss}(\cO(p_2), U;h_m \circ f)$.
\end{proof}

Let $U_\bl$ be an open subset of $M$
 whose closure $\cl{U_{\bl}}$ is contained in $U$
 and suppose that $f$ admits a blender of index $\dd$ in $U_{\bl}$
 with disk system $(\cD^u,\cD^s)$.
We say that $p\in \Per_\dd(f,U)$ is \emph{linked to the blender in $U$} 
if $W^{ss}(\cO(p),U)$ contains a disk in $\cD^s$ and
$W^{uu}(\cO(p),U)$ contains a disk in $\cD^u$.
We will use the following lemma for establishing
 that periodic orbits we create by perturbing heteroclinic chains
 are linked to a blender. 
\begin{lemma}
\label{lemma:inherit tangle}
Let $f$ be a $C^r$ diffeomorphism of $M$,
 $U_{\cC}$ an open subset of $M$
 with an $f$-invariant transverse pair of cone fields  of index $\dd$ and
 $p_1,\dots,p_k$ ($k\geq 1$) be periodic points in $\Per_\dd(f,U_{\cC})$.
Suppose that each pair $(p_i,p_{i+1})$ admits a heteroclinic point
 $q_i \in W^{u}(p_i, U_{\cC}) \cap W^{s}(p_{i+1}, U_{\cC})$ 
 for $1 \leq i \leq k$, where we put $p_{k+1}=p_1$.
Then, for any $j_1, j_2 =1,\dots, k$,
 any $d_u$-dimensional $C^1$-disk $\sigma_u$ in $W^{uu}(p_{j_1},U_{\cC})$
 and $d_s$-dimensional $C^1$-disk $\sigma_s$ in $W^{ss}(p_{j_2},U_{\cC})$,
 and any neighborhoods $\cD^u$ of $\sigma_u$ in $\cD^{d_u}(M)$
 and $\cD^s$ of $\sigma_s$ in $\cD^{d_s}(M)$,
 there exist a neighborhood $V$ of $\bigcup_{i=1}^k(\cO(p_i) \cup \cO(q_i))$,
 a neighborhood $U_{p_{j_1}}$ of $p_{j_1}$ and $U_{p_{j_2}}$ of $p_{j_2}$,
 and a $C^1$-neighborhood $\cU$ of $f$ which satisfy the following;
 if $\bar{f} \in \cU$ and
 $\bar{p} \in \Per_\dd(\bar{f},U)$
 satisfies that
 $\cO(\bar{p},\bar{f}) \cap U_{p_{j_1}} \neq \emptyset$,
 $\cO(\bar{p},\bar{f}) \cap U_{p_{j_2}} \neq \emptyset$,
 and $\cO(\bar{p},\bar{f}) \subset V$,
 then the strong unstable manifold $W^{uu}(\cO(\bar{p}),U ; \bar{f})$
 contains a disk in $\cD^u$
 and the strong stable manifold $W^{ss}(\cO(\bar{p}),U ; \bar{f})$
 contains a disk in $\cD^s$.
\end{lemma}
\begin{proof}
Put $\Lambda=\bigcup_{i=1}^k(\cO(p_i) \cup \cO(q_i))$.
It is a strongly partially hyperbolic set of index $\dd$, according to
Remark~\ref{rmk:hetero PH}.
By the persistence of strongly partial hyperbolic sets,
 there exists a neighborhood $V \subset U$ of $\Lambda$
 and a $C^1$-neighborhood $\cV$ of $f$ such that
 for any $\bar{f} \in \cU$,
 the set $\Lambda_{\bar{f}}=\bigcap_{n \in \ZZ}\bar{f}(\cl{V})$
 admits a strongly partially hyperbolic splitting of index $\dd$.
We fix small $\delta >0$ and
take large $N \geq 1$ such that
 $f^{-N}(\sigma_u) \subset W^{uu}_\delta(p_{j_1};f)$
 and $f^{N}(\sigma_s) \subset W^{ss}_\delta(p_{j_2};f)$.
The local strong unstable and strong stable manifolds
 $W^{uu}_\delta(p',\bar{f})$ and $W^{ss}_\delta(p',\bar{f})$
 depend continuously on $\bar{f} \in \cU$ and $p' \in \Lambda_{\bar{f}}$
 as $C^1$-embedded disks, for any sufficiently small $\delta>0$.
Hence, if $\bar{f}$ is sufficiently close to $f$
 and the orbit of a periodic point
 $\bar{p} \in \Per_\dd(\bar{f}) \cap \Lambda_{\bar{f}}$ contains
 a point $p'$ sufficiently close to $p_{j_1}$
 and a point $p''$ sufficiently close to $p_{j_2}$,
 then $W^{uu}(p',U_{\cC};\bar{f})$ contains  a disk
 $\bar{\sigma_u} \in \cD^{d_u}(U_{\cC})$ which is close to $f^{-N}(\sigma_u)$
 and $W^{ss}(p'',U_{\cC};\bar{f})$ contains  a disk
 $\bar{\sigma_s} \in \cD^{d_s}(U_{\cC})$ which is close to $f^{N}(\sigma_s)$.
Then, $f^N(\bar{\sigma_u})$ is the disk
 in $W^{uu}(\cO(\bar{p}),U_{\cC};\bar{f})$ close to $\sigma_u$
 while $f^{-N}(\bar{\sigma_s})$ is the disk
 in $W^{ss}(\cO(\bar{p}),U_{\cC};\bar{f})$ close to $\sigma_s$.
\end{proof}

%
%

\section{Main theorem and outline of the proof}
\label{sec:outline}

In this section, we give the precise statement of our result. Then 
we will give two key propositions employed for the proof of the main theorem
and explain how we complete the proof with them.

\medskip

Fix an index $\dd=(d_c,d_s,d_u)$ with $d_c=1$.
Let $M$ be a compact smooth manifold of dimension $|\dd| = d_c +d_s+d_u$.
We define the subset $\cW^r$ of $\Diff^r(M)$ for each $1 \leq r \leq \infty$.
Let $\cW^1$ be the set of $C^1$ diffeomorphisms of $M$
which satisfy the following three conditions.

\paragraph{Cone condition.} 
There exists an open set $U_\cC \subset M$, 
 an $f$-invariant transverse pair $(\cC^{cs},\cC^{cu})$
 of cone fields of index $\dd$ defined in a neighborhood of 
 $\cl{U_\cC}$,
 and an $f$-invariant $c$-orientation 
 $\omega$ of $(\cC^{cs},\cC^{cu})$.

\paragraph{Existence of a blender.}
There is an open set $U_{\bl} \subset M$
 satisfying $\cl{U_{\bl}} \subset U_{\cC}$  
on which we have a $C^1$-robust blender
 with disk system $(\mathcal{D}^s, \mathcal{D}^u)$.

\paragraph{Existence of heteroclinic pairs.}
There exists two heteroclinic pairs 
(see Section~\ref{ss:signature} for the definition of 
heteroclinic pairs)
 $(p^{\ast}_1,p^{\ast}_2)$ and $(p^{\ast}_3,p^{\ast}_4)$ in $U_\cC$
 such that
\begin{itemize}
 \item $\{p^{\ast}_1,p^{\ast}_2, p^{\ast}_3,p^{\ast}_4\} \not\in \cl{U_{\bl}}$,
 \item $W^{ss}(p^{\ast}_1, U_\cC)$ and $W^{ss}(p^{\ast}_3, U_\cC)$
 contain disks in $\cD^s$,
 \item $W^{uu}(p^{\ast}_2,U_\cC)$ and $W^{uu}(p^{\ast}_4,U_\cC)$
 contain disks in $\cD^u$.
\end{itemize}
We do not exclude the case $p_1^* = p_3^*$ or/and $p_2^* =p_4^*$.
\medskip

Let $\cW^2$ be the set of maps $f \in \cW^1 \cap \Diff^2(M)$
 which satisfy the following condition (see (\ref{sgnhet}) for the definition of the signature $(\tau_A,\tau_S)$):
\paragraph{Sign condition I.}
There exist $U_{\cC}$-heteroclinic points $q$ of the pair $(p^*_1,p^*_2)$
 and $q'$ of $(p^*_3,p^*_4)$ such that
\begin{equation*}
\tau_A(q) \cdot \tau_A(q')<0. 
\end{equation*}

Let $\cW^3$ be the set of maps $f \in \cW^2 \cap \Diff^3(M)$
 which satisfy the following condition:
\paragraph{Sign condition II.}
 $$\tau_S(q) \cdot \tau_S(q')<0.$$
\medskip

For $4 \leq r \leq \infty$, we put $\cW^r=\cW^3 \cap \Diff^r(M)$.
By Remark~\ref{rmk:signa}, $\cW^r$ is a $C^r$-open 
subset of $\Diff^r(M)$. Now we are ready to 
give the precise statement  of our Main Theorem.

\begin{thm}
\label{thm:precise}
Take any $r=1, \dots, \infty$.
For every sequence of integers $(a_n)_{n \in \mathbb{N}}$
there exists a $C^r$-residual subset $\mathcal{R}$ of $\cW^r$
such that for every $f \in \mathcal{R}^r$ we have 
\begin{equation}\label{dfrspg}
\limsup_{n \to +\infty} 
 \frac{ \#\{x \in U_\cC\mid f^n(x) =x \}}{a_n} = +\infty.    
\end{equation}
\end{thm}
As we mentioned in the introduction,
 the case $r=1$ follows from a result
 by Bonatti, D\'iaz, and Fisher \cite{BDF}.

The proof of the theorem is based on the following two propositions.
The first one (Proposition~\ref{prop:one-flat};
 the proof is given in Section \ref{sec5},
 see Propositions \ref{prop:1-flat 1} and \ref{prop:1-flat + AS})
 states that $1$-flat periodic points can be produced
 by an arbitrarily $C^r$-small perturbation of any map from $\cW^r$. 
\begin{prop}
\label{prop:one-flat}
Let $f$ be a diffeomorphism in $\cW^1 \cap \Diff^\infty(M)$.
Suppose that the heteroclinic pair $(p_1,p_2)$
 admits a $U_\cC$-heteroclinic point $q$
 such that $\tau_A(q) \neq 0$ and $\tau_S(q) \neq 0$.
Then, for any neighborhood $V$ of $\{p_1,p_2\}$ and
 any $C^\infty$ neighborhood $\cU$ of the identity map in $\Diff^\infty(M)$,
 there exists $h \in \cU$ with the support contained in $V$,
 such that $h\circ f$ has a $1$-flat periodic point
 $\bar{p} \in \Per_\dd(h \circ f, U_{\cC})$ linked to the blender in $U_{\bl}$
 and
\begin{equation*}
\tau_A^{\Per}(\bar{p}; h \circ f)=\tau_A(q;f),
 \qquad \tau_S^{\Per}(\bar{p}; h \circ f)=\tau_A(q;f)
\end{equation*}
(where the signature $(\tau_A^{\Per},\tau_S^{\Per})$
 is defined in (\ref{pers})).
\end{prop}

The second proposition (for the proof 
see Section \ref{kflatsec}) states that 
 if we have a number of $k$-flat periodic 
 points (satisfying, at $k\leq 2$, certain conditions on their signatures), 
 then by an arbitrarily small perturbation 
 we can create $(k+1)$-flat periodic points. 

\begin{prop}
\label{prop:r-flat}
Let $f \in \cW^1 \cap \Diff^\infty(M)$
 and suppose that $f$ has eight $k$-flat periodic points  
 $p_1,\ldots, p_8 \in \Per_{\dd}(f, U_{\cC}) \setminus \cl{U_{\bl}}$
 (belonging to different periodic orbits)
 and heteroclinic points 
 $q_i \in W^{uu}(p_i, U_{\cC}) \cap W^{ss}(p_{i+1}, U_{\cC})$ 
 (we put $p_9 = p_1$).
Suppose that the following conditions are satisfied:
\begin{enumerate}
\item $p_i$ is linked to the blender in $U_{\bl}$ for every $i= 1,\ldots, 8$.
\item if $k=1$, then $\tau^{\Per}_{A}(p_1) \cdot \tau^{\Per}_{A}(p_3) <0$
 and $\tau^{\Per}_{S}(p_1) \cdot \tau^{\Per}_{S}(p_3) <0$.
\item if $k=2$, then $\tau^{\Per}_{S}(p_1) \cdot \tau^{\Per}_{S}(p_3) <0$.
\end{enumerate}
Then, for any neighborhood $V$ of $\{ p_i\}_{i =1,\ldots, 8}$,
 arbitrarily close to the identity map in the $C^{\infty}$ topology
 there exists $h \in \Diff^{\infty}(M)$
 with the support contained in $V$ such that 
 $h \circ f$ has a $(k+1)$-flat periodic point
 $\bar{p} \in \Per_{\dd}(h\circ f, U_{\cC})$
 linked to the blender in $U_{\bl}$.
\end{prop}
\begin{rmk}
\label{rmk:peri}
\begin{enumerate}
\item In the Proposition above, we can take $\bar{p}$
 such that its period is as large as we want. 
\item If $k=1$, then the sign of the center Schwarzian derivative 
$\tau^{\Per}_{S}(\bar{p})$ at the $2$-flat periodic point $\bar p$
 can be made as we want it ($+$ or $-$).
\end{enumerate}
\end{rmk}

Let us now see how Theorem~\ref{thm:precise} follows from
 these two propositions.
We will use the following perturbation result.

\begin{lemma}
\label{lem:flainfi}
Let $1 \geq r < +\infty$, $f\in \Diff^r(M)$
 and let $p$ be an $r$-flat periodic point of period $\pi$.
Then for every integer $a>0$ there exists $\tilde{f}\in \Diff^r(M)$
 which is arbitrarily $C^r$-close to $f$
 and the number of hyperbolic fixed points of $\tilde f^\pi$ exceeds $a$.
\end{lemma}
\begin{proof}
For an $r$-flat periodic point $p$,
 there exists a one-dimensional $C^r$-smooth center manifold $W^c$
 which is an $f^\pi$-invariant $r$-central curve.
The restriction of $f^\pi$ to $W^c$ is given by
\begin{equation*}
 F_c(t)=t+o(|t|^r), 
\end{equation*}
 see Definition \ref{rflatdef}.
Obviously, one can add an arbitrarily $C^r$-small perturbation to $f$,
 supported in a small neighborhood of the point $f^{\pi-1}(p)$,
 such that the manifold $W^c$ will remain invariant and
\begin{equation*}
F_c(t)=t 
\end{equation*}
 for all small $t$.
Next, choose sufficiently small $\varepsilon>0$ and $\delta>0$
 and add a $C^r$-small perturbation to $f$
 such that $W^c$ remains invariant and
\begin{equation*}
F_c(t)=t + \varepsilon \prod_{j=0}^a \left(t- \frac{j \delta }{a}\right) 
\end{equation*}
 for $|t|\leq \delta$.
This map has $(a+1)$ hyperbolic fixed points $t=j \delta/a$,
 which gives $(a+1)$ different periodic orbits
 which are hyperbolic in the restriction to $W^c$.
 Because of the strong partial hyperbolicity of the original $r$-flat point,
 we also have the hyperbolicity of the newly obtained periodic orbits
 transverse to $W^c$;
 the same argument can be found in \cite{Ka}.
\end{proof}

\begin{proof}
[Proof of Theorem~\ref{thm:precise}]
Let $\Fix^h(f^n)$ denote the set of hyperbolic fixed point of $f^n$.
Given a sequence $(a_n)$, we will show that for every $N$ the following set
\begin{equation*}
\cU_N := \{ f \in \cW^r \mid 
\#\Fix^h(f^n) \geq  n \cdot a_n \;\text{for some}\; n \geq N, \}
\end{equation*}
 is an open and dense set in $\cW^r$ for every $r\geq 1$.
Obviously, every map $f$ from the set
 $\bigcap_{N=1}^{\infty} \mathcal{U}_N$ satisfies (\ref{dfrspg})
 and this set is residual, which gives the theorem.

The openness of the sets $\mathcal{U}_N$ is an obvious consequence
 of the hyperbolicity of the periodic points counted.
So, it is enough to prove the denseness of $\mathcal{U}_N$.
Let us first prove the theorem for the case of finite $r$.
By Lemma~\ref{lem:flainfi},
 it is enough to prove that for any given $f \in \cW^r$
 there exists $g$ which is arbitrarily $C^r$-close to $f$
 such that $g$ has an $r$-flat periodic point
 of the period as large as we want.

Because of the density of $C^\infty$ diffeomorphisms in $\Diff^r(M)$
 and because the set $\cW^r$ is $C^r$-open, 
 we may from the very beginning assume that $f$ is $C^\infty$.
Applying Proposition~\ref{prop:one-flat}
 to the  heteroclinic pair $(p^{\ast}_1, p^{\ast}_2)$,
 we obtain a one-flat periodic point, say, $\tilde{p}$,
 by an arbitrarily $C^\infty$-small perturbation.  
This finishes the case $r=1$.

For $r \geq 2$,
 we again apply Proposition~\ref{prop:one-flat}
 letting the support of the perturbation disjoint from $\cO(\tilde{p})$ 
 (this is possible since the support can be taken arbitrarily 
 close to $(p_1^{\ast}, p_2^{\ast})$). 
Thus, by another arbitrarily small perturbation,
 we obtain another $1$-flat periodic point. 
Notice that the signatures of these $1$-flat points are
 the same as the signature of the heteroclinic point $q$ of 
 $(p^{\ast}_1, p^{\ast}_2)$ and
 that these $1$-flat periodic points are linked 
 to the blender in $U_{\bl}$. 
We repeat this process and produce $(8^{r-1}-8^{r-2})$
 of $1$-flat periodic orbits
 (e.g. $7$ orbits if $r=2$, and $56$ orbits if $r=3$, etc.).
We call the set of these periodic points $P$. 
Then we construct another $8^{r-2}$ periodic orbits
 from the heteroclinic pair $(p^{\ast}_3, p^{\ast}_4)$.
Recall that the signatures of these 1-flat periodic points are
 the same as those of the heteroclinic point $q'$
 of $(p^{\ast}_3, p^{\ast}_4)$.
We call the set of these $1$-flat points $P'$.

Now we have $8^{r-1}$ of $1$-flat periodic points.
Choose one point from $P'$ and seven points from $P$
 (all belonging to different periodic orbits);
 this creates an octuple of $1$-flat periodic points $p_1,\ldots, p_8$,
 all linked to the blender. 
By applying Lemma \ref{lemma:connecting} to pairs of these $8$ points,
 we create, by an arbitrarily small perturbation
 supported outside of $P$ and $P'$,
 heteroclinic intersections
 $q_i \in W^{uu}(p_i, U_{\cC}) \cap W^{ss}(p_{i+1}, U_{\cC})$
 (where $p_9=p_1$).
Now, we apply Proposition~\ref{prop:r-flat} to this octuple
 (where $p_3$ refers to the only point from $P'$).
This perturbation gives us a $2$-flat periodic point linked to the blender. 
Notice that we can assume that the perturbation
 to create this periodic point does not affect other 
 $1$-flat periodic points in $P$ or $P'$. 

By Remark \ref{rmk:peri},
 the created $2$-flat point can have the period as large as we want,
 so this finishes the proof of the theorem in the case $r=2$.
For $r>2$, we continue the process
 (using the rest of the points of $P$ and $P'$)
 and produce $8^{r-2}$ of $2$-flat periodic orbit 
which is linked to the blender.
We can assume that half of them have the central Schwarzian derivative $S$
 with the sign opposite to that for the other half,
 see Remark~\ref{rmk:peri}.

Then we again apply Lemma \ref{lemma:connecting}
 and Proposition~\ref{prop:r-flat}
 to construct $8^{r-3}$ of $3$-flat periodic orbits,
 all linked to the blender. As the period can be chosen arbitrarily large,
 see Remark~\ref{rmk:peri}, we have the theorem for $r=3$.

The further induction for $r>3$ does not require any sign condition
 in order to apply Proposition~\ref{prop:r-flat}.
Thus, by repeating the process starting with $k=3$,
 we produce $8^{r-k}$ of $k$-flat periodic orbits linked to the blender
 and, at the end, we obtain one $r$-flat periodic point, as required.
In each perturbation the size of the perturbation
 can be chosen arbitrarily small.
Thus, we have proved that
 arbitrarily close to the initial diffeomorphism $f$
 there exists another diffeomorphism having an $r$-flat periodic point,
 which gives the theorem for the finite $r$ case.

We can now proceed to the case $r =\infty$. 
Recall that two $C^\infty$ diffeomorphisms are close
 in the $C^{\infty}$ topology
 if they are close in the $C^r$ topology for some $r$.
Thus, proving the density of $\cU_N$ in $\cW^{\infty}$
 is the same as proving the density of $\cU_N$ in $\cW^r$ for all finite $r$.
Therefore, the proof for the finite $r$ case completes the case $r=\infty$ too.
\end{proof}

%
%

\section{Local perturbations around periodic points}\label{s:local}

In this section, we prepare several local perturbation techniques 
for the proof of Proposition~\ref{prop:one-flat}
 and Proposition~\ref{prop:r-flat}.
The proof of these propositions is divided into two steps:
We first construct a network of periodic points,
 with the help of Lemma~\ref{lemma:connecting}, and then perform perturbations 
 near this network to obtain flat periodic points. 
We will use two techniques for the second step. 
The first one (Proposition~\ref{lemma:Takens}) enables us
 to take convenient coordinates around 1-flat periodic points. 
The second one (Proposition~\ref{lemma:r-flat 2}) describes perturbations
 we use to create periodic points.

\subsection{Matrix-valued polynomial maps}\label{ss:matrix}
We start with some general estimates on the composition of maps 
 (Lemma~\ref{lem:prod_norm} and Lemma~\ref{lemma:poly 2}). 
For a multi-index $I=(i_1,\dots,i_d) \in (\ZZ_{\geq 0})^d$
 and $t=(t_1,\dots,t_d) \in \RR^d$,
 put $|I|=i_1+\dots+i_d$ and $t^I=t_1^{i_1}\cdots t_d^{i_d}$.
For positive integers $k$ and $l$,
 we denote the set of real $(k \times l)$-matrices by $\Mat(k,l)$.
Let $\|A\|$ be the operator norm of $A \in \Mat(k,l)$
 with respect to the Euclidean norms in $\RR^l$ and $\RR^k$.
For $n \geq 1$, we identify $\RR^n$ with $\Mat(n,1)$.
Remark that the operator norm $\|v\|$ for $v \in \Mat(n,1)=\RR^n$
 coincides with the Euclidean norm in $\RR^n$.

Let $F$ be a local $C^r$ map from $(\RR^d,p)$ to $(\Mat(k,l),q)$.
Write the Taylor expansion for $F$:
\begin{equation*}
 F(p+t)=\sum_{|I|\leq r} t^I A_I+o(\|t\|^r),
\end{equation*}
where $A_I \in \Mat(k,l)$.
Then we define a semi-norm $\|F\|_{p,r}$ by
\begin{equation*}
 \|F\|_{p,r}=\sum_{|I| \leq r} \|A_I\|.
\end{equation*}
It is easy to see that
 $\|F+G\|_{p,r}\leq\|F\|_{p,r}+\|G\|_{p,r}$
 for local maps $F,G$ from $(\RR^d,p)$ to $\Mat(k,l)$.

\begin{lemma}
\label{lem:prod_norm}
For local $C^r$ maps $A:(\RR^d,p) \ra \Mat(k,l)$,
 $F:(\RR^{d},p) \ra \RR^l=\Mat(l,1)$
 and $G:(\RR^{d},p) \ra \mathbb{R}$, we have
\begin{equation*}
 \|A \cdot F\|_{p,r} \leq \|A\|_{p,r} \cdot \|F\|_{p,r}, \quad 
  \|G \cdot A \|_{p,r} \leq \|G\|_{p,r} \cdot \|A\|_{p,r},
\end{equation*}
where $A \cdot F$ denotes the product of matrices 
$A$ and $F$ while 
$G \cdot A$ denotes the scalar product. 
\end{lemma}
\begin{proof}
The proof for the scalar product case is 
similar to the first case.  So we only consider the 
first case. 
Put $A(p+t)=\sum_{|I| \leq r}t^I A_I+o(|t|^r)$ 
 and $F(p+t)=\sum_{|I| \leq r}t^I v_I+o(|t|^r)$.
Then, 
\begin{align*}
\|A \cdot F\|_{p,r}
& = \|\sum_{|I| \leq r}\sum_{|J| \leq r} t^{I+J} A_I v_J  \|_{p,r} \;
\leq \; \sum_{|I| \leq r} \sum_{|J| \leq r}\|A_I \cdot v_J\|\; \\
& \leq \sum_{|I| \leq r} \sum_{|J| \leq r}\|A_I\| \cdot \|v_J\|
=\|A\|_{p,r} \cdot \|F\|_{p,r}.
\end{align*}
Thus the proof is completed.
\end{proof}
Note that, by Lemma~\ref{lem:prod_norm}, for local $C^r$-functions $\gamma, \beta$ from $(\RR^d,p)$ to $\RR$
we have
\begin{equation}
\label{eqn:norm 1}
 \|\gamma \cdot \beta\|_{p,r} \leq \|\gamma\|_{p,r} \cdot \|\beta\|_{p,r}.
\end{equation}
For a local $C^r$-map $F=(F_1,\dots,F_l):(\RR^d,p) \ra \RR^l=\Mat(l,1)$
 and $I=(i_1,\dots,i_l)$,
 we define a local function $F^I:(\RR^d,p) \ra \RR$ by the rule
\begin{align*}
 F^I(t) & =F_1(t)^{i_1}\dots F_l(t)^{i_l}.
\end{align*}
By (\ref{eqn:norm 1}), we have
\begin{equation}
\label{eqn:norm power}
 \|F^I\|_{p,r}
 \leq \|F_1\|_{p,r}^{i_1} \dots \|F_l\|_{p,r}^{i_l}
 \leq \|F\|_{p,r}^{|I|}.
\end{equation}
\begin{lemma}
\label{lemma:poly 2} 
Let $F:(\RR^d,p_1) \ra (\RR^{d'},p_2)$, 
$G:(\RR^{d'},p_2)\ra \Mat(k,l)$ be local $C^r$-maps.
Then,
\begin{align*}
 \|G \circ F\|_{p_1,r}
 & \leq \|G\|_{p_2,r} \cdot \max\{1,\|F-p_2\|_{p_1,r}\}^r
\end{align*}
\end{lemma}
\begin{proof}
Let $F(p_1+t)=\sum_{|I| \leq r}t^I v_I+o(t^r)$, $F_0=F-p_2$ and $G(p_2+t)=\sum_{|I| \leq r}t^I A_I$.
By Lemma~\ref{lem:prod_norm} and 
(\ref{eqn:norm power}), we have
\begin{align*} 
\|G \circ F\|_{p_1,r} &= \|\sum_{|I| \leq r} F_0^I \cdot A_I\|_{p_1,r}
\leq \sum_{|I| \leq r} \|F_0^I \cdot A_I\|_{p_1,r} \\
& \leq \sum_{|I| \leq r} \|F_0^I\|_{p_1,r} \cdot \|A_I\|
 \leq \sum_{|I|\leq r} \|F_0\|_{p_1,r}^{|I|} \cdot \|A_I\|.
\end{align*}
Since $\|F_0\|_{p_1,r}^{|I|} \leq \max\{1,\|F_0\|_{p_1,r}\}^r$
for $0 \leq |I| \leq r$,
 we have
$$ \|G \circ F\|_{p_1,r}
  \leq  \sum_{|I| \leq r} \max\{1,(\|F_0\|_{p_1,r})\}^r \cdot \|A_I\|  = \|G\|_{p_2,r} \cdot \max\{1,\|F -p_2\|_{p_1,r}\}^r.
$$
Thus the proof is completed.
\end{proof}

\subsection{Takens coordinates}

We introduce the following notation.
We write
$\mathbb{R}^{\dd}$ for the 
product $\mathbb{R}^{d_c} \times \mathbb{R}^{d_s} \times \mathbb{R}^{d_u}$.
Fix an index $\dd=(d_c,d_s,d_u)$.
We call a subset $B_\dd \subset \mathbb{R}^{\dd}$ 
as a \emph{polyball of index $\dd$}
 if there exists $\alpha_c, \alpha_s, \alpha_u>0$ such that
\begin{equation*}
B_\dd=\{(x,y,z) \in \RR \times \RR^{d_s} \times \RR^{d_u} 
 \mid |x| <\alpha_c, \|y\| <\alpha_s, \|z\|<\alpha_u \}.
\end{equation*}
For a polyball $B_\dd$ of index $\dd$,
 we define its subsets
\begin{align*}
 B_{\dd}^{s}  & =\{(x,y,z) \in B_\dd \mid x=z=0\},&
 B_{\dd}^{cs} & =\{(x,y,z) \in B_\dd \mid z=0\},\\
 B_{\dd}^{u}  & =\{(x,y,z) \in B_\dd \mid x=y=0\},&
 B_{\dd}^{cu} & =\{(x,y,z) \in B_\dd \mid y=0\},\\
 B_{\dd}^{c}  & =\{(x,y,z) \in B_\dd \mid y=z=0\}, &
 B_{\dd}^{su}  & =\{(x,y,z) \in B_\dd \mid x=0\}.
\end{align*}

Let $B_\dd$ be a polyball of index $\dd$
and $\hat f:B_\dd \ra \RR^{\dd}$ be a $C^r$-diffeomorphism onto its image.
We say that $\hat f$ is in the \emph{Takens standard form} if, for all $(x,y,z) \in B_\dd$,
\begin{equation}\label{takform} 
 \hat f(x,y,z)=(F_c(x), A^s(x)y, A^u(x)z),
\end{equation}
where $F_c(0)=0$,  
 $A^s(x)$ and  $A^u(x)$ are square matrices
whose entries are $C^r$ smooth functions of $x$.

Below we restrict out attention to the case
 where $d_c=1$ and $F_c'(0)=1$,
 while the eigenvalues of $A^s(0)$ lie strictly inside the unit circle
 and the eigenvalues of $A^u(0)$ lie strictly outside the unit circle.
Thus, the origin is a non-hyperbolic ($1$-flat),
 $r$-strongly partially hyperbolic fixed point,
 the invariant curve (the center manifold) $l_c(t)=(t,0,0)$
 is an $r$-central curve near the origin
 and $F_c$ is the central germ associated with $l_c$.

For a diffeomorphism $f$
 and a $1$-flat periodic point $p \in \Per_\dd(f)$ of period $\pi$,
we call a $C^r$-coordinate chart $\vphi$ around $p$
 {\it a $C^r$ Takens coordinate}
if $\vphi$ maps $p$ to the origin
 and $\vphi \circ f^{\pi} \circ \vphi^{-1}$ is in the Takens standard form
 on some polyball.
By \cite{Ta} (see Theorem in p.134), 
 we know that Takens coordinates exist if $p$ satisfies
 non-resonance conditions.
Let us recall the definition of it.
 Let $\lambda_0,\lambda_1, \dots,\lambda_{d_s+d_u}$ be
 the eigenvalues of $(Df^\pi)_p$;
 by the 1-flatness and strong partial hyperbolicity of $p$,
 we have $\lambda_0=\lambda_c(p)=1$ and $|\lambda_i|\neq 1$ for $i\geq 1$.
We say that the {\it non-resonance conditions up to degree $K$} 
(or Sternberg $K$-condition) are satisfied
 for non-neutral eigenvalues
 if
 $$\lambda_j \neq \prod_{i\geq 1} \lambda_i^{m_i} $$
 for all $j\geq 0$ and all non-negative integers $m_1,\dots, m_{d_s+d_u}$
 such that $2 \leq m_1+\dots + m_{d_s+d_u}\leq K$.

\begin{prop}[\cite{Ta}]
\label{prop:Takens_cite}
Let $f \in \Diff^{\infty}(M)$.
Let $p\in \Per_{\dd}(f)$ be $1$-flat, i.e., $\lambda_c(p) =1$.
Then, for every $k$ there exists a $C^\infty$-neighborhood
 $\cV \subset \Diff^{\infty}(M)$ of $f$
 and an integer $K(k) >0$ such that the following holds.
Given any $g \in \cV$,
 if $p$ is a $1$-flat periodic point of $g$
 (i.e., $g^{\pi(p)}p=p$ and $\lambda_c(p, g) =1$)
 and the non-resonance conditions up to degree $K(k)$ are satisfied 
 for non-neutral eigenvalues of $(Dg^{\pi(p)})_p$,
 then $g$ admits $C^k$-Takens coordinates around $p$. 
\end{prop}
For the proof of our main theorem, it would be convenient
 if $C^\infty$-Takens coordinates are available. 
Lemma~\ref{lemma:Takens} below states that we can always have such coordinates
after adding an arbitrarily $C^\infty$-small perturbation
 supported near a flat periodic point,
 without changing the central germ up to any given finite order.

\begin{lemma}
\label{lemma:Takens} 
Let $f \in \Diff^{\infty}(M)$.
Let $p \in \Per_\dd(f)$ be $1$-flat
 and let $\pi$ be the period of $p$.
Take an $r$-central curve $l:(\RR,0) \ra (M,p)$ of $f$ at $p$, where $r \geq 1$,
 and let $F_0:(\RR,0) \ra (\RR,0)$ be a central germ associated with $l$
 (see Section~\ref{ss:central germs} for the definition).
Take any neighborhood $U$ of $p$ in $M$
 and a neighborhood $\cU$ of the identity map in $\Diff^\infty(M)$.
There exist $h \in \cU$ with the support in $U$
 and a $C^\infty$-coordinate chart $\vphi: U_p \ra \RR^\dd$ around $p$,
 where $U_{p} \subset U$,
 such that $p$ remains a periodic point of period $\pi$
 for the perturbed map $h\circ f$, and the corresponding return map
 $\hat f=\vphi \circ (h \circ f)^{\pi} \circ \vphi^{-1}$
 is in the Takens standard form with a $C^{\infty}$ central germ $F_c(x)$
which satisfies $F_c(x)=F_0(x)+o(|x|^r)$ in (\ref{takform}).
\end{lemma}
\begin{proof}
First, without loss of generality we may and do assume that $F_0(x)$ is a $C^{\infty}$ germ,
since the notion of the central germ is well-defined only up to some degree.

Take a $C^\infty$ coordinate chart $\vphi_1$ at $p$ such that
 $\vphi_1 \circ l(t)=(t,0,0)+o(|t|^r)$
 and $D(\vphi_1)_p$ sends the partially hyperbolic splitting of index $\dd$
 at $p$ to $\RR \oplus \RR^{d_s} \oplus \RR^{d_u}$.
Take $h_1 \in \cU$ with small support
 such that $\vphi_1 \circ h_1 \circ (\vphi_1)^{-1}(x,y,z)=(z,A y,B z)$
 in some small neighborhood of the origin (where $A$, $B$ are some square matrices)
 and non-neutral eigenvalues of $D(h_1 \circ f^\pi)_p$
 satisfy non-resonance conditions of all degrees.

Take any $\bar{r} \geq r$
 and a small $C^{\bar{r}}$ neighborhood $\cU'$ of $h_1$ in $\Diff^\infty(M)$
 such that $\cU' \subset \cU$.
By Proposition~\ref{prop:Takens_cite},
 there exists a $C^{\bar{r}}$ coordinate chart $\vphi_2$ around $p$
 such that the return map
 $\vphi_2 \circ (h_1 \circ f)^\pi \circ \vphi_2^{-1}$
 for the perturbed map $g=h_1 \circ f$ is
 in the Takens standard form (\ref{takform}).
By construction of $h_1$,
 the curve $l$ is an $r$-central curve for $h_1 \circ f$
 and its central germ coincides with $F_0$ up to order $r$.
The invariant center manifold $\vphi_2^{-1}(x,0,0)$ of $g$ is
 also an $r$-central curve.
By Lemma \ref{lemma:central curve}, we get
\begin{equation}
\label{vpf0}
\vphi_2 \circ (h_1 \circ f^\pi) \circ \vphi_2^{-1}(x,0,0)
 = (F_0(x),0,0)+o(|x|^r)
\end{equation}
 after a $C^\infty$ change of the $x$ coordinate in (\ref{takform}).

To prove the lemma,
 we need to modify the chart $\vphi_2$ to make it $C^\infty$
 while keeping the relation (\ref{vpf0}).
Let $\hat g$ be the Taylor polynomial
 of $\vphi_2 \circ (h_1 \circ f) \circ \vphi_2^{-1}$ of order $\bar{r}$
 at the origin.
Then, $\hat g$ is also in the Takens standard form with 
\begin{equation}
\label{htgf}
\hat g(x,0,0)=(F_0(x),0,0)+o(|x|^r).
\end{equation}
Take a $C^\infty$ coordinate chart $\vphi$ at $p$
 such that $\vphi \circ \vphi_2^{-1}(x,y,z)=(x,y,z)+o(\|(x,y,z)\|^{\bar{r}})$.
Such $\vphi$ can be taken 
as follows: take any $C^\infty$ chart around $p$, 
say $\psi$. 
 Let $\eta$ be 
 the Taylor polynomial of order $\bar r$
 for $\varphi_2\circ\psi^{-1}$. 
 Then $\vphi = \eta\circ \psi$ satisfies the condition above.
 Now, since
 $\vphi \circ (h_1 \circ f^\pi) \circ \vphi^{-1}(x,y,z)
  =\hat g(x,y,z)+ o(\|(x,y,z)\|^{\bar{r}})$
 and all maps in this formula are $C^\infty$,
 we can find a small perturbation $h$ of $h_1$ in $\cU'$ with small support 
 such that
 $\vphi \circ (h \circ f^\pi) \circ \vphi^{-1}(x,y,z) = \hat g(x,y,z)$
 in a small neighborhood of the origin.
Thus, the map $\vphi \circ (h \circ f^\pi) \circ \vphi^{-1}$
 is in the Takens standard form and, by (\ref{htgf})
\begin{equation*}
\vphi \circ (h \circ f^\pi) \circ \vphi^{-1}(x,0,0) =(F_0(x),0,0)+o(|x|^r)
\end{equation*}
 as required.
Since $h_1$ was chosen arbitrarily close to the 
identity map
 and we can choose $\bar{r}$ as large as we want
 and $\cU'$ as small as we want,
 we see that $h$ can be chosen 
 arbitrarily $C^{\infty}$-close to
 the identity map. 
\end{proof}

\subsection{Perturbations near flat periodic points}
Lemma~\ref{lemma:r-flat 2} describes the main step of the perturbation
 we will use for the creation of flat periodic points. 

\begin{lemma}
\label{lemma:r-flat 2}
Assume the following:
\begin{itemize}
\item a $C^{\infty}$ map $\hat f:B_{\dd} \ra \RR^{\dd}$ 
 which is in the Takens standard form $\hat f(x,y,z)=(F_c(x),A^s(x)y,A^u(x)z)$
 and satisfies the following pinching conditions
 (for the definition of the norm $\| \, \cdot \,\|_{0, r}$,
 see Section~\ref{ss:matrix}):
\begin{align*}
  \|A^s(\cdot)\|_{0,r} \cdot \max\{1,\|F_c^{-1}\|_{0,r}\}^r  < 1, \quad 
  \|A^u(\cdot)^{-1}\|_{0,r} \cdot \max\{1,\|F_c\|_{0,r}\}^r  < 1.
\end{align*}
\item $q_-$ and $q_+$ are points of $B_\dd^{s} \setminus \{(0,0,0)\}$
 and $B_\dd^u \setminus \{(0,0,0)\}$ respectively. 
\item  $l_-:(\RR,0) \ra (B_\dd^{cs},q_-)$
 and $l_+:(\RR,0) \ra (B_\dd^{cu}, q_+)$ are smooth non-singular curves
 of the form $l_{-}(t) = (t, \hat y(t), 0)$ and $l_{+}(t) = (t, 0, \hat z(t))$. 
\end{itemize}
Then, there exists a sequence $(h_n)_{n \geq 1}$
 of compactly supported $C^\infty$ diffeomorphisms of $B_\dd$ such that:
\begin{itemize}
\item the support of $h_n$ converges to $\{\hat f(q_-), q_+\}$
 in the Hausdorff topology,
 \item $(h_n)$ converges to the identity map in the $C^\infty$-topology,
 \item $(h_n \circ \hat f)^j(q_-)$ is contained in $B_\dd$
 for every $j=0,1,\dots,n$,
 \item for every sufficiently large $n$, 
 there exists $\delta_n >0$ such that 
 for $|t| < \delta_n$ we have
\begin{equation}
 (h_n \circ \hat f)^n(l_-(t)) =l_+(F_c^n(t))+o(|t|^r).
\end{equation}
\end{itemize}
\end{lemma}
In our applications of this Lemma,
 $\hat f$ will be the first return map of a periodic point $p$.
This point will have heteroclinic connection with other periodic points 
 $p_{-}$ and $p_{+}$ so that 
 $l_{-}(t)$ will be a part of the intersection $W^{cs}(p) \cap W^{cu}(p_{-})$
 and $l_{+}(t)$ will be a part of the intersection
 $W^{cu}(p) \cap W^{cu}(p_{+})$.\\

\begin{proof}
[Proof of Lemma \ref{lemma:r-flat 2}]
Let
 $A^{s, n}:(\RR,0) \ra \GL(\RR^{d_s})$,
 $A^{u, n}:(\RR,0) \ra \GL(\RR^{d_u})$,
 $A^{u,-n}:(\RR,0) \ra \GL(\RR^{d_u})$
 be local $C^{\infty}$ maps defined by
\begin{gather*} A^{s,n}(x)  = A^s(F_c^{n-1}(x))  \cdots A^s(x), \quad 
 A^{u,n}(x) = A^u(F_c^{n-1}(x))  \cdots A^u(x),  \\
 A^{u,-n}(x) 
 =\left(A^{u,n}(F_c^{-n}(x))\right)^{-1}
 =A^u(F_c^{-n}(x))^{-1} \cdots A^u(F_c^{-1}(x))^{-1}.
\end{gather*}
Take $0<\beta<1$ such that
 $\|A^s(x)\|_{0,r} \cdot \max\{1, \|F_c^{-1}(x)\|_{0,r}\}^r  < \beta$.
By Lemma~\ref{lem:prod_norm} and Lemma~\ref{lemma:poly 2}, we have
\begin{align*}
\| (A^{s,n+1} \cdot \hat{y}) \circ F_c^{-n-1}\|_{0,r}
 & = \|\left( A^{s} \cdot (A^{s,n} \cdot \hat{y}) \right) \circ [F_c^{-n}
   \circ F_c^{-1}]\|_{0,r}\\
 & \leq \|A^s\|_{0,r} \cdot \| (A^{s,n} \cdot \hat{y}) \circ F_c^{-n}\|_{0,r}
  \cdot \max\{1, \|F_c^{-1}\|_{0,r}\} \\
 & \leq  \beta\|(A^{s,n} \cdot \hat{y}) \circ F_c^{-n}\|_{0,r},
\end{align*}
 where $\hat y$ is the function that defines the curve $l_-$,
 so $(0,\hat y(0), 0)= q_-$.

Hence, $\| (A^{s,n} \cdot \hat{y}) \circ F_c^{-n}\|_{0,r}$
 converges to zero as $n\to\infty$.
Similarly, $\|(A^{u,-n} \cdot \hat{z}) \circ F_c^n\|_{0,r}$
 converges to zero as $n\to\infty$
 (where $\hat z$ is the function that defines the curve $l_+$,
 so $(0,0,\hat z(0))= q_+$).
Let $P_n:\RR \ra \RR^{d_s}$ and $Q_{-n}:\RR \ra \RR^{d_u}$ 
 be the Taylor polynomial of order $r$
 for $\big( (A^{s,n} \cdot \hat{y}) \circ F_c^{-n} \big)(t)$
 and $\big((A^{u,-n} \cdot \hat{z}) \circ F_c^n \big)(t)$ respectively.
Then, the polynomials $P_n(t)$ and $Q_{-n}(t)$ converge to zero.

Recall that $F_c(0)=0$, so
\begin{equation}
\label{fcn0}
 F_c^n(0)=0,
\end{equation}
 which implies $A^{s,n}(0)=A^s(0)^n$ and
 $A^{u,-n}(0)=(A^{u,n}(0))^{-1}=(A^u(0))^{-n}$.
Hence,
\begin{equation}
\label{anq}
Q_{-n}(0)=(A^u(0))^{-n}\hat z(0), \quad P_n(0)=(A^s(0))^n \hat y(0).
\end{equation}

Define $C^\infty$ diffeomorphisms $h_{n,-}$ and $h_{n,+}$ of $\RR^{\dd}$ by
\begin{align*}
h_{n,-}(x,y,z)  & = (x,y,z+ Q_{-n}(x)),&
h_{n,+}(x,y,z) & = (x,y-P_{n}(x), z).
\end{align*}
They converge to the identity map on $B_\dd$ in the $C^\infty$ topology
 as $n\to+\infty$.
Thus, we can find a sequence $(h_n)_{n\geq 1}$ of diffeomorphisms
 of $B_{\dd}$ such that:
 \begin{itemize} 
 \item $h_n=\hat f \circ h_{n,-} \circ \hat f^{-1}$
 in a neighborhood of $\hat f(q_-)$,
 \item $h_n=h_{n,+}$ in a neighborhood of $q_+$,
 \item the support of $h_n$ converges to $\{\hat f(q_-),q_+\}$, and
 \item $h_n$ converges to the identity in the $C^\infty$ topology
 as $n\to+\infty$.
 \end{itemize}
For the time being we assume the existence 
of such $(h_n)$ and proceed the proof. 
We discuss the construction of $(h_n)$ later.

For sufficiently large $n$, we have
 $(h_n \circ \hat f)^j(q_-) \in B_\dd$ for $j=0, 1,\dots, n$.
Also, since $h_n$ differs from identity only in a small neighborhood
 of the points $q_+$ and $\hat f(q_-)$,
 we have 
\begin{equation*}
\hat f \circ (h\circ \hat f)^{n-1}=\hat f^n \circ h_{n,-} 
\end{equation*}
 in a small neighborhood of the point $q_-$.
Thus, for small $t$
\begin{equation*}
\hat f \circ (h\circ \hat f)^{n-1}(t,\hat y(t),0)=( (F_c)^n(t),
\; A^{s, n}(t) \hat{y}(t), \; A^{u,n}(t) Q_{-n}(t)).
\end{equation*}
By (\ref{fcn0}) and (\ref{anq}),
 the right-hand side is close to $q_+=(0,0,\hat z(0))$ 
 for small $t$.
Near this point $h_n$ equals to $h_{n,+}$, which gives
\begin{align*}
(h\circ \hat f)^n(t,\hat y(t),0)
 & = h_{n,+}\big((F_c)^n(t),\; A^{s, n}(t) \hat{y}(t),
  \; A^{u,n}(t) Q_{-n}(t)\big)\\
 & = ( (F_c)^n(t),\; A^{s, n}(t) \hat{y}(t) -P_{n}(F_c^n(t)),
  \; A^{u,n}(t)Q_{-n}(t)).
\end{align*}
Since $A^{s, n}(t)\hat{y}(t) -P_{n}(F_c^n(t)) =o(|t|^r)$
 and $A^{u,n}(t)Q_{-n}(t) = \hat{z}(F_c^{n}(t)) + o(|t|^r)$
 by definition of $P_{n}$ and $Q_{-n}$, 
 we have
\begin{equation*}
 (h_n \circ \hat f)^n(l_-(t)) =l_+(F_c^n(t))+o(|t|^r)
\end{equation*}
 as required.

Finally, let us discuss the 
construction of $(h_n)$.
Choose a $C^\infty$-smooth bump function
 $\rho :\RR \to \RR$ such that $\rho(s)=1$ for $|s|\leq 1$ and
 $\rho(s)=0$ for all $|s|\geq 2$.
For $k \geq 1$
 and $(x,y,z) \in B_\dd$ with $\hat{f}^{-1}(x,y,z)=(x',y',z')$,
 put 
\begin{align*}
\lefteqn{\bar{h}_{n,k}(x,y,z)}\\
& =
\begin{cases}
 \left(x,y-\rho\left(\|(x,y,z)-q_+ \|\right/k) P_n(x),\, z \right)
 & \left(\|(x,y,z)-q_+ \|\leq 2/k\right),\\
 \hat f \left(x',y',\, z'+\rho\left(\|(x',y',z')-q_- \|/k \right)
 Q_{-n}(x') \right)
 & \left(\|(x',y',z')-q_- \| \leq 2/k\right),\\ 
 (x,y,z) & \text{(otherwise)}.
\end{cases}
\end{align*}
Let $\delta_{n,k}$ be the $C^k$-distance between $\bar{h}_{n,k}$
 and the identity map.
Since $Q_{-n}$ and $P_n$ converge to zero in the $C^\infty$-topology
 as $n$ goes to infinity,
 the sequence $(\delta_{n,k})_{n \geq 1}$ converges to zero
 as $n$ goes to infinity for each $k \geq 1$.
Hence, there exists a sequence $(n_k)_{k \geq 1}$
 such that $\delta_{n,k}<1/k$ for any $k \geq 1$ and any $n \geq n_k$.
We may assume that $n_{k+1} \geq n_k$ for any $k \geq 1$.
Remark that $\delta_{n,k} \leq \delta_{n,k'} <1/k'$
 if $k' \geq k$ and $n \geq n_{k'}$
 since the $C^k$-norm is not greater than the $C^{k'}$ norm if $k' \geq k$.
Define $h_n=\bar{h}_{n,1}$ if $n <n_1$
 and $h_n=\bar{h}_{n,k}$ if $n_k \leq n < n_{k+1}$ with some $k \geq 1$.
Then, the $C^k$-distance between $h_n$ and the identity map
 is smaller than $1/k'$ if $k' \geq k$ and $n \geq n_k'$,
This implies that
 $h_n$ converges to the identity map in the $C^k$-topology for any $k \geq 1$,
 and hence, in the $C^\infty$-topology.
For $n_k \leq n \leq n_{k+1}$,
 the support of $h_n$ is contained in
 the ball of radius $1/k$ centered at $q_+$
 and the $\hat{f}$-image of the ball of radius $1/k$ centered at $q_-$.
This implies that the support of $h_n$ shrinks to $\{p_+,p_-\}$.
\end{proof}

\begin{rmk}
\label{rmk:conju}
Suppose $f:B_{\dd} \to B_{\dd}$ is 
in the Takens standard form
and assume that the origin is a $1$-flat fixed point, i.e., $F_c'(0)=1$.  
Let $\psi_\alpha$ be a diffeomorphism of $R^{\dd}$ given by
 $\psi_\alpha(x,y,z)=(\alpha x,y,z)$.
Then,
\begin{equation*}
\psi_\alpha^{-1} \circ \hat f \circ \psi_\alpha(x,y,z)
 =(\alpha^{-1} F_c(\alpha x), A_s(\alpha x)y, A_u(\alpha x)z)
\end{equation*}
  is in the Takens standard form and the norms
\begin{equation*}
 \|\alpha^{-1} F_c(\alpha x) - x\|_{0,r}, \hsp
 \|A^s(\alpha x)-A^s(0)\|_{0,r}, \hsp
 \|A^u(\alpha x)-A^u(0)\|_{0,r}
\end{equation*}
 converge to zero as $\alpha$ goes to zero.
Thus, by taking a conjugacy by $\psi_\alpha$
 with any sufficiently small $\alpha>0$,
 we can obtain that the pinching condition of Lemma \ref{lemma:r-flat 2}:
\begin{align*}
 \|A^s(\cdot)\|_{0,r} \cdot \max\{1,\|F_c^{-1}\|_{0,r}\}^r & < 1,\\
 \|A^u(\cdot)^{-1}\|_{0,r} \cdot \max\{1,\|F_c\|_{0,r}\}^r & < 1.
\end{align*}
By performing this modification 
we can apply this lemma for any $1$-flat point.

When using Lemma \ref{lemma:r-flat 2} near $1$-flat points,
 we will need to control the quantity $S(F_c)/A(F_c)$.
Suppose $S(F_c), A(F_c)$ are not equal to $0$.
The coordinate change given by $\psi_\alpha$ behaves
 as a conjugacy by a linear transformation in the center direction
 (the $x$-coordinate). 
Accordingly, by Lemma~\ref{lemma:A,S conjugacy},
 we see that $|S(F_c)/A(F_c)|$ will be multiplied by $\alpha$
 after the conjugacy. 
Thus, we can also assume that 
the value $|S(F_c)/A(F_c)|$ is
 larger than any given real number.  
\end{rmk}

\subsection{Hyperbolic case}

In this section, we consider a variant of Lemma~\ref{lemma:r-flat 2}
 for the case where the fixed point is hyperbolic.
The difference with Lemma~\ref{lemma:r-flat 2} is that
 the point of the heteroclinic connection now is not necessarily
 in the strong stable manifold, that is, 
 $q_{-}$ is not necessarily in $B^{s}$ but in $B^{cs}$.
We consider only the case where the map is linear
 in a neighborhood of the fixed point.
 This is enough for our purpose,
 as the maps that admit linearizing $C^\infty$ coordinates
 near a hyperbolic fixed point are dense in $C^\infty(M)$,
 as it follows from Sternberg theorem \cite{S}. 
\begin{lemma}
\label{lemma:linearization}
Assume the following:
\begin{itemize}
\item $\hat f:B_{\dd} \ra \RR^{\dd}$ is a linear isomorphism of the form
 $\hat f(x,y,z)=(\lambda_c x,A^s y, A^u z)$ such that 
\begin{equation}
\label{spga}
 \|A^s\|<\lambda_c<1 \;\mbox{ and }\; \|(A^u)^{-1}\|<1,
\end{equation}
\item $q_-=(x_*,y_*,0)$ and $q_+=(0,0,z_*)$ are points of
 $B_\dd^{cs}\setminus \{(0, 0, 0)\}$ and
 $B_\dd^u \setminus \{(0, 0, 0)\}$ respectively,
\item  $l_-:(\RR,0) \ra (B_\dd^{cs},q_-)$
 and $l_+:(\RR,0) \ra (B_\dd^{cu}, q_+)$ are
 smooth non-singular curves of the form
 $l_{-}(t) = (x_*+ t, \hat{y}(t), 0)$
 and $l_{+}(t) = (t, 0, \hat{z}(t))$,
 where $\hat{y}(0)=y_*$ and $\hat{z}(0)=z_*$. 
\end{itemize}
Then, there exists a sequence $(h_n)_{n \geq 1}$
 of compactly supported $C^\infty$ diffeomorphisms of $B_\dd$ such that:
 \begin{itemize}
\item The support of $(h_n)$ converges to $\{\hat f(q_-), q_+\}$
 in the Hausdorff topology,
 \item $(h_n)$ converges to the identity map in the $C^\infty$ topology,
 \item $(h_n \circ \hat f)^j(q_-)$ is contained in $B_\dd$
 for any $j=0,1,\dots,n$, and 
 \item $(h_n \circ \hat f)^n \circ l_-(t) =l_+(\lambda_c^n t)+o(|t|).$
\end{itemize}
\end{lemma}
\begin{proof}
The proof is done by an argument similar to
 the proof of Lemma~\ref{lemma:r-flat 2} 
 with minor modifications concerning the position of $q_{-}$.

Let $P_n(t)= (A^s)^n(y_* + \hat{y}'(0) t)$ and
$Q_{-n}(t)=(A^u)^{-n}(z_*+\hat{z}'(0) \lambda_c^n t)$
(they are the Taylor polynomials 
of $(A^s)^n\hat{y}(t) $ and 
 $(A^u)^{-n} \hat{z}(\lambda_c^n t)$ 
 up to order $1$).
Take sequences $(h_{n,+})$ and $(h_{n,-})$
 of diffeomorphisms of $\RR^\dd$ given by
\begin{equation*}
 h_{n,-}(x,y,z)
  = (x,y,z +Q_{-n}(x-x_*)), \quad 
 h_{n,+}(x,y,z)
  = (x-\lambda_c^n x_*, y-P_{n}(\lambda_c^{-n}x-x_*),z).
\end{equation*}
By (\ref{spga}),
 $h_{n,-}$ and $h_{n,+}$ on any given ball in $\RR^\dd$
 converge to identity the $C^\infty$ 
 topology as $n\to +\infty$.
As in the proof of Lemma~\ref{lemma:r-flat 2},
 we can construct a sequence of diffeomorphisms  $(h_n)_{n \geq 1}$ such that
 the support of $h_n$ is compact and converges to $\{\hat f(q_-),q_+\}$,
 as $n\to+\infty$,
 $h_n = \hat f \circ (h_{n,-}) \circ \hat f^{-1}$
 in (an $n$-dependent) neighborhood of $\hat f(q_-)$,
 $h_n = h_{n,+}$ in a neighborhood of $q_+$,
 and $(h_n)_{n \geq 1}$ converges to the identity map
 in the $C^{\infty}$ topology.

The, for every sufficiently large $n$, we have that for small $t$
\begin{align*}
(h_n \circ \hat f)^n \circ l_-(t)
 & = h_{n,+} \circ \hat f^n \circ h_{n,-}(x_*+t,\hat{y}(t),0)\\
 & = h_{n,+} \circ (\lambda_c^n (x_*+t), (A^s)^n\hat{y}(t),
 (A^u)^n Q_{-n}(t)) \\
 & = (\lambda_c^n t , (A^s)^n\hat{y}(t)-P_n(t), (A^u)^n Q_{-n}(t)).
\end{align*}
Since $(A^s)^n\hat{y}(t)=P(t)+o(t)$ and 
 $(A^u)^{-n}Q_{-n}(t) = \hat{z}(\lambda_c^n t) +o(t)$,
 this implies that
 $(h_n \circ \hat f)^n \circ l_-(t)=l_+(\lambda_c^n t)+o(|t|)$.
\end{proof}

%
%

\section{Creating $1$-flat periodic points}\label{sec5}
In this section, we prove Proposition~\ref{prop:one-flat}.
The proof is divided into two steps.
First, by applying techniques of Section~\ref{s:local},
 we create a $1$-flat periodic point
 in the way similar to that we use for the construction
 of $r$-flat periodic points for $r \geq 2$, see Section \ref{kflatsec}.
However, in Proposition~\ref{prop:one-flat}
 we need to ensure that the $1$-flat point has signature of a given type.
The computation of the signature will be done at the second step.

\subsection{Construction of a candidate orbit}
Recall that every diffeomorphism $f\in \cW^1$ has
 a pair of hyperbolic periodic points $p_1$ and $p_2$
 (of periods $\pi_1$ and $\pi_2$, respectively)
 such that the $(d_u+1)$-dimensional unstable manifold of $p_1$ has
 a non-empty transverse intersection
 with the $(d_s+1)$-dimensional stable manifold of $p_2$;
 the heteroclinic point $q$ belongs to this intersection.

Let $\cW^1_*$ be the set of diffeomorphisms
 in $\cW^1\cap \Diff^{\infty}(M)$
such that 
\begin{itemize}
\item the return maps $f^{\pi_1}$ and $f^{\pi_2}$
near the points $p_1$ and $p_2$
 are locally linear in certain $C^\infty$ coordinates;
\item $\log\lambda_c(p_1)$ and $\log \lambda_c(p_2)$ are
 rationally linearly independent, and
\item in addition to the heteroclinic intersection
 of $W^{u}(p_1,U_{\cC})$ and $W^{s}(p_2,U_{\cC})$,
 there exists a non-transverse intersection
 of $W^{s}(p_1,U_{\cC})$ and $W^{u}(p_2,U_{\cC})$.
\end{itemize}
It is not difficult to see that
 $\cW^1_*$ is $C^r$-dense in $\cW^1 \cap \Diff^r(M)$
 for any $1 \leq r \leq \infty$.
Indeed, the first condition is fulfilled by Sternberg theorem \cite{S}
 if the eigenvalues $\lambda(p_1)$ of the linearization matrix
 $(D f^{\pi_1})_{p_1}$ and $\lambda(p_2)$ of $(D f^{\pi_2})_{p_2}$
 are non-resonant, i.e., for $p=p_1$ and $p=p_2$,
\begin{equation*}
 \lambda_j(p)\neq \prod_{i=1}^{|\dd|} \lambda_i(p)^{m_i} 
\end{equation*}
 for all $j=1,\dots, |\dd|$
 and all integer indices $m_i\geq 0$ such that $m_1+\dots + m_{\dd}\geq 2$.
One can easily achieve this, and the second condition as well,
 by an arbitrarily small perturbation of $f$
 supported in a small neighborhood of the points $p_1$ and $p_2$. 
In order to obtain the last condition
 while keeping the previous two conditions intact,
 we use Proposition~\ref{lemma:connecting}.
Notice that the perturbation in Proposition~\ref{lemma:connecting}
 is chosen such that it does not affect the local behavior
 near $\cO(p_1)$ and $\cO(p_2)$.
Thus, we can apply Proposition~\ref{lemma:connecting} without disturbing the 
conditions about the eigenvalues at $p_1$ and $p_2$.

For diffeomorphisms in $\cW_{*}^1$, we will prove the following:
\begin{prop}
\label{prop:1-flat 1}
Let $f \in \cW^1_*$ and $\{p_1,p_2\}$ be the heteroclinic pair
 with $U_{\cC}$-heteroclinic point $q$.
For any neighborhood $V$ of $\{p_1,p_2\}$
 and any $C^\infty$ neighborhood $\cU$ of the identity map in $\Diff^\infty(M)$,
 there exist $h \in \cU$ with the support of $h$ contained in $V$ such that
 the diffeomorphism $h\circ f$ has
 a $1$-flat periodic point $p_* \in \Per_\dd(h \circ f,{U_\cC})$ such that
 $p_{*} \not\in \cl{U_{\bl}}$ and
 $p_{*}$ is linked to the blender in $\cU_{\bl}$.
\end{prop}

This proposition gives the part of Proposition~\ref{prop:one-flat}.
The remaining part is given by
\begin{prop}
\label{prop:1-flat + AS}
In the hypothesis of Proposition \ref{prop:1-flat 1},
 we furthermore assume $\tau_A(q)\neq 0$
 and $\tau_S(q) \neq 0$. 
Then, in the conclusion of Proposition~\ref{prop:1-flat 1}, 
we can take $p_{\ast}$ such that
 $\tau_A^{\Per}(p_{\ast}; h \circ f)=\tau_A(q;f)$,
 and $\tau_S^{\Per}(p_{\ast}; h \circ f)=\tau_A(q;f)$.
\end{prop}

\begin{proof}
[Proof of Proposition~\ref{prop:1-flat 1}]
Take $q=q_1 \in W^u(p_1,U_\cC) \cap W^s(p_2,U_\cC)$
 and $q_2 \in W^{u}(p_2,U_\cC) \cap W^{s}(p_1,U_\cC)$;
 these heteroclinic points exist by the definition of $\cW_{*}^1$.
Let $\Lambda=\cO(p_1) \cup \cO(p_2) \cup \cO(q_1) \cup \cO(q_2)$.
By Lemma \ref{lemma:heteroclinic PH} and Remark~\ref{rmk:hetero PH},
the $f$-invariant set $\Lambda$ admits a partially hyperbolic splitting
 $E^c\oplus E^s\oplus E^u$ of index $\dd$.
Let $\pi_i$ be the period of $p_i$ for $i=1,2$.
By Lemma \ref{lemma:inherit tangle},
 there exist neighborhoods $U_\Lambda$, $U_1$, and $U_2$
 of $\Lambda$, $p_1$, and $p_2$, respectively,
 and a neighborhood $\cU_*$ of the identity map in $\Diff^r(M)$
 such that the following holds:
 \begin{itemize}
\item the open sets $U_1,\dots,f^{\pi_1}(U_1), U_2,\dots, f^{\pi_2}(U_2)$
 are mutually disjoint subsets of $U_\Lambda  \setminus \cl{U_{bl}}$,
 \item for any $h \in \cU_*$, any compact $h \circ f$-invariant subset of $U_\Lambda$
 is a strongly partially hyperbolic set of index $\dd$, and 
 \item any $p_* \in \Per(h \circ f,U_\cC) \cap U_1$
 is linked to the blender in $U_{\bl}$ for $h \circ f$.
\end{itemize}

By the linearizability assumption, for each $i=1,2$,
 if $U_i$ is small enough, then there exists a $C^\infty$ coordinate chart
 $\vphi_i:U_i \ra \RR^{\dd}$ at $p_i$ such that
 the map $\vphi_i \circ f^{\pi_i} \circ \vphi_i^{-1}$,
 defined on a polyball $B_{\dd,i}$, has the form
\begin{equation*}
\hat f_i(x,y,z)=(\lambda_i x,A^s_i y, A^u_i z)
\end{equation*}
 where $\lambda_i=\lambda_c(p_i)$ are the central multipliers
 of the points $p_i$ ($i=1,2$), $\lambda_1>1>\lambda_2>0$,
 $A^s_i \in \GL(\RR^{d_s})$, and $A^u_i \in \GL(\RR^{d_u})$.

Notice that $\lambda_i=\lambda_c(p_i;f)$.
Let $\vphi_i^x$ be the $x$-coordinate function of $\vphi_i$.
The map $\vphi_2^x$
 satisfies $\vphi_2^x \circ f^{\pi_i}=\lambda_2 \cdot \vphi_2^x$.
Thus, it is a $c$-linearization
 (see Section~\ref{ss:invari} for the definition).
By replacing $\vphi_2$ with $-\vphi_2$, if necessary,
 we may always assume that 
 $\vphi_2^x$ is compatible with 
 the $c$-orientation $\omega_\cC$.
Similarly, we may assume that
 $\vphi_1^x$ is compatible with $\omega_\cC$.
By Lemma \ref{lemma:c-orientation}, the
 compatibility implies that
 $D\vphi_i^x(v)>0$ if and only if $\omega_\cC(v)>0$,
 where $v$ is any vector from $\cC^c(z)$
and $z$ is any point in $U_i$.

Take $\delta>0$ such that
\begin{align*}
 W^s_\delta(p_1)  & \subset \vphi_1^{-1}(B_{\dd,1}^s), &
 W^u_\delta(p_1)  & \subset \vphi_1^{-1}(B_{\dd,1}^{cu}),\\
 W^s_\delta(p_2)  & \subset \vphi_2^{-1}(B_{\dd,2}^{cs}), &
 W^u_\delta(p_2)  & \subset \vphi_2^{-1}(B_{\dd,2}^{u}).
\end{align*}
Since $q_1 \in W^u(p_1) \cap W^s(p_2)$
 and $q_2 \in W^{u}(p_1) \cap W^{s}(p_1)$,
 we can pick 4 heteroclinic points
\begin{align*}
 q_1^s & \in W^s_\delta(p_2) \cap \cO(q_1), &
 q_1^u & \in W^u_\delta(p_1) \cap \cO(q_1), \\
 q_2^s & \in W^s_\delta(p_1) \cap \cO(q_2), &
 q_2^u & \in W^u_\delta(p_2) \cap \cO(q_2).
\end{align*}
Take $n_1,n_2 \geq 1$ such that
 $f^{n_1}(q_1^u)=q_1^s$, $f^{n_2}(q_2^u)=q_2^s$
 and let
 \[
Q=\{f^j(q_i^u) \mid i=1,2, \, 1 \leq j \leq n_i\}.
 \]
The intersection of $W^u(p_1)$ and $W^s(p_2)$ is transverse
 at the points of the orbit of $q_1$.
As $\dim (W^u(p_1))+ \dim(W^s(p_2))=|\dd|+1$,
 near each of these points the intersection is a regular smooth curve. 
For the heteroclinic points $q_1^u$ and $q_1^s$
 we denote these curves as  $l^u_1$ and $l^s_1$.
We parameterize these curves such that
 they can be viewed as local $C^\infty$ maps
 $l^u_1:(\RR,0) \ra (\vphi_1^{-1}(B_{\dd,1}^{cu}), q_1^u)$,
 $l^s_1:(\RR,0) \ra (\vphi_2^{-1}(B_{\dd,2}^{cs}),q_1^s)$,
 such that  $\vphi_1^x(l^u_1(t))=\vphi_2^x(l^s_1(t))=t$.
By construction, the curve $l^s_1$ is taken to $l^u_1$ by $f^{n_1}$,
 so there exists a local $C^\infty$-map
 $G:(\RR,0) \ra (\RR,0)$ such that
\begin{equation}
\label{transg1}
 f^{n_1} \circ l^u_1(t)=l^s_1(G(t)).
\end{equation}
Note that by construction,
 $G$ coincides with the transition map $\psi_q$ defined by (\ref{transmap})
 where one should put $q=q_1^s$.

Recall that the tangent $(dl^u_1/dt)(0)$ lies
 in the center subspace $E^c(q_1^u)$
 and $(dl^s_2/dt)(0)$ lies in $E^c(q_2^s)$.
In particular, they are contained in the cone field $\cC^c$.
The invariance of the $c$-orientation $\omega_{\cC}$ and
 the compatibility of $\vphi_i^x$ implies that $G'>0$.
Denote $\mu_1=G'(0)$. By (\ref{transg1}), we have
\begin{equation*}
 f^{n_1} \circ l^u_1(t)=l^s_1(\mu_1 t)+o(t).
\end{equation*}

In a similar way, we consider curves lying
 in the intersection of $W^{cs}(p_1)$ and $W^{cu}(p_2)$
 near the heteroclinic points $q_2^{s,u}$.
The manifolds $W^{cs}(p_1)$ and $W^{cu}(p_2)$ are not defined uniquely;
 we choose them such that in the linearizing coordinates
 $\vphi_{1}$ 
 the manifold $W^{cs}_\delta(p_1)$ is 
 given by the equation $z=0$
 and in $\vphi_{2}$ the 
 manifold $W^{cu}_\delta(p_2)$ is given by the equation $y=0$.
In particular, the manifolds $W^{cs}(p_1)$ and $W^{cu}(p_2)$
 are of $C^\infty$ class.
So, $W^{cu}_\delta(p_2)\cap f^{-n_2} (W^{cs}_\delta(p_1))$ contains
 a $C^\infty$ curve $l^u_2$ through $q_2^u$ tangent to $E^c(q_2^u)$
 and $ f^{n_2} (W^{cu}_\delta(p_2))\cap W^{cs}_\delta(p_1)$ contains
 a $C^\infty$ curve $l^s_2$ through $q_2^s$ tangent to $E^c(q_2^s)$.
In other words, we have local $C^\infty$ maps
 $l^u_2:(\RR,0) \ra (B_{\dd,2}^{cu},q_2^u)$,
 $l^s_2:(\RR,0) \ra (B_{\dd,1}^{cs},q_1^s)$
such that
 $\vphi_2^x \circ l^u_2(t)=\vphi_1^x \circ l^s_2(t)=t$ and
\begin{equation*}
 f^{n_2} \circ l^u_2(t)=l^s_2(\mu_2 t)+o(|t|)
\end{equation*}
for some $\mu_2>0$

Now, we apply Lemma \ref{lemma:linearization}
 for $\hat f=(\hat f_1)^{-1}$, $l_-=l^u_1$, and $l_+=l^s_2$.
We also apply Lemma \ref{lemma:linearization}
 for $\hat f=\hat f_2$, $l_-=l^s_1$, and $l_+=l^u_2$.
We thus find that   there exist $N \geq 1$ and
 a family $(h_{m_1,m_2})_{m_{1}, m_{2}\geq N}$ of diffeomorphisms in $\cU_1$
 such that
\begin{gather*}
 \supp(h_{m_1,m_2})
   \subset (U_1 \cup U_2) \setminus Q,\\
 \{(h_{m_1,m_2} \circ f)^{j\pi_i}(q^s_i) \mid 0 \leq j \leq m_i\}
   \subset U_i,\\
 (h_{m_1,m_2} \circ f)^{m_i\pi_i} \circ l^s_i(t)
   =l^u_i(\lambda_i^{m_i} t)+o(t) 
\end{gather*}
 for $i=1,2$.
Then, for all $m_{1}, m_{2} \geq N$,
\begin{equation*}
 (h_{m_1,m_2} \circ f)^{m_1\pi_1+m_2\pi_2+n_1+n_2} \circ l^u_1(t)
 =l^u_1(\lambda_1^{m_{1}}\mu_2 \cdot \lambda_2^{m_2} \mu_1 t)+o(|t|). 
\end{equation*}
Since $\lambda_1 > 1 > \lambda_2>0$ and
 $\log\lambda_1$ and $\log\lambda_2$ are rationally linearly independent,
 there exist a sequence $((m_{1,j},m_{2,j}))_{j \geq 1}$
 of pairs of positive integers
 such that $\lambda_1^{m_{1,j}}\mu_2\lambda_2^{m_{2,j}}\mu_1$
 converges to $1$ as $j\to \infty$.
 
Thus, $q_1^u=l^u_1(0)$ becomes a periodic point
 in $\Per_\dd(h_{m_{1,j}, m_{2,j}} \circ f,{U_\cC}) $
 such that $\cO(q_1^u; h_{m_{1,j},m_{2,j}} \circ f) \subset U_{\Lambda}$,
 and $(h_{m_{1,j},m_{2,j}})^{n_1}(q_1^u)=q_1^s \in U_2$.
Moreover, its central multiplier
 $\lambda^*=\lambda_c(q_1^u;h_{m_{1,j},m_{2,j}} \circ f)$ gets
 as close as we want to $1$ as $j\to \infty$.
It remains to note that for any sufficiently large $j$,
 by an additional small perturbation
 one can modify the diffeomorphism $h_{m_{1,j},m_{2,j}}$ near $q_1^u$
 such that for the modified diffeomorphism $\tilde h_{m_{1,j},m_{2,j}}$
 the point $q_1^u$ will remain the periodic point
 of $\tilde h_{m_{1,j}, m_{2,j}} \circ f$
 and $\lambda^*$ will become exactly $1$. 
Indeed, by construction there is a neighborhood $W$ of $q_1^u$
 such that $\cO(q_1^u; h_{m_{1,j},m_{2,j}} \circ f) \cap W =q_1^u$
 holds for every $j$. 
Thus the closeness of the central multiplier $\lambda^*$ to $1$,
 together with the strong partial hyperbolicity 
 guarantees the existence of the required sequence of perturbations
 $\tilde{h}_{m_{1,j},m_{2,j}}$.

Thus, we have constructed the 1-flat periodic point $p_*=q_1^u$
 for the map $h\circ f$
 where $h=\tilde{h}_{m_{1,j},m_{2,j}}$ and $j$ is large enough.
Note that the orbit of $p_*$ lies entirely in $U_\Lambda$
 and intersects with both $U_1$ and $U_2$.
By the choice of $U_\Lambda$, $U_1$, and 
$U_2$,
the periodic point $p_*$ is linked to the blender.
\end{proof}
 
\subsection{$C^r$-flatness of the center-stable and center-unstable manifolds}

In the rest of this section we prove Proposition~\ref{prop:1-flat + AS}.
The proof is done by comparing the behavior of the first return map
 at the periodic point restricted to a central curve
 and the transition map at the heteroclinic point
 (see Proposition~\ref{prop:conv1}).
As our calculation shows,
 the $C^r$-distance between these maps may be quite large,
 see Remark~\ref{rmk:huge}. 
Nevertheless, we are able to establish the
closeness of the 
characteristics we are interested in, namely,
the non-linearities and the Schwarzian derivatives.

Note that for large enough $m_{1,2}$,
 the maps $f_{m_1,m_2}=\tilde{h}_{m_1, m_2} \circ f$ constructed
 in Proposition~\ref{prop:1-flat 1}
 keep many of the properties of the map $f$ itself, as detailed below.
We therefore will omit the indices $m_{1,2}$ but denote the original map $f$
 as $f_\infty$ from now on. 
So, $f\to f_\infty$ in the $C^\infty$-topology as $m_{1,2}\to+\infty$.
Note that the support of the perturbations $\tilde{h}_{m_1, m_2}$ is
 away from a neighborhood of the orbits of periodic points $p_{1, 2}$,
 so the return maps $f^{\pi_1}$ and $f^{\pi_2}$ coincide with
 $f_\infty^{\pi_1}$ and $f_\infty^{\pi_2}$ near $p_{1,2}$.
Thus, by shrinking the domain of the 
definition of $\vphi_{1, 2}$, we can assume 
that they remain linear in the charts:
\begin{equation}
\label{eqlini}
\vphi_i\circ f^{\pi_i} \circ (\vphi_i)^{-1}(x, y, z)
 = (\lambda_i x, A^s_i y, A^u_i z ), \qquad i =1,2,
\end{equation}
 where the origin corresponds to $\vphi_i (p_i)$,
 and $\pi_i$ is the period of $p_i$. 
We assume that the linearization is defined
 over polyballs $B_{\mathbf{d}, i}$, $i=1, 2$. 
Since $p_1$ and $p_2$ have a transverse heteroclinic connection,
 which is robust under small perturbations of the map,  
 there is a point $q^u_{1} \in (\vphi_1)^{-1}(B^{cu}_{\mathbf{d}, 1})$
 which is sent to 
 $q_{1}^s \in (\vphi_2)^{-1}(B^{cs}_{\mathbf{d}, 2})$
 by some iteration of $f$, that is,
 there exists $n_1>0$ such that $f^{n_1}(q^u_{1}) = q^s_{1}$. 
We choose the points $q_1^u$ and $q_1^s$
 continuously dependent on the map $f$,
 so they tend to the original heteroclinic points $q_1^{s,u}$
 as $f\to f_\infty$.

Near each of these points the intersection of
 the $(d_u+1)$-dimensional manifold $W^u(p_1)$
 and the $(d_s+1)$-dimensional manifold $W^s(p_2)$ is a smooth curve.
We denote the corresponding curves as $l^u_1$ and $l^s_1$.
We parameterize them such that they are given by local $C^\infty$ maps
 $l^u_1:(\RR,0) \ra (\vphi_1^{-1}(B_{\dd,1}^{cu}), q_1^u)$,
 $l^s_1:(\RR,0) \ra (\vphi_2^{-1}(B_{\dd,2}^{cs}),q_1^s)$,
 such that $\vphi_1^x(l^u_1(t))=\vphi_2^x(l^s_1(t))=t$.
The corresponding transition map $G:(\RR,0) \ra (\RR,0)$ is defined
 by relation (\ref{transg1}).
 It depends continuously on $f$ in 
 the $C^\infty$ topology, thus
 $S(G)$ and $A(G)$ have the same sign
 for every $f$ which is $C^3$ close to $f_\infty$.

For the original map $f_\infty$
 we also have heteroclinic points
 $q_2^u \in (\vphi_2)^{-1}(B^{uu}_{\mathbf{d}, 2})$,
 $q_2^s \in (\vphi_1)^{-1}(B^{ss}_{\mathbf{d}, 1})$
 such that $f^{n_2}(q^u_{2}) = q^s_{2}$ for some $n_2>0$.
These points correspond to a non-transverse intersection
 of $W^u(p_2)$ and $W^s(p_1)$, that is,
the sum of dimensions of these manifolds 
is less than $|\dd|$. 
Thus, in general, this heteroclinic intersection disappears
 as the map $f_\infty$ is perturbed.

By Proposition~\ref{prop:1-flat 1}
 the maps $f$ we consider here have a 1-flat periodic point
 $p_{\ast}$ of period $\pi_{\ast} :=m_1\pi_1+m_2\pi_2 +n_1 +n_2$. 
\begin{rmk}
\label{lamrmk}
Recall that the sequence $(m_{1, j}, m_{2,j})$ was chosen so that 
 $\lambda_1^{m_{1,j}} \lambda_2^{m_{2,j}}$ converges to some positive constant
 as $j\to +\infty$.
Therefore, we further assume that the product
 $\lambda_1^{m_1} \lambda_2^{m_2}$ is bounded away from zero and infinity.
\end{rmk}

The periodic point $p_{\ast}$ has the following itinerary: 
\begin{itemize}
\item $f^{n_1}(p_{\ast})$ is a point close to $q^s_{1}$,
\item $f^{n_1+\pi_2 j}(p_{\ast})$ is in the linearized region 
 $(\vphi_2)^{-1}(B_{\mathbf{d}, 2})$ for $j=0,\ldots, m_2$ 
 and $f^{n_1+m_2\pi_2}(p_{\ast})$ is a point close to $q^u_{2}$,
\item $f^{n_1+m_2\pi_2+n_2}(p_{\ast})$ is a point close to $q^s_{2}$, and  
\item $f^{n_1+m_2\pi_2+n_2+j\pi_1}(p_{\ast})$ is in the linearized region 
 $(\vphi_1)^{-1}(B_{\mathbf{d}, 1})$ for $j=0, \dots, m_1$.
\end{itemize}

Recall that $\mathcal{O}(p_{\ast})$ admits
 a partially hyperbolic splitting
 deriving from the strong partial hyperbolicity in $\cU_{\cC}$. 
By construction, the return map $f^{\pi_*}$ has derivative equal to $1$
 in the center direction at the point $p_*$.
Therefore, by the center manifold theorem,
 we can find a center-unstable manifold $\mathcal{W}^{cu}$ 
 passing through $p_{\ast}$
 which is $C^k$ where $k$ can be chosen arbitrarily large. 
We choose $k \geq 3$.
Similarly, we consider a center-stable $C^k$-manifold $\mathcal{W}^{cs}$
 through the point $f^{n_1}(p_{\ast})$.

The manifold $\mathcal{W}^{cu}$ is tangent to the subspace $E^{cu}$
 at the point $p_*$.
Since $p_*$ is close to $p_1$, $E^{cu}$ is close to $y=0$
in the linearizing chart $\vphi_1$.
Therefore, $\mathcal{W}^{cu}$ is a hypersurface of the form
 $y=\psi_{cu}(x,z)$ for some $C^k$-smooth function $\psi_{cu}$.
Similarly, $\mathcal{W}^{cs}$ is, in the chart $\vphi_2$,
 a surface of the form $z=\psi_{cs}(x,y)$
 for some $C^k$-smooth function $\psi_{cs}$.

In general,
 suppose $\Psi$ is a piece of a hypersurface in $B_{\mathbf{d}, i}$, $i=1,2$,
 which is a graph of a $C^k$-function $\psi$,
 that is, $\Psi = \{(x, \psi(x, z), z )\}$,
 where $x, z$ vary in some domain in $B_{\dd, i}^{cu}$.
Let $X= (x_0, y_0, z_0)$ be a point in $\Psi$.
We define $\| \partial^k_{cu} \Psi \|_{X}$
 as the value $\|\psi(x, z)-y_0\|_{(x_0, z_0), k}$ 
 (see Section~\ref{ss:matrix} for the definition of this semi-norm).
In a similar way we define $\|\partial^k_{cs} \Psi \|_{X}$
 where $\Psi$ is a surface of the form $\{(x, y, \psi(x, y))\}$.

\begin{lemma}
\label{lemma:surf}
As $m_1$ and $m_2$ go to $+\infty$,  
 $\|\partial^k_{cu} \mathcal{W}^{cu}\|_{\vphi_1(p_{\ast})}$
 and $\|\partial^k_{cs} \mathcal{W}^{cs}\|_{\vphi_2(f^{n_1}(p_{\ast}))}$ 
 converge to zero.
\end{lemma}
We only give a proof for the manifold $\cW^{cu}$.
The proof for $\cW^{cs}$ is obtained in the same way,
 due to the symmetry of the problem.
We will use the following:
\begin{lemma}
\label{lemma:esti}
Let $\Psi := \{(x, \psi(x, z), z) \} \subset B_{\dd, 1}$
 be a piece of $C^k$-surface such that $\vphi_1 (p_{\ast}) \in\Psi$
 and $\|\partial^k_{cu} \Psi\|_{\vphi_1(p_{\ast})} \leq \delta$.
 Suppose that the tangent space of $(\vphi_1)^{-1}(\Psi)$ at the point $p_{\ast}$
 is contained in the invariant cone field $\cC^{cu}$.
Let $\bar{\Psi} = \vphi_{1}\circ  f^{\pi_{\ast}}\circ \vphi_{1}^{-1}(\Psi)$.
Then $\bar\Psi$ is a surface of the form $y=\bar \psi(x, z)$
 and
\begin{equation}
\label{bpb}
\|\partial_{cu}^k \bar{\Psi}\|_{\vphi_1(p_{\ast})}
 \leq  C_{1, k}(\lambda_1^{m_1}\lambda_2^{m_2})^{-k}
 \|A_{1}^s\|^{m_1}\|A_{2}^s\|^{m_2}(\delta + C_{2, k}) 
+ \|A_{1}^s\|^{m_1}\lambda_1^{-km_1}C_{3,k},
\end{equation}
where $C_{1, k}$, $C_{2, k}$, and $C_{3, k}$ are constants
 independent of the choice of $m_i$,
 and $\lambda_{1,2}$, $A^s_{1,2}$ are defined in (\ref{eqlini}).
\end{lemma}

Once we have proven this lemma,
we can obtain Lemma~\ref{lemma:surf} for $\cW^{cu}$ as follows.
Let $\|\partial_{cu}^k \cW^{cu}\|_{\vphi_1(p_{\ast})} = \delta_{m_1,m_2}$.
Since $\cW^{cu}$ is invariant with respect to $f^{\pi_*}$,
 we find from (\ref{bpb}) that
\begin{equation*}
\delta_{m_1,m_2}\leq  C_{1, k}(\lambda_1^{m_1}\lambda_2^{m_2})^{-k}
\|A_{1}^s\|^{m_1}\|A_{2}^s\|^{m_2}(\delta_{m_1,m_2} + C_{2, k}) 
+ \|A_{1}^s\|^{m_1}\lambda_1^{-km_1}C_{3,k}.
\end{equation*}
Since $\lambda_1>1$, $\|A_{1}^s\|<1$, $\|A_2^s\| < 1$
 and $\lambda_1^{m_1}\lambda_2^{m_2}$ is uniformly bounded
 (see Remark \ref{lamrmk}), we have that
 $(\lambda_1^{m_1}\lambda_2^{m_2})^{-k}
  \|A_{1}^s\|^{m_1}\|A_{2}^s\|^{m_2} \to 0$ 
 and $\|A_{1}^s\|^{m_1}\lambda_1^{-km_1} \to 0$ 
 as $m_1, m_2 \to \infty$.
This implies that $\delta_{m_1,m_2} \to 0$ as $m_1, m_2 \to \infty$.

It remains to prove Lemma~\ref{lemma:esti}.

\begin{proof}
[Proof of Lemma~\ref{lemma:esti}]
Since the tangent space of $\Psi$ is contained
 in the invariant cone field $\cC^{cu}$,
 the tangent space of the image of $\Psi$ under the iteration of $f$ 
 also lies in $\cC^{cu}$.
In particular, it follows that $\bar \Psi$ 
 is also a surface of the form $y=\bar\psi(x,z)$. 

Put $\Psi_1 = \Psi$ and
\begin{gather*}
\Psi_2 = \vphi_2 \circ  f^{n_1} \circ \vphi^{-1}_1 (\Psi_1), \quad
\Psi_3 = \vphi_2 \circ  f^{m_2\pi_2} \circ \vphi^{-1}_2 (\Psi_2), \\
\Psi_4 = \vphi_1 \circ  f^{n_2} \circ \vphi^{-1}_2 (\Psi_3), \quad 
\bar\Psi =\Psi_5 = \vphi_1 \circ  f^{m_1\pi_1} \circ \vphi^{-1}_1 (\Psi_4).
\end{gather*}

First, let us estimate
 $\| \partial^k_{cu}\Psi_2\|_{\vphi_2( f^{n_1}(p_\ast))}$.
We have that $\|\partial_{cu}^k \Psi_1\|_{\vphi_1(p_\ast)} \leq \delta$.
Since $\Psi_2$ is a image
 by the map $\vphi_2 \circ  f^{n_1} \circ \vphi^{-1}_1$
 whose derivatives are bounded independently of the values
 of $m_1$ and $m_2$
 (since $n_1$ is a constant that does not depend on $m_{1,2}$),
 we have an estimate of the form
\begin{equation*}
\|\partial_{cu}^k \Psi_2\|_{\vphi_2( f^{n_1}(p_\ast))} \leq  C'_{1, k}\delta + C'_{2,k},
\end{equation*}
 where $C'_{i, k}$ are constants defined
 in terms of the supremum norms of the map
 $\vphi_2 \circ  f^{n_1} \circ \vphi^{-1}_1$. 

Next, we estimate
 $\|\partial_{cu}^k \Psi_3\|_{\vphi_2( f^{n_1+m_2\pi_2}(p_\ast))}$.
Let $(x, y, z) \in \Psi_3$.
As the map $ f^{\pi_2}$ is linear in the chart $\vphi_2$
 (see (\ref{eqlini})) and the effect of $\tilde{h}_{m_1, m_2}$ goes 
 to zero as $m_{1,2} \to +\infty$, we have
\begin{equation*}
y = (A^s_2)^{m_2}\psi_2((\lambda_2)^{-m_2}x, (A_2^u)^{-m_2}z), 
\end{equation*}
 where $\psi_2$ is the function whose graph is the hypersurface $\Psi_2$. 
By taking the derivatives up to order $k$,
 and taking in account that $\|(A_2^u)^{-m_2}\|<1$,
 we obtain the following estimate: 
\begin{equation*}
\|\partial_{cu}^k \Psi_3\|_{\vphi_2( f^{n_1+m_2}(p_\ast))} 
\leq  (A_2^s)^{m_2}\lambda_2^{-km_2}\|\partial_{cu}^k \Psi_2\|_{\vphi_2( f^{n_1}(p_\ast))}.
\end{equation*}
One can see that similar estimates hold
 for $\|\partial_{cu}^k \Psi_4\|$ and $\|\partial_{cu}^k \Psi_5\|$.
By combining all the estimates, we obtain the conclusion.
\end{proof}

\subsection{Calculation of $A$ and $S$}

Let us compute the non-linearity and the Schwarzian derivative at $p_{\ast}$.
For this purpose we fix a center curve passing through $p_{\ast}$.
A convenient curve is
$\ell_1 = \vphi_1( \mathcal{W}^{cu})
 \cap( \vphi_1\circ  f^{-n_1}\circ \vphi_2^{-1})(\mathcal{W}^{cs})$. 
It is tangent to the center direction $E^c$ at $\vphi_1(p_{\ast})$.
 Notice that, by Lemma~\ref{lemma:surf}, 
 $\vphi_1(\mathcal{W}^{cu})$ is $C^3$-close to $\vphi_1(W^{u}(p_1))$
 and $\vphi_2(\mathcal{W}^{cs})$ is $C^3$-close to $\vphi_1(W^{s}(p_2))$
 in a neighborhood of $f^{n_1}(p_{\ast})$,
 if $m_1$ and $m_2$ are sufficiently large. 
Thus, the curve $\ell_1$ is $C^3$-close to the curve $l_1^u$,
 i.e., to the heteroclinic intersection 
$\vphi_1(W^{u}(p_1) \cap W^{s}(p_2))$ near the point $\vphi(q_1^u)$. 

We denote 
\begin{gather}
\label{eltm}
\ell_2 = F_1(\ell_1)
  := \vphi_2\circ f^{n_1}\circ \vphi_1^{-1}(\ell_1), \quad 
\ell_3 = F_2(\ell_2)
 := \vphi_2\circ f^{m_2\pi_2}\circ \vphi_2^{-1}(\ell_2), \\ \label{eltm1}
\ell_4 = F_3(\ell_3)
 :=\vphi_1\circ f^{n_2}\circ \vphi_2^{-1}(\ell_3), \quad
\ell_5 = F_4(\ell_4)
 :=\vphi_1\circ f^{m_1\pi_1}\circ \vphi_1^{-1}(\ell_4). 
\end{gather}
We want to analyze the nonlinearity and 
 the Schwarzian derivative $A(F)$ and $S(F)$
 of the map $F=\vphi_1\circ f^{\pi_\ast}\circ\vphi^{-1}_1$
 defined on the curve $\ell_1$.
Note that $F = F_4 \circ F_3 \circ F_2 \circ F_1$,
 where the maps $F_i$ are defined in formulas (\ref{eltm}) and (\ref{eltm1}).

Let $(x, y, z)$ be the coordinates of $B_{\mathbf{d}, 1}$ and 
$(\tilde{x}, \tilde{y}, \tilde{z})$ be the coordinates of $B_{\mathbf{d}, 2}$.
Let the point $\vphi_1(p_*)\in\ell_1$ have coordinates $(x_0,y_0,z_0)$,
 the point $F_1(\vphi_1(p_*)) \in \ell_2$ have coordinates
 $(\tilde x_0,\tilde y_0,\tilde z_0)$,
 the point $F_2\circ F_1(\vphi_1(p_*))\in\ell_3$ have coordinates
 $(\tilde x_1,\tilde y_1,\tilde z_1)$,
 and the point $F_3 \circ F_2\circ F_1(\vphi_1(p_*)) \in \ell_4$ have coordinates
 $(x_1, y_1, z_1)$.
We set 
\begin{gather}
\label{parmtel}
\ell_{1} = \{ (x=x_0+t, y=y_0+\sigma_1(t), z=z_0+\eta_1(t)) \}, \\
\label{parmtel2}
\ell_{2} = \{ (\tilde{x}=\tilde x_0+t, \tilde y =\tilde y_0+\sigma_2(t),
   \tilde z=\tilde z_0 + \eta_2(t)) \}, \\
\label{parmtel3}
\ell_{3} = \{ (\tilde{x}=\tilde x_1+t, \tilde y =\tilde y_1 + \sigma_3(t),
   \tilde z=\tilde z_1+\eta_3(t)) \}, \\
\label{parmtel4}
\ell_{4} = \{ (x=x_1+t, y=y_1+ \sigma_4(t), z=z_1+\eta_4(t)) \}.
\end{gather}
Since the curve $\ell_1$ is at least $C^3$-close to the curve $l_1^u$
 for large $m_{1,2}$,
 the restrictions of the map $F_1$ to these two curves
 are also $C^3$-close.
Therefore, the map $F_1$ is, in the parameterization
 (\ref{parmtel}),(\ref{parmtel2})
 at least $C^3$-close to the transition map $G$
 for the heteroclinic point $q_1^u$, as defined by (\ref{transg1}).

Let us state the goal of this section in the form of the proposition. 
\begin{prop}
\label{prop:conv1}
As $m_1, m_2 \to \infty$, we have
$$A(F)(\vphi_1(p_*))  \to A(G)(\vphi_1(q)), \qquad S(F)(\vphi_1(p_*)) \to S(G)(\vphi_1(q)) .$$ 
In particular, if $m_1$ and $m_2$ are sufficiently large, then the signature 
of the 1-flat point $p_{*}$ is the same as for the heteroclinic point $q$. 
\end{prop}
\begin{proof}
We have
\begin{equation*}
A(F) = A((F_4\circ F_3) \circ (F_2\circ F_1) ),
 \qquad S(F) = S((F_4\circ F_3) \circ (F_2\circ F_1)).
\end{equation*}
Notice that the map $F_2$ is affine: 
\begin{equation}
\label{t2l}
\tilde{x} \mapsto \lambda_2^{m_2}\tilde{x}.
\end{equation}
Thus,  we have $A(F_2) =0$, $S(F_2)=0$. Hence, by (\ref{cocpr}),
\begin{equation*}
 A(F_2 \circ F_1) = A(F_2)\cdot F_1' + A(F_1) = A(F_1), \qquad 
 S(F_2 \circ F_1) = S(F_1). 
\end{equation*}
Similarly, we have $A(F_4 \circ F_3) =A(F_3)$, $S(F_4 \circ F_3) =S(F_3)$.
Hence,
\begin{equation}
\label{as34}
 A(F) = A(F_3)\cdot(F_2\circ F_1)'+A(F_1),
 \qquad S(F) = S(F_3)\cdot\left[(F_2\circ F_1)'\right]^2+S(F_1).
\end{equation}

As we mentioned,
 by the $C^k$-convergence of $\ell_1$ to $l_1^u$,
 the map $F_1$ converges to the transition map $G$ in $C^3$ topology,
 which implies $A(F_1)\to A(G)$ and $S(F_1)\to S(G)$ as $m_{1,2}\to+\infty$.
Note also that
\begin{equation*}
 (F_2\circ F_1)' = O(\lambda_2^{m_2}), 
\end{equation*}
 as follows from (\ref{t2l}) and the uniform boundedness of $F_1$ in $C^k$.
Therefore, we will immediately obtain Proposition~\ref{prop:conv1}
 from equation (\ref{as34}),
 once we prove the following:
\begin{prop}
\label{prop:esti}
As $m_1, m_2 \to +\infty$, 
\begin{equation*}
 A(F_3) = o(\lambda_2^{-m_2}), \qquad S(F_3) = o(\lambda_2^{-2m_2}). 
\end{equation*}
\end{prop}
The proof these estimates involves calculations of
the derivatives of $F_3$ up to order 3.
For this purpose,
 we first collect information about the curves $\ell_{3}$ and $\ell_4$.
Since $\ell_3$ is the image of the curve $\ell_2$
 by $f^{m_2\pi_2}$ and $\ell_4$ is the image of $\ell_1$
 by $f^{-m_1\pi_1}$,
 it follows from (\ref{parmtel})-(\ref{parmtel4}) and (\ref{eqlini}),
 that
\begin{align}
\label{s3e4}
 \sigma_3(\tilde{x}-\tilde x_1)
 & = (A_2^s)^{m_2}\sigma_2(\lambda_2^{-m_2}(\tilde{x}-\tilde x_1)),\\
\label{s3e40}
\eta_4(x-x_1) & = (A_1^u)^{-m_1}\eta_1(\lambda_1^{m_1}(x-x_1)).
\end{align}

Notice that, $\ell_2$ is contained in $\mathcal{W}^{cs}$ and 
$\ell_1$ is contained in $\mathcal{W}^{cu}$.
Since these surfaces are close to $W^s(p_2)$ and $W^u(p_1)$ respectively,
 the derivatives of the functions $\sigma_2$ and $\eta_1$ in
 (\ref{parmtel}), (\ref{parmtel2}) are uniformly bounded.
Thus, we infer from (\ref{s3e4})
 (together with the fact that $o(\lambda_1^{m_1})=o(\lambda_2^{-m_2})$,
 see Remark \ref{lamrmk})
 that
\begin{align}
\label{st1}
 \sigma_3'  & = o(1), &
  \sigma_3'' & = o(\lambda_2^{-m_2}), &
 \sigma_3''' & = o(\lambda_2^{-2m_2}), \\
\label{st2}
  \eta_4' & = o(1), &
  \eta_4''& = o(\lambda_2^{-m_2}), &
 \eta_4'''& = o(\lambda_2^{-2m_2}).
\end{align}
Let
\begin{equation*}
 \Phi(\tilde{x}, \tilde{y}, \tilde{z}) 
 =\vphi_1 \circ f^{n_2} \circ \vphi_2^{-1} (\tilde{x}, \tilde{y}, \tilde{z})
 =(x, y, z) 
\end{equation*}
 be the coordinate representation of the map
 $\vphi_1 \circ f^{n_2} \circ \vphi_2^{-1}$
 from a neighborhood of the point
 $\vphi_2(f^{n_1+m_2\pi_2}(p_{\ast})) \in B_{\dd, 2}$
 to a neighborhood of the point
 $\vphi_1(f^{n_1+m_2\pi_2+n_2}(p_{\ast})) \in B_{\dd, 1}$.
It will be convenient for us to rewrite this map
 in the so-called {\em cross-form} (see e.g. \cite{GST8}). 

Let us explain what it is. Since the map $f$ is strongly partially hyperbolic,
 its iteration $\Phi: (\tilde{x}, \tilde{y}, \tilde{z})\mapsto (x,y,z)$
 is also strongly partially hyperbolic,
 so $\Phi$ takes a hypersurface whose tangent space lies in $\cC^{cu}$
 into a surface whose tangent space lies in $\cC^{cu}$.
In particular, it takes a surface $(\tilde x,\tilde y)=\text{const.}$
 into a surface of the form $z=\psi(x,y)$
 where $\psi$ is a smooth function with uniformly bounded derivative
 (for all $m_{1,2}$ large enough).
This means that $\partial z/\partial \tilde z$ is invertible
 (uniformly for all $m_{1,2}$ large enough).
Therefore, there exists a local diffeomorphism $U$
 (the cross-form of $\Phi$) such that
\begin{equation}
\label{crf}
 (x, y, \tilde{z}) = U(\tilde{x}, \tilde{y}, z)
\end{equation}
 if and only if $(x,y,z)=\Phi(\tilde{x}, \tilde{y}, \tilde{z})$.
The derivatives of $U$ up to any given order are uniformly bounded.

We denote the first coordinate of $U$ by $u$, that is, we put 
\begin{equation}
\label{xcross}
 x = u(\tilde{x}, \tilde{y}, z)
\end{equation}
 in (\ref{crf}).
Note that $\partial u/\partial \tilde{x}$ is bounded away from zero.
Indeed, by differentiating (\ref{xcross}), we get
\begin{equation}
\label{dfcrf}
 dx
  = \frac{\partial u}{\partial \tilde x} d\tilde{x}
  + \frac{\partial u}{\partial \tilde y} d\tilde{y}
  + \frac{\partial u}{\partial z} d z.
\end{equation}
By the strong partial hyperbolicity of $\Phi$,
 the image of any hypersurface whose tangent space lies in $\cC^{cu}$
 is transverse to any hypersurface whose tangent space lies in $\cC^s$.
In particular,
 the image under $\Phi$ of $\tilde y= \text{const.}$ and the hyperplane $(x,z)= \text{const.}$
 are transverse.
The transversality condition implies that the relation
\begin{equation*}
0 = \frac{\partial u}{\partial \tilde x}\; d\tilde{x},
\end{equation*}
which is obtained by putting $d\tilde y=0$ and $(dx,dz)=0$ in (\ref{dfcrf}),
 has only trivial solution $d\tilde x=0$.
This means $\partial u /\partial \tilde x\neq 0$
 for the diffeomorphism $f_\infty$, and hence,
 for every $C^1$ close diffeomorphism $f$
 that is, for $m_{1,2}$ sufficiently large.

By construction,
 we have the following formula for the map $F_3: \ell_3\to \ell_4$
 (whose derivatives we need to estimate):
\begin{equation}
\label{f3cross}
 F_3(\tilde x)
  =u(\tilde x, \;\tilde y_0+\sigma_3(\tilde x-\tilde x_1),
   \;z_1 + \eta_4(F_3(\tilde x)-x_1))
\end{equation}
 (see (\ref{parmtel3}), (\ref{xcross})).
By differentiating both sides by $\tilde x$,
 we obtain 
\begin{equation*}
F_3' =\frac{\partial u}{\partial \tilde x}
+ \frac{\partial u}{\partial \tilde y} \sigma_3'
+ \frac{\partial u}{\partial z} \eta_4'  F_3'
\end{equation*}
 and hence,
\begin{equation}
\label{fstar}
F_3' = \frac{\,\displaystyle \frac{\partial u}{\partial \tilde x}
+\frac{\partial u}{\partial \tilde y} \sigma_3'\,}{\displaystyle 1- \frac{\partial u}{\partial z} \eta_4'},
\end{equation}
 where the derivatives of $u$ are evaluated at
 $(\tilde x, \tilde y_0+\sigma_2(\tilde x-\tilde x_1),
  z_1 + \eta_4(F_3(\tilde x)-x_1))$,
 and $\eta_4'$ at $F_3(\tilde x)-x_1$.

Since $\sigma'_3=o(1)$, $ \eta'_4=o(1)$ (see (\ref{st1}) and (\ref{st2})),
 and the derivatives of $u$ are uniformly bounded,
 we find that
\begin{equation*}
 F_3'= \frac{\partial u}{\partial \tilde x}+o(1). 
\end{equation*}
Since $\partial u/\partial \tilde{x}$ is bounded away
 from zero and infinity, we find that
\begin{equation}
\label{f3pb}
 F_3'=O(1), \qquad (F_3')^{-1}=O(1),
\end{equation}
 uniformly for all large $m_{1,2}$.  

Let us differentiate equation (\ref{fstar}) with respect to $\tilde{x}$.
Note that the partial derivatives of $u$ are differentiated
 by the rule
\begin{equation*}
\frac{d}{d\tilde x}
 = \frac{\partial}{\partial \tilde x}
 + \sigma_3'\frac{\partial}{\partial \tilde y}
 + \eta_4' F_3' \frac{\partial}{\partial \tilde z}
\end{equation*}
 and the function $\eta_4$ and its derivatives $\eta_4^{(s)}$
 are differentiated by the rule
\begin{equation*}
\frac{d}{d\tilde x} (\eta_4^{(s)}) = \eta_4^{(s+1)} \cdot F_3'. 
\end{equation*}
Since  $\sigma_3'=o(1)$ and $\eta_4'=o(1)$ (see (\ref{st1}) and (\ref{st2})),
 $u$ and its derivatives are bounded and
 $F_3'$ is bounded by (\ref{f3pb}),
 we find from (\ref{fstar}) that
\begin{align*}
|F_3''| 
 & \leq C_1 +C_2 |\sigma_3''| + C_3 |\eta_4''|,\\
|F_3'''|
& \leq  C_1 +C_2 |\sigma_3''| + C_3 |\eta_4''|
 +C_4 |\sigma_3'''| + C_5 |\eta_4'''|
\end{align*}
 where $C_{1,2,3,4,5}$ are some constants.
By (\ref{st1}) and (\ref{st2}),  this gives us
\begin{equation}\label{f23}
F_3''  =o(\lambda_2^{-m_2}), \qquad F_3'''  =o(\lambda_2^{-2m_2}).
\end{equation}
By combining the estimates (\ref{f23}) and (\ref{f3pb})
 in the definition (\ref{asdef})
 of the nonlinearity $A$ and the Schwarzian $S$,
 we immediately obtain Proposition~\ref{prop:esti},
 which concludes the proof of Proposition~\ref{prop:conv1}.
\end{proof}

\begin{rmk}
\label{rmk:huge}
This calculation does not preclude of the map $F_3$
 having very large second and third derivatives,
 in some cases resulting in very large derivatives of the map $F$.
\end{rmk}

%
%

\section{$k$-flat periodic points}\label{kflatsec}
In this section, we prove Proposition~\ref{prop:r-flat}.

\subsection{Germs of one-dimensional diffeomorphisms}
We start with a result about the composition of germs.
\begin{prop}
\label{prop:germ}
Let $k$ be a positive integer.
Let $\{F_i\}_{i=1,\dots,8}$ and $\{G_i\}_{i=1,\dots,8}$
 be orientation-preserving germs in $\Diff_\loc^\infty(\RR,0)$ such that
 each $F_i$ is $k$-flat.
If $k=1$, assume also that
\begin{equation}
\label{keq1a} A(F_1) \cdot A(F_3)<0, \quad
 S(F_1) \cdot S(F_3)<0, \quad
 |S(F_1)/A(F_1)| >|S(F_3)/A(F_3)|.
\end{equation}
If $k=2$, assume that
\begin{equation}
\label{keq2a}
 S(F_1) \cdot S(F_3)<0.
\end{equation}
Then, for any neighborhood $\cV$ of the identity map in $\cP^{k+1}(\RR,0)$
 and any $N \geq 1$,
 there exist polynomial maps $H_i\in \cV$
 and integers $n_i\geq N$ ($i=1,\dots,8$) such that the germ
\begin{equation}
\label{fbard}
 \bar{F}=G_8 \circ (H_8 \circ F_8)^{n_8} \circ \dots \circ G_1
 \circ (H_1 \circ F_1)^{n_1}
\end{equation}
 is $(k+1)$-flat.
Furthermore, $S(\bar{F}) \cdot S(F_1)>0$ in the case $k=1$.
\end{prop}
\begin{proof}
For a proof, we use several results of \cite{AST}.  
\begin{lemma}
[{\cite[Lemma 3.1]{AST}}]
\label{lemma: AST3.1}
For any $\Phi \in \Diff_\loc^\infty(\RR,0)$ satisfying $\Phi' >0$
 and for every $k \geq 1$,
 there exists a continuous (in the $C^{\infty}$-topology) family of germs 
 of diffeomorphisms $\{\Phi^\mu\}_{\mu \in \RR}$ such that 
\begin{equation*}
 \Phi^0=\Id, \qquad \Phi^1(t)=\Phi+o(t^{k}), 
\end{equation*}
and 
\begin{equation*}
\Phi^\mu \circ \Phi^{\mu'}(t)=\Phi^{\mu+\mu'}(t)+o(t^{k}) 
\end{equation*}
 for all $\mu, \mu' \in \RR$.
\end{lemma}

\begin{lemma}
[{\cite[Lemma 3.5]{AST}}]
\label{lemma: AST3.5}
Let $Q_1$ and $Q_2$ from $\Diff_\loc^\infty(\RR,0)$ be $1$-flat germs.
Assume that $A(Q_1) \cdot A(Q_2)<0$, $S(Q_1) \cdot S(Q_2)<0$,
 and $|S(Q_1)/A(Q_1)|>|S(Q_2)/A(Q_2)|$.
Then, for any neighborhood $\cV$ of the identity map in $\cP^2(\RR,0)$
 and any $\alpha,\beta \in \RR$, there exists a $1$-flat map $H \in \cV$
 such that
\begin{align*}
 A(Q_2^n\circ (H \circ Q_1)^m)+\alpha & = 0,\\
 S(Q_1) \cdot \left(S(Q_2^n\circ (H \circ Q_1)^m))+\beta \right) & >0
\end{align*}
for some integers $m,n \geq 1$.
\end{lemma}

\begin{lemma}
[{\cite[Lemma 3.6]{AST}}]
\label{lemma: AST3.6}
Let $Q_1$ and $Q_2$ from $\Diff_\loc^\infty(\RR,0)$ be $2$-flat germs 
 such that $Q_1'''(0) \cdot Q_2'''(0) <0$.
Then, for any neighborhood $\cV$ of the identity map in $\cP^3(\RR,0)$
 and any $\gamma \in \RR$,
 there exists a $1$-flat polynomial $H \in \cV$ such that
\begin{equation}
\label{f2mh}
 Q_2^n \circ (H \circ Q_1)^m(t) =t +\gamma t^3+o(t^3)
\end{equation}
 for some integers $m,n\geq 1$.
Moreover, $m$ and $n$ can be taken arbitrarily large
\end{lemma}
\begin{rmk}
According to Lemma~\ref{lemma: AST3.6}, 
the fact $m$ and $n$ can be chosen arbitrarily large
is not stated explicitly in the original statement
but it is clear from the proof.
\end{rmk}

\begin{lemma}
 [{\cite[Lemma 3.7]{AST}}]
\label{lemma: AST3.7}
Suppose $k \geq 3$.
Let $Q_1,\dots,Q_4$ from $\Diff_\loc^\infty(\RR,0)$ be $k$-flat germs.
Then, for any neighborhood $\cV$ of the identity map in $\cP^{k+1}(\RR,0)$
 and any $\alpha \in \RR$,
 there exist $1$-flat maps $R_1,\dots,R_4 \in \cV$ such that
\begin{equation}
\label{h4f4}
 (R_4 \circ Q_4)^n \circ \dots \circ (R_1 \circ Q_1)^n(t)
 =t+\alpha t^{k+1} +o(t^{k+1})
\end{equation}
 for an integer $n$ which can be chosen arbitrarily large.
\end{lemma}

\begin{rmk}
In the original formulation of Lemma~3.6 and Lemma~3.7 in \cite{AST} above, 
 the germs $H$ are $R_i$ are $C^\infty$ diffeomorphisms.
However,
 if we replace them by their Taylor polynomials up to order $k+1$,
 the relations (\ref{f2mh}) and (\ref{h4f4}) will still hold.
Therefore, we can take $H$ and $R_i$ as polynomial maps.
The $1$-flatness of the maps $H$ and $R_i$ immediately follows
 from the explicit formulas for them given
 in the proof of Lemmas~3.6 and 3.7 in \cite{AST}.
\end{rmk}

We can now proceed to the proof of Proposition~\ref{prop:germ}.
First, we define the maps 
 $H_2, H_4, H_6, H_8$ and the numbers $n_2,n_4,n_6,n_8$.
For that, we apply Lemma \ref{lemma: AST3.1}
 to each of the maps $\Phi_j=G_{2j}^{-1} \circ G_{2j-1}^{-1}(t)$
 for $j=1,\dots,4$.
This gives us four continuous families
 $\{ \Phi_j^\mu \}_{\mu \in \RR}$ of smooth germs
 such that $\Phi_j^0$ is the identity map,
 $\Phi_j^1(t)=G_{2j}^{-1} \circ G_{2j-1}^{-1}(t)+o(t^k)$,
 and $\Phi_j^{\mu+\mu'}(t)=\Phi_j^\mu \circ \Phi_j^{\mu'}(t)+o(t^k)$.
Let $H_{2j}$ be the Taylor polynomial 
 of $\Phi_j^{1/n_{2j}}(t)$ up to order $k$. 
We take $n_{2j} \geq N$ large enough such that $H_{2j} \in \cV$.

With this choice of $H_{2,4,6,8}$ and $n_{2,4,6,8}$, we have
\begin{equation}
\label{repl}
\bar{F}=\bar F_8 \circ (H_7 \circ F_7)^{n_7}
  \circ \bar F_6 \circ (H_5 \circ F_5)^{n_5}
  \circ \bar F_4 \circ (H_3 \circ F_3)^{n_3}
  \circ \bar F_2 \circ (H_1 \circ F_1)^{n_1}
\end{equation}
 for the map $\bar F$ from (\ref{fbard}),
 where we denote
\begin{equation*}
 \bar{F}_{2j}=G_{2j} \circ (H_{2j} \circ F_{2j})^{n_{2j}} \circ G_{2j-1}. 
\end{equation*}
Since $F_{2j}$ is $k$-flat,
 it follows from the construction of $\bar H_{2j}$ that
 $\bar{F}_{2j}$ is also $k$-flat:
\begin{equation*}
\bar{F}_{2j}(t)
 =G_{2j} \circ H_{2j}^{n_{2j}} \circ G_{2j-1}(t)+o(t^k)=t+o(t^{k+1}).
\end{equation*}
The maps $H_{1,3,5,7}$ (which will be determined below) are $1$-flat,
 so the maps $(H_{2j-1} \circ F_{2j-1})^{n_{2j-1}}$ in (\ref{repl})
 are also $1$-flat ($j=1,\dots,4$).
Therefore, it follows from the $k$-flatness of the maps $\bar F_{2j}$ that
\begin{equation}
\label{repl0}
\bar{F}
 =\hat F \circ (H_7 \circ F_7)^{n_7} \circ (H_5 \circ F_5)^{n_5}
  \circ (H_3 \circ F_3)^{n_3} \circ (H_1 \circ F_1)^{n_1}
  + o(t^{k+1}),
\end{equation}
where
\begin{equation}
\label{htf}
 \hat F=\bar F_8 \circ \bar F_6 \circ \bar F_4 \circ \bar F_2.
\end{equation}

We can now prove Proposition~\ref{prop:germ} for the case $k\geq 3$.
Indeed, apply Lemma \ref{lemma: AST3.7}
 to the quadruple $(Q_1,Q_2,Q_3,Q_4)=(F_1,F_3,F_5,F_7)$
 and the constant $\alpha=-\frac{1}{k!}\hat{F}^{(k+1)}$.
Then, letting $H_{2j-1}=R_j$ and $n_{2j-1}=n$
 where $n$ and $R_j$ ($j=1,\dots,4$) are given by Lemma \ref{lemma: AST3.7},
 the map $\bar F$ becomes $(k+1)$-flat 
(see (\ref{h4f4}) and (\ref{repl0})), as required.

In the cases $k=1$ and $k=2$, let $H_3=H_5=H_7=\mathrm{Id}$ and $n_5=n_7=N$.
Then, equation (\ref{repl}) becomes
\begin{equation*}
\bar{F}
 = \bar{F}_8 \circ F_7^N \circ \bar{F}_6 
 \circ F_5^N \circ \bar{F}_4 \circ F_3^{n_3}
 \circ \bar{F}_2 \circ (H_1 \circ F_1)^{n_1}
\end{equation*}
As all the maps in the right-hand side of this formula are $1$-flat,
 the map $\bar F$ is also $1$-flat (at least).
Note also that it follows from the $1$-flatness
 of all the maps in the right-hand side
 and from the cocycle property (\ref{cocpr}) that 
\begin{align}
\label{asbarf1}
A(\bar{F})
 & = A(\hat{F}) + N\cdot (A(F_7)+A(F_5))+
  A(F_3^{n_3}\circ (H_1 \circ F_1)^{n_1}), \\
\label{asbarf2}
S(\bar{F})
 & = S(\hat{F}) + N\cdot (S(F_7)+S(F_5))
 + S(F_3^{n_3}\circ (H_1 \circ F_1)^{n_1}),
\end{align}
 where $\hat F$ is given by (\ref{htf}).

Define constants $\alpha$ and $\beta$ by
\begin{align}
\label{defalbk1}
\alpha & = A(\hat{F}) + N \cdot (A(F_5)+A(F_7)), &
\beta  & = S(\hat F) + N\cdot(S(F_5)+S(F_7)).
\end{align}
Recall that $A$ vanishes for $2$-flat maps,
 therefore $\alpha=0$ if $k = 2$.

In the case $k=1$, apply Lemma \ref{lemma: AST3.5}
 to the pair of germs $(Q_1,Q_2)=(F_1,F_3)$.
 Letting $H_1=H$, $n_1=m$, and $n_2=m$,
 where $H$ are $m$ are given in the conclusion of Lemma~\ref{lemma: AST3.5},
 we have $A(\bar F) =0$, see (\ref{asbarf1}) and (\ref{defalbk1}).
 This means the $2$-flatness of the map $\bar F$.
 We also have $S(F_1) \cdot S(\bar F) >0$ as required.

In the case $k=2$, 
notice that $S(F_1)\cdot S(F_3) <0$ and the $2$-flatness of $F_1$ and $F_3$
imply $F_1^{(3)}\cdot F_3^{(3)} <0$. 
Thus, we can apply Lemma \ref{lemma: AST3.6} to $(Q_1,Q_2)=(F_1,F_3)$
and the constant $\gamma=-\beta/6$.
Then, letting  $H_1=H$, $n_1=m$, and $n_2=m$,
where $H$, $m$, and $n$ are given in the conclusion 
of Lemma~\ref{lemma: AST3.6},
 we obtain 
 $A(\bar F)=0$ and $S(\bar F)=0$, see (\ref{asbarf1}), (\ref{asbarf2}), and (\ref{defalbk1}).
 This implies that the 3-flatness of the map $\bar F$.
\end{proof}

\subsection{Construction of $k$-flat points}\label{kflatsubsec}

We start the proof of Proposition~\ref{prop:r-flat}. First, given eight 
$k$-flat periodic points, after a preparatory perturbation, 
we apply Lemma~\ref{lemma:connecting} repeatedly
 to obtain a network of heteroclinic connections between the periodic points. 
Then, we use Lemma~\ref{lemma:r-flat 2} and Proposition~\ref{prop:germ}
 together and construct a $(k+1)$-flat periodic point. 

Let us recall the setting.
We consider a diffeomorphism $f \in \cW^1 \cap \Diff^\infty(M)$. 
Thus, there exist open subsets $U_{\cC}$ and $U_{\bl}$ of $M$
 with $\cl{U_{\bl}} \subset U_{\cC}$
 such that $f$ admits an invariant transverse pair of cone fields on $U_{\cC}$,
 preserves an orientation on it, and admits a blender in $U_{\bl}$.

\begin{lemma}
\label{lemma:network}
Suppose $f \in \cW^1 \cap \mathrm{Diff}^{\infty}(M)$ has
 eight $k$-flat periodic points
 $p_1,\dots,p_8 \in \Per_{\dd}(f, U_{\cC})$
 with mutually different orbits
 and there exists a blender in an open set $U_{\bl}$
 with $\cl{U_{\bl}} \subset U_{\cC}$
 such that each $p_i$ is linked to $U_{\bl}$
 but not contained in $\cl{U_{\bl}}$.
Then, for any given neighborhood $V$ of $\{ p_i\}$
 and any $C^{\infty}$ neighborhood $\cV \subset \Diff^{\infty}(M)$
 of the identity map, 
 there exists a diffeomorphism $h \in \cV$ such that
 $\supp(h) \subset V \setminus \left( \bigcup_{i=1}^8 \cO(p_i)\right)$
 and
 $W^{uu}(p_i, U_{\cC} ;h \circ f) \cap W^{ss}(p_{i+1},U_{\cC} ;h \circ f)
 \neq \emptyset$, 
 for every $i =1, \dots, 8$ (we put $p_9 = p_1$).
\end{lemma}
This lemma creates a heteroclinic network of flat periodic points.
The next lemma produces a $(k+1)$-flat point 
by adding a perturbation to this network.

\begin{lemma}
\label{lemma:r-flat 1}
Let $k \geq 1$ and
 suppose that $f \in \cW^1 \cap \mathrm{Diff}^{\infty}(M)$ admits
 eight $k$-flat periodic points 
 $p_1,\dots,p_8 \in \Per_{\dd}(f, U_{\cC})$
 with mutually different orbits
 and there exists a heteroclinic point
 $q_{j} \in W^{uu}(p_{j},U_{\cC}) \cap W^{ss}(p_{j+1}, U_{\cC})$
 for each $j=1,\dots,8$ (where we put $p_9 = p_1$). 
We also suppose that each $p_{i}$ admits $C^\infty$ Takens coordinates
 with a center germ $F_{i}$.
Put  $\Lambda = \bigcup_{i=1}^8 (\cO(p_i) \cup \cO(q_i))$.
Then, for any neighborhood $\cU \subset \Diff^{\infty}(M)$ 
 of the identity map,
 any neighborhood $V \subset M$ of $\Lambda$,
 and any neighborhoods $U_{p_1},\dots,U_{p_8}$ of $p_{1}, \dots, p_{8}$
 respectively,
 there exist eight germs $G_{1},\dots,G_{8} \in \Diff_\loc^\infty(\RR,0)$
 such that for any 
 $H_{1},\dots,H_{8}\in \cP^{k+1}(\RR,0)$
that are sufficiently close to the identity maps
 and any sufficiently large integers $n_{1},\dots,n_{8}$,
 there exist $h \in \cU$ and a periodic point
 $\bar{p} \in \Per_\dd(h \circ f, U) \setminus \cl{U_\bl}$
 for which the following holds:
 \begin{itemize} 
 \item $\cO(\bar{p}; h \circ f) \subset V$, 
 $\cO(\bar{p} ;h \circ f)
  \cap U_{p_1} \neq \emptyset, 
  \cO(\bar{p} ;h \circ f) \cap U_{p_8} \neq \emptyset$,
 \item $\supp(h) \subset \bigcup_{i=1}^{8}U_{p_i}$,
 \item the germ 
\begin{equation}
\label{asteqbf}
\bar{F} = G_{8} \circ (H_{8} \circ F_{8})^{n_{8}}
 \circ \cdots \circ G_{1} \circ (H_{1} \circ F_{1})^{n_{1}} +o(t^{k+1})
\end{equation}
 is a center germ of $\bar{p}$ of $h \circ f$
 for some $(k+1)$-central curve, and
\item the period of $\bar{p}$ is larger than $\sum_{i=1}^8 n_i$.
\end{itemize}
\end{lemma}

First, let us show
 how Proposition~\ref{prop:r-flat} and Remark~\ref{rmk:peri}
 follow from these lemmas and Proposition~\ref{prop:germ}.
\begin{proof}
[Proof of Proposition~\ref{prop:r-flat} and Remark~\ref{rmk:peri}]
Let eight $k$-flat periodic points $p_1,\dots,p_8$ satisfying 
 the assumptions of Proposition~\ref{prop:r-flat} be given.
 
For each $p_i$, thanks to their flatness
 we can find an $m$-central curve for any $m$. 
Put $\tilde{k} = \max \{ k, 3\}$. 
 Take a $\tilde{k}$-central curve $\ell_p$
 and apply Lemma~\ref{lemma:Takens} to $p_i$ and $\ell_p$.
 Then, by adding a perturbation,
 which is arbitrarily small in the $C^{\infty}$ topology and
 whose support is contained in an arbitrarily small neighborhood of ${p_i}$,
 we can assume that each $p_i$ admits $C^{\infty}$ Takens coordinates
 and $p_i$ is still $k$-flat with the same central germ.   
 
Next, we apply Lemma~\ref{lemma:network}.
Then by adding an arbitrarily small perturbation to $f$
we obtain the heteroclinic connection 
$W^{uu}(p_i; f) \cap W^{ss}(p_{i+1}; f) \neq \emptyset$
for every $i$.
Since the support of this perturbation can be
 chosen arbitrarily close to $p_i$
 but is disjoint from $p_i$,
we can assume that each $p_i$ still admits Takens coordinates,
possibly with a smaller domain of definition. 
Notice that this perturbation also does not change the 
central germs.

When $k=1$ we perform preparatory coordinate transformations 
that gives us the condition on $S(F)/A(F)$ in the 
assumptions of Proposition~\ref{prop:germ} if necessary,
see Remark~\ref{rmk:conju}.
Notice that the coordinate transformations
correspond to a conjugacy by an orientation-preserving diffeomorphism,
so it does not change the signatures of the germs,
see Lemma~\ref{lemma:A,S conjugacy}.

Now, $\{p_i\}$ and $f$ satisfy the assumptions of 
Lemma~\ref{lemma:r-flat 1} and the center germs $\{F_i\}$ of $\{p_i\}$
satisfy the assumptions of Proposition~\ref{prop:germ}.
We apply Lemma~\ref{lemma:r-flat 1},
which gives us germs $\{G_i\}$.
Then we apply Proposition~\ref{prop:germ}, which gives us
germs $\{H_i\}$ arbitrarily close to the identity and large numbers $n_1,\ldots, n_8$.
Then, the conclusion of Lemma~\ref{lemma:r-flat 1} gives us
 a diffeomorphism $h$ which is $C^\infty$ close to the identity map
 such that $h \circ f$ admits a periodic point $\bar{p}$
 whose orbit is contained in a given neighborhood $V$ of $\Lambda$
 and passes through given neighborhoods $U_{p_1}$ and $U_{p_8}$.
Since the center germ of $\bar{p}$ is given by (\ref{asteqbf}),
 by Proposition~\ref{prop:germ},
 we know that $\bar{p}$ is a $(k+1)$-flat periodic point.

Let us discuss the signature of 
the Schwarzian derivative in the case $k=2$. 
In this case, from Proposition~\ref{prop:germ}
we know that $\tau^{\Per}_S(\bar{p}) = \tau^{\Per}_S(p_1)$.
Thus, if we exchange $p_1$ and $p_3$ at the very beginning,
then we obtain the conclusion with the equality  
$\tau^{\Per}_S(\bar{p}) = \tau^{\Per}_S(p_3)$.
Since $p_1$ and $p_3$ have opposite Schwarzian derivative,
it follows that we can choose the signature $\tau^{\Per}_S(\bar{p})$ as we want.

Finally, by Lemma~\ref{lemma:inherit tangle} we can guarantee that
 if we chose the neighborhood $V$ and the neighborhoods $U_{p_1}$ and $U_{p_8}$
 sufficiently small,
 the orbit of $\bar p$ is linked to the blender. 
\end{proof}

It remains to prove the two lemmas. Let us complete their proofs. 
\begin{proof}
[Proof of Lemma~\ref{lemma:network}]
First, let us create a perturbation
 which produces the connection between
 $W^{uu}(p_1, U_{\cC})$ and $W^{ss}(p_{2}, U_{\cC})$.
Given the neighborhood $V$ of $\{p_1,\dots,p_8\}$,
 we choose a neighborhood $V_i$ of $p_i$ for each $i=1,\dots,8$
 such that $V_i \subset V$ and $V_i \cap V_j=\emptyset$ if $i \neq j$.
By assumption, since $p_1$ and $p_2$ are linked to the blender,
 $W^{uu}(p_1, U_{\cC})$ contains a disk in $\cD^u$
 and $W^{ss}(p_2, U_{\cC})$ contains a disk in $\cD^u$
 where $(\cD^s, \cD^u)$ are the disk system of the blender.
Lemma~\ref{lemma:connecting}, together with the fact that $p_1$ and $p_2$ are 
outside $U_{\mathrm{bl}}$,
 gives us a point $x_1 \in V_1 \cap (W^{uu}(p_1,U_{\cC}) \setminus \{p_1\})$
 and an arbitrarily small perturbation $h_1$ 
 such that there exists a point of heteroclinic connection
 $q_1 \in W^{uu}(p_1, U_{\cC}; h_1 \circ f)
 \cap W^{ss}(p_{2}, U_{\cC}; h_1 \circ f)$.
Notice that, $h_1$ can be chosen in such a way that
 the support of $h_1$ does not intersects with $\bigcup_{i=1}^8 \cO(p_i)$.
Indeed, the support of $h_1$ can be chosen so that
it is concentrated in a very small neighborhood of the point $x_1$
 in $W^{uu}(p_1, U_{\cC}) \setminus \{p_1\}$,
 which is disjoint from  $\bigcup_{i=1}^8 \cO(p_i)$.

Then, we again apply Lemma~\ref{lemma:connecting}
 for the pair of points $p_2$ and $p_3$
 and the map $h_1\circ f$ as follows.
Since $p_2$ and $p_3$ are linked to the blender,
 we can apply Lemma~\ref{lemma:connecting}
 and obtain a point
 $x_2 \in V_2 \cap (W^{uu}(p_2, U_{\cC},h_1 \circ f) \setminus \{p_2\})$
 and an arbitrarily small perturbation $h_2$ 
 whose support is contained in a very small neighborhood of $x_2$ 
 such that $h_2 \circ h_1 \circ f$ has a heteroclinic point
 $q_2 \in W^{uu}(p_2, U_{\cC}; h_2\circ h_1 \circ f)
  \cap W^{ss}(p_{3}, U_{\cC}; h_2\circ h_1 \circ f)$.
Notice that the support of $h_2$ can be chosen so that it is disjoint from $\cO(q_1)$
 since $x_2$ is a point of $W^{uu}(p_2;h_1 \circ f)$
 and it is disjoint from $\cO(q_1) \subset W^{uu}(p_1;h_1 \circ f)$.
Thus, the perturbation by the composition of $h_2$ does not 
affect the previous connection
 $q_1 \in W^{uu}(p_1, U_{\cC}; h_1 \circ f) \cap
 W^{ss}(p_{2}, U_{\cC}; h_1 \circ f)$.

We continue this construction and obtain $h_3,\dots,h_8$.
Then the diffeomorphism $h := h_8 \circ \cdots \circ h_1$ satisfies
 the  desired conditions. 
\end{proof}

\begin{proof}
[Proof of Lemma~\ref{lemma:r-flat 1}]
First, by shrinking $U_{p_i}$,
 we may assume that $U_{p_i} \cap U_{p_j} = \emptyset$ for $i\neq j$.
We also assume that
 we have $f^k(U_{p_i})  \subset V$ for every $i$ and $0\leq k \leq \pi_i$,
 where $\pi_{i}$ is the period of $p_{i}$.
For each $i$, we take a $C^{\infty}$ Takens coordinate chart $(\vphi_i, U_i)$.
Assume $U_i \subset U_{p_i}$.
We fix a polyball $B_{\dd,i} \subset \vphi_i(U_i)$ such that 
 $f_{i}=\vphi_i \circ f^{\pi_{i}} \circ \vphi_i^{-1}$ is
 in the Takens $C^{\infty}$ standard form on $B_{\dd, i}$,
 that is, for $(x, y, z) \in B_{\dd,i}$,
 $f_{i}(x,y,z)=(F_{i}(x),A^s_{i}(x)y, A^u_{i}(x)z)$
 where $A^s_{i} :B_{\dd, i}^c \ra \GL(\RR^{d_s})$
 and $A^u_{i}: B_{\dd, i}^c \ra \GL(\RR^{d_u})$,
 are matrices which depend $C^\infty$-smoothly on $x \in B_{\dd, i}^c$.
We denote by $\vphi_{i}^x$ the $x$-coordinate function of $\vphi_{i}$.

For each $i=1,\ldots 8$, 
we take heteroclinic points
 $q^s_{i} \in \vphi_{i}^{-1}(B_{\dd,i}^s) \cap \cO(q_{i-1})$
(we put $q_0 = q_8$) and
 $q^u_{i} \in \vphi_{i}^{-1}(B_{\dd,i}^u) \cap \cO(q_{i})$. 
For each $i$,
 we can choose a small polyball $B_{i} \subset B_{\dd, i}$ such that 
\begin{equation*}
\vphi^{-1}_{i}(\overline{B_{i}}) \cap \Lambda = 
\{ p_i \} \cup \{ f^{m\pi_i}(q^s_{i}), f^{-n\pi_i}(q^u_{i}) \mid 
m \geq M_s, n \geq M_u\}
\end{equation*}
 for some non-negative integers $M_s$ and $M_u$.
By replacing $q^u_{i}$ with $f^{-n_u\pi_i}(q^u_{i})$
 and $q^s_{i}$ with $f^{n_s\pi_i}(q^s_{i})$,
 we may assume $M_s = M_u =0$. So, from now on, we assume that
\begin{equation*}
\vphi^{-1}_{i}(\overline{B_{i}}) \cap \Lambda = 
\{ p_i \} \cup \{ f^{m\pi_i}(q^s_{i}), f^{-n\pi_i}(q^u_{i}) \mid 
m \geq 0, n \geq 0\}
\end{equation*}
We also choose another smaller polyball $B'_{i}$
 satisfying $\overline{B'_i} \subset B_i$ such that
\begin{equation*}
\vphi^{-1}_{i}(\overline{B'_{i}}) \cap \Lambda = 
\{ p_i \} \cup \{ f^{m\pi_i}(q^s_{i}), f^{-n\pi_i}(q^u_{i}) \mid 
m \geq 1, n \geq 0\},
\end{equation*}
In other words, $B'_{i}$ is a polyball shrunk
 in such a way that it does not contain the point
 $q_i^s$ but contains all the other points of $\Lambda \cap B_i$.
We fix smooth bump functions $\rho_i(x, y, z): B_i \to \mathbb{R}$ such that:
\begin{itemize}
\item $\rho_i(x, y, z) =1$ if $(x, y, z) \in B'_i$.
\item $\rho_i(x, y, z) =0$ if $(x, y, z)$ is near the boundary of $B_i$.
\item $\rho_i(x, y, z) =0$ near $q_i^s$.
\end{itemize}
In a way similar to Lemma \ref{lemma:linearization},
for every $i$ we find integers $N_{i} \geq 1$ such that
\begin{equation*}
 f^{N_{i}}(q^u_{i})=q^s_{i+1} 
\end{equation*}
 and $C^\infty$ local maps  
 $l^s_{i}:(\RR,0) \ra (B_{i}^{cs},q^s_{i})$ and
 $l^u_{i}:(\RR,0) \ra (B_{i}^{cu},q^u_{i})$ such that
\begin{equation*}
\vphi_{i}^x \circ l^s_{i}(t)=\vphi_{i}^x \circ l^u_{i}(t)=t, 
\end{equation*}
and 
\begin{equation*}
f^{N_{i}} \circ l^u_{i}(t) =l^s_{i+1}(G_{i}(t)) 
\end{equation*}
 for some orientation-preserving germs $G_{i} \in \Diff_\loc^\infty(\RR,0)$.

Take $y^s_{i} \in \RR^{d_s}$ and $z^u_{i} \in \RR^{d_u}$ 
 for each $i$ such that
 $\vphi_{i}(q_{i}^s)=(0,y^s_{i},0)$ and $\vphi_{i}(q_i^u)=(0,0,z^u_{i})$ hold.
Take a neighborhood $\cU^\frac{1}{8}$ of the identity map in $\Diff^\infty(M)$ 
 such that $g_{1} \circ \dots \circ g_{8} \in\cU$ 
 for any $g_{1},\dots,g_{8} \in \cU^\frac{1}{8}$.
 
For a polynomial map $H \in \cP^{r+1}(\RR,0)$,
 let $\gamma_{i, H}$ ($i=1,\dots,8$) be a diffeomorphism such that
$\gamma_{i, H}= \Id$ outside $\vphi_i^{-1} (B_i)$,
 while inside $\vphi_i^{-1} (B_i)$ the map $\gamma_{i, H}$ satisfies
\begin{equation*}
 \vphi_{i} \circ \gamma_{i,H} \circ \vphi_{i}^{-1}(x,y,z)
 =( \rho_i (x, y, z) (H(x) -x) + x, y, z).
\end{equation*}
for all $(x,y,z) \in B_i$.

Let $\cN_i$ be a neighborhood of the identity map in $\cP^{r+1}(\RR,0)$
 such that if $H\in \cN_i$, then $\gamma_{i, H} \in \cU^{\frac{1}{8}}$
 and the pinching conditions in the assumption of Lemma~\ref{lemma:r-flat 2}
 holds with $r=k+1$ for $A^s=A^s_i$, $A^u=A^u_i$ and $F_c=H\circ F_i$.

Notice that the following holds:
\begin{itemize}
\item On $B'_i$, we have
 $ \vphi_{i} \circ \gamma_{i,H}  \circ \vphi_{i}^{-1}(x,y,z)
   = ( H(x), y, z)$.
Thus, $\vphi_i$ is a Takens coordinate chart
 for the map $(\gamma_{i,H} \circ f)^{\pi_i}$ near $p_i$.
\item The point $p_{i}$ is a $(k+1)$-partially hyperbolic fixed point
 for $(\gamma_{i,H} \circ f)^{\pi_{i}}$.
\item On $B^{su}_{i}$, we have 
 $\vphi_i\circ \gamma_{i,H} \circ \vphi^{-1}_i|_{B^{su}_{i}} = \mathrm{Id}$.
In particular,
 perturbing the map $f$ by taking the composition with $\gamma_{i,H}$
 does not affect the orbits contained in $B_i^{su}$. 
\item Since $\gamma_{i,H} \circ f = f$
 near $\{f^m(q^u_i) \mid  0 \leq m \leq N_i \}$,
 the curves $\{l_i^s\}, \{l_i^u\}$ and the germs $\{G_i\}$
 are not affected when $f$ is composed with $\gamma_{i,H}$.
\end{itemize}
Now, we apply Lemma~\ref{lemma:r-flat 2} to
 $\hat f=\vphi_{i} \circ (\gamma_{i,H} \circ f)^{\pi_{i}} \circ \vphi_{i}^{-1}$
 letting $r = k+1$:
for each $i=1,\dots,8$,
 we obtain a sequence of diffeomorphisms
 $\{h_{i,H,n}\}_{n \geq 1}$ of $B_{i}$ such that
 $\supp(h_{i,H,n})$ converges to the pair of points
 $\{f^{\pi_{i}, q_{i}^u}(q_{i}^s)\}$,
 the map $h_{i,H,n}$ converges to the identity map as $n\to +\infty$
 in the $C^\infty$ topology, and
\begin{equation*}
 ( h_{i,H,n} \circ \vphi_i \circ (\gamma_{i, H} \circ f)^{\pi_{i}}
  \circ \vphi_i^{-1})^{ n} \circ l_{i}^s(t)
 =l_{i}^u((H \circ F_{i})^n(t))+o(t^{k+1})
\end{equation*}
for sufficiently large $n$ (say, for $n \geq N_i$).

We define $\tilde{h}_{i, H, n} =\vphi^{-1}_i \circ h_{i, H, n} \circ \vphi_i$
 on $\vphi^{-1}_i(B_i)$
 and extend it to the whole of $M$ as the identity map.
Then, $\tilde{h}_{i, H, n}\in \mathrm{Diff}^{\infty}(M)$.
Put $\cN = \cap_{i=1}^8 \cN_i$ and $N = \max_{i=1,\ldots,8} \{N_{i}\}$.
Now, take any maps $H_{i}\in \cN$ and any integers $n_{1},\dots,n_{8} \geq N$.
Put $\bar{F}= G_{8} \circ (H_{8} \circ F_{8})^{n_{8}} \circ
 \dots \circ G_{1} \circ (H_{1} \circ F_{1})^{n_{1}}$
 and $\Pi = \left(\sum_{i=1}^8 \pi_{i} n_{i}+N_{i}\right)$.
Then by Lemma~\ref{lemma:r-flat 2},
 if $n_{1},\dots,n_{8}$ are sufficiently large,
 the diffeomorphism 
 $\tilde{h}=(\tilde{h}_{1, H_1, n_{1}} \circ \gamma_{1, H_1}) \circ
  \dots \circ  (\tilde{h}_{8, H_8, n_{8}} \circ \gamma_{8, H_8})$ 
 satisfies $\tilde{h} \in \cU$, 
 $\supp(\tilde{h}) \subset  \bigcup_{i=1}^8 U_{p_i}$,
\begin{equation*}
 (\tilde{h} \circ f)^{\Pi} \circ l_{1}^s(t)
 =l_{1}^s(\bar{F}(t))+o(t^{k+1}),
\end{equation*}
 and $q_1^s$ is a periodic point of $\tilde{h} \circ f$ of period $\Pi$ such that
\begin{gather*}
\cO(q_{1}^s;\tilde{h} \circ f) \subset
 \bigcup_{i=1}^8\left(\{f^n(q_{i}^u) \mid 0 \leq n \leq N_{i}\}
 \cup  \left(\bigcup_{j=0}^{\pi_i}
   f^j\left(\vphi_{i}^{-1}(B_{i})\right) \right)\right) \subset V, \\
  \cO(q_{1}^s;\tilde{h} \circ f) \cap U_{p_1} \neq \emptyset,
 \qquad \cO(q_{1}^s ;\tilde{h} \circ f) \cap U_{p_8} \neq \emptyset,
\end{gather*}
see (\ref{asteqbf}) for the definition of $\bar{F}$.
Therefore, the point $\bar{p} = q_{1}^s$
 and the diffeomorphism $\tilde{h}$ satisfy the required conditions.
\end{proof}

%
%

\section{Examples}
\label{sec:examples}
In this section, we give a construction of
 partially hyperbolic diffeomorphisms
 which satisfy the hypotheses of the main theorem.
We also give examples
 which elucidate the importance of the conditions on the nonlinearity
 and the Schwarzian derivatives. 

\subsection{Fundamental construction}\label{funds}
In this subsection, we describe a construction
of a partially hyperbolic diffeomorphism from a one-step skew-product.

We begin with the Shub-Wilkinson's example
 of a partially hyperbolic diffeomorphism on the 3-torus
 $\mathbb{T}^3 = \mathbb{T}^2 \times \mathbb{T}^1$,
 which has the following form:
\begin{equation}
\label{shwi}
 (x, y) \in \mathbb{T}^2 \times \mathbb{T}^1 \mapsto (Bx, f_{x}(y)),
\end{equation}
 where $B$ is a hyperbolic toral automorphism
 and $f_{x}$ is a smooth family of diffeomorphisms
 of the circle $\mathbb{T}^1$.
Let $m$ be a positive integer.
We choose the matrix $B$ such that the following holds:
\begin{itemize}
\item there exist $m$ fixed points $\{P_i\}$
 of $B : \mathbb{T}^2 \to \mathbb{T}^2$,
\item for each $i=1,\dots,m$ there is a parallelogram
 $R_i \subset \mathbb{T}^2$ with the center at $P_i$
 whose edges are parallel to the eigenvectors of $B$,
 such that different parallelograms $R_i$ are mutually disjoint
 and they behave in Markovian fashion under $B$: for each $i$ and $j$,
 $B(R_i) \cap R_j$ is a (non-empty) disjoint union
 of subrectangles of $R_j$ and $B(R_i)$ crosses $R_j$ properly.
\end{itemize}
One can see that these properties persist,
 with the same set of rectangles, when $B$ is replaced
 by $B^n$ for any $n\geq 1$.

Let $\{g_i\}_{i = 1, \ldots, m}$ be
 a family of orientation-preserving diffeomorphisms
 of the circle $\mathbb{T}$.
Consider a diffeomorphism $\tilde{F}$ of $\mathbb{T}^3$ by (\ref{shwi})
which satisfies $f_x(y) = g_{i}(y)$ if $x \in R_i \subset \mathbb{T}^2$.
The map $\tilde{F}$ restricted to the locally maximal invariant set
 of $ \left( \sqcup R_i \right) \times \mathbb{T}^1$ acts
 as a one-step skew-product map.

It is not difficult to see that
 there is an extension of the locally defined map $\tilde{F}$
 to the ambient space $\mathbb{T}^3$ in such a way that 
 the extended map defines a diffeomorphism of $\mathbb{T}^3$
 keeping the condition (\ref{shwi}).
Let $F$ be such a diffeomorphism.
It has an $F$-invariant foliation $\mathcal{F}^c = \mathcal{F}^c (F)$
 given by $\{ \{ x \} \times \mathbb{T} \mid x \in \mathbb{T}^2 \}$.
Notice that, by replacing $B$ with its power,
 we can assume that the foliation $\mathcal{F}^c$ is
 $k$-normally hyperbolic for $F$ for any fixed integer $k >0$.

In the following, we are interested in the behavior of
 perturbations $G$ of $F$, that is,
 the $C^k$ diffeomorphisms $G$ which are $C^k$ close to $F$
 in $\mathrm{Diff}^k(\mathbb{T}^3)$. 
Recall two theorems from \cite{HPS}.

\begin{prop}
[\cite{HPS}, Theorems 7.1, 7.4]\label{prop:hps-persist}
Let $k \geq 2$.
If $G$ is $C^k$ close to $F$
 then there is a center-leaf conjugacy $h$ between $F$ and $G$.
Namely, there is a homeomorphism $h : \mathbb{T}^3 \to \mathbb{T}^3$
 such that $\cF(G) := h(\mathcal{F}^c)$ is
 a $G$-invariant $C^k$ lamination which is $k$-normally hyperbolic.
\end{prop}
\begin{prop}
[\cite{HPS}, Theorem 6.1, 6.8]\label{prop:hps-close}
If $F$ is $C^k$-close to $G$ then each perturbed leaf $h(L)$ 
is uniformly $C^k$-close to the original leaf $L$ of $\mathcal{F}^c$.
\end{prop}

\subsection{The set $\cW^{\infty}$ is non-empty}
\label{ss:non-emp}
In this section, we give an example of a diffeomorphism in $\cW^{\infty}$,
namely, a $C^\infty$-diffeomorphism which satisfies conditions
of the Main Theorem. See Section~\ref{sec:outline} for the definition of $\cW^{\infty}$.

We use the convention $\mathbb{T}^1 = \mathbb{R} / (10 \mathbb{Z})$.
Choose four orientation-preserving diffeomorphisms
 $g_1, g_2, g_{b+}, g_{b-}$ of $\mathbb{T}$
 which satisfy the following conditions for all $x \in [-4, 4]$:
\begin{itemize}
\item[(C1)] $g_{i}(x)$ ($i =1, 2$) has
 a repeller-attractor pair contained in $(-2, 2)$ 
 with a heteroclinic connection $q^i$.
In other words, $g_i$ has two fixed points $p_{i, +}$ and $p_{i, -}$ such that
 $p_{i, +}$ is repelling, $p_{i, -}$ is attracting and
 there exists a point $q^i \in W^u(p_{i, +}) \cap W^s(p_{i, -})$. 
We also assume that $\{p_{i, +} , p_{i, -}\} \cup \cO(q^i) \subset (-2, 2)$.
\item[(C2)] The signs of $A$ and $S$ at $q^1$ and $q^2$
 are opposite (see \cite[Section 2.1]{AST} for the definitions of signs).
\item[(C3)] $g_{b+}$ and $g_{b-}$ have a $C^1$-robust blender
 (in the sense of \cite[Section 2.1]{AST}) containing the interval $[-3, 3]$.
Furthermore, $g^{\pm1}_{b+}([-2, 2]) \subset (-3, 3)$
 and $g^{\pm1}_{b-}([-2, 2]) \subset (-3, 3)$
 (for instance take $g_{b\pm}(x)  = \displaystyle 0.99(x\pm4) \mp4$). 
\end{itemize}
The construction of the maps $g_{1,2}$ can be done
 in the same way as in \cite[Section 8]{AST}. 

For the four maps $g_1, g_2, g_3=g_{b+}, g_4=g_{b-}$, 
 we take the skew product diffeomorphism $F$
 as described in Section \ref{funds}.
We assume that $F$ is $1$-normally hyperbolic to the center foliation $\cF^{c}$.
Furthermore, we choose $F$ such that
 it preserves the orientation of the center direction.
Thus, we can define the transverse pair of cone fields
 and the $c$-orientation which are invariant under $F$.
Recall that these properties are $C^1$ robust.

Let us check that $F$ satisfies the assumptions of the Main Theorem.
First, as it is explained in \cite[Section 5]{BDV},
 we can see that $F$ has a $C^1$robust blender
 in $U_{\mathrm{bl}} := (R_{b+} \sqcup R_{b-}) \times [-3, 3]$.
Namely, we can take 
\begin{itemize}
\item an open set $\mathcal{D}^u \subset \mathrm{Emb}^{1}(I, U_{\mathrm{bl}})$
 which contains the set of segments (an embedding of a one-dimensional disk) 
 parallel to strong unstable direction in $U_{\mathrm{bl}}$, and
\item an open set $\mathcal{D}^s \subset \mathrm{Emb}^{1}(I, U_{\mathrm{bl}})$
 which contains the set of segments parallel to strong stable direction
 in $U_{\mathrm{bl}}$,
\end{itemize}
 such that
\begin{itemize}
\item for every $G$ sufficiently $C^1$ close to $F$, the following holds:
 for every $\sigma \in \mathcal{D}^s$, 
 the set of segments in $\mathcal{D}^u$ which is tangled with $\sigma$
 is dense in $\mathcal{D}^u$.
\end{itemize}

Next, we check that the conditions on the heteroclinic connections. 
In each parallelogram $R_{i} \times \mathbb{T}^1$ ($i=1, 2$), 
 there is a fixed fiber corresponding to the point $P_i\in \mathbb{T}^2$.
In this fiber, there is a hyperbolic fixed point
 which corresponds to the repeller $p_{i,+}$ of $g_i$.
We denote this fixed point as $p^{i}_{u}$.
In each fixed fiber, there is another hyperbolic fixed point,
 corresponding to the attractor $p_{i,-}$ of $g_i$,
 which is attracting in the fiber direction
 and is heteroclinically connected to $p^i_{u}$.
We denote this fixed point as $p^i_{s}$.
 They all have one-dimensional strong unstable direction
 and one-dimensional strong stable direction. 
Thus the pair $(p^i_{u}, p^i_{s})$ satisfies the condition
 about the dimension of the stable and unstable subspaces.
Furthermore, all the fixed points admit dominated splitting of type $(1, 1, 1)$.

Since the blender we are interested in is
 in the region $(R_{b+} \sqcup R_{b-}) \times \mathbb{T}^1$,
 the fixed points $\{p^i_u, p^i_s\}$ are outside the blender. 
Also, the pairs $\{p^i_u, p^i_s\}$ have heteroclinic connections $q^i$
 and their signatures are opposite by construction. 

Let us check the connection between the stable manifold of $p^i_u$
 and the blender.
In the region $R_{i}\times \mathbb{T}^1$,
 the stable manifold $W^{ss}(p^i_u)$ is given
 by a segment which is parallel to the strong stable eigenvector at $p^i_u$.
In particular, we have
 $W^{ss}_{\mathrm{loc}}(p^i_u) \subset R_{i} \times [-2, 2]$,
 where $W^{ss}_{\mathrm{loc}}(p^i_u)$ denotes
 the connected component of $W^{ss}(p^i_u) \cap R_{i} \times [-2, 2]$ 
 containing $p^i_u$. Since $W^{ss}(p^i_u)$ properly crosses
 the region $R_{i}\times \mathbb{T}^1$
 and $B(R_{i}) \cap R_{b\pm} \neq \emptyset$,
 it follows by condition (C3) that $W^{ss}(p^i_u)$ contains a segment
 in $R_{b\pm} \times (-3, 3)$
 which, by construction, is parallel to the strong stable direction. 
Since this segment belongs to $\mathcal{D}^s$,
 we have confirmed the condition about the connection
 between $p^i_u$ and the blender. 
Similarly, one can check the condition
 about the connection between $p^i_s$ and the blender. 

Thus we have checked that the map $F$ satisfies
 all the requirements of the theorem.

\subsection{An example of $C^1$-generic super-exponential growth
 which is not $C^2$-generic}
\label{c123}
In this section,
 we give an example of a partially hyperbolic diffeomorphism $F$
 such that for every $C^2$-close map
 the growth of the number of periodic points is at most exponential,
 while a $C^1$-generic in any sufficiently small $C^1$-neighborhood of $F$
 exhibits a super-exponential growth.

In this section we use the convention $\mathbb{T}^1 = \mathbb{R} / (10\mathbb{Z})$ in this section.
Let $\mathcal{A}$ be the set of $C^2$ diffeomorphism $f$ of
 the circle $\mathbb{T}^1$ which have the following properties:
\begin{itemize}
\item $f$ preserves the orientation of $\mathbb{T}^1$;
\item there exists $k >0$ such that $f^{-k}([-4, -1]) \subset (-3, -2)=(7,8)$;
\item there exists $k >0$ such that $f^{k}([4, 7]) \subset (5, 6)=(-5,-4)$;
\item $f'(y) >1$ for every $y \in [-3, -2]$;
\item $ 0< f'(y) <1$ for every $y \in [5, 6]$;
\item $A(f)(y) >0$ for $y\in [-1, 4]$.
\end{itemize}

\begin{lemma}
\label{lemma:2-super}
Let $\{f_x\}_{x \in \mathbb{T}^2}$ be a smooth family of diffeomorphisms 
 such that $f_x \in\mathcal{A}$ for all $x \in \mathbb{T}^2$. 
Let $F \in \mathrm{Diff}^2(\mathbb{T}^3)$ be
 the corresponding map $F(x, y) = (Bx, f_x(y))$.
If $F$ is $2$-normally hyperbolic to the flat center leaves $\cF^c(F)$,
 then there exists a $C^2$-neighborhood $\mathcal{N}$ of $F$
 in $\mathrm{Diff}^2(\mathbb{T}^3)$ such that 
 for every $G \in \mathcal{N}$
 the number of periodic points grows at most exponentially with period. 
\end{lemma}
\begin{proof}
Let $G$ be a $C^2$ diffeomorphism of $\mathbb{T}^3$
 which is $C^2$ close to $F$.
By Proposition~\ref{prop:hps-persist},
 the diffeomorphism $G$ has a $C^2$ center lamination $\cF^c(G)$
 and there is a global leaf conjugacy $h: \mathbb{T}^3 \to \mathbb{T}^3$
 between $\cF^c(F)$ and $\cF^c(G)$.
Notice that to count the number of periodic points of $G$,
 we only need to consider periodic leaves of $\cF^c(G)$.

By definition of $F$ and a compactness argument, 
one can see that there exists $k_1 >0$ and $\delta_1>0$ such that
\begin{equation*}
F^{k_1}(\mathbb{T}^2\times [4-\delta_1, 7+\delta_1])
 \subset \mathbb{T}^2\times (5, 6).
\end{equation*}
Since this is a $C^0$-robust property, 
 the same holds for every $G \in \mathrm{Diff}^2(\mathbb{T}^3)$ 
 sufficiently $C^2$-close to $F$.
Notice that, since $F$ is uniformly contracting
 in the center direction everywhere in $\mathbb{T}^2\times (5, 6)$
 and, by Proposition~\ref{prop:hps-close},
 the center leaves of $G$ are uniformly $C^2$-close to the center leaves
 of $F$, we see that $G$, too,
 uniformly contracts the center leaves everywhere
 in $\mathbb{T}^2\times [5, 6]$.
It follows that each periodic center leaf of $G$ has only one periodic orbit
 that passes through $\mathbb{T}^2\times [4-\delta_1, 7+\delta_1]$.
By a similar argument for $F^{-1}$,
 we obtain that there exists $\delta_2 >0$ such that 
 for $G$ sufficiently close to $F$,
 every periodic leaf has only one periodic orbit
 that passes through $\mathbb{T}^2\times [-4-\delta_2, -1+\delta_2]$.

Let us consider the periodic orbits lying in the region
 $X =\mathbb{T}^2 \times [-1 + \delta_2, 4 -\delta_1]$.
Recall that if $x \in X$ is a periodic point of $G$
 then we have $G^n(x) \in X$ for every $n \in \mathbb{Z}$,
 since the region outside $X$ is
 in the basin of an attracting region $\mathbb{T}^2\times (5, 6)$
 or a repelling region $\mathbb{T}^2\times (-3, -2)$. 
Now, recall that the action of $F$ on any central fiber
 is orientation-preserving and the value $A(f_x)$
 is positive for $(x,y)\in X$.
This means that the restriction of $F$
 on the intersection of any central fiber with $X$
 is a monotonically increasing, convex function.
 The same holds for the map $G$,
 as its restriction to center fibers is $C^2$-close to that of $F$.
As a composition of increasing convex functions is also convex,
 it follows that every return map on a periodic fiber is a convex function
on the intersection of the fiber with $X$.
Since a convex map on an interval can have no more than 2 fixed points,
we see that each periodic leaf has at most two periodic orbits in $X$. 

As a consequence, for $G$ which is sufficiently close to $F$ in the $C^2$-topology,
 each periodic leaf contains at most for periodic orbits.
Since periodic leaves correspond to the periodic points
 of the hyperbolic toral automorphism $B$,
 and the number of those periodic points grows exponentially with period,
 we obtain that the growth number of periodic points of $G$ cannot be
 super-exponential. 
\end{proof}

On the other hand,
 one can choose the map $F$ in this lemma
 in such a way that a $C^1$-generic map
 in its arbitrarily small $C^1$-neighborhood exhibits super-exponential growth.
This is done as follows.
We choose the base map $B$ so that
 it has three Markovian parallelograms $(R_i)_{i=1,2,3}$.
Then consider the maps on $[-1, 4]$ given as follows:
\begin{align*}
f_1(y) &= (1-3\varepsilon)y + \varepsilon y^2\\
f_2(y) &= (1-3\varepsilon)(y-1) +1 + \varepsilon (y-1)^2\\
f_3(y) &= (1-3\varepsilon^2)y + \varepsilon (y- 1/3)(y-2/3)
\end{align*}
where $\varepsilon >0$ is a small real number. 
They are orientation-preserving diffeomorphisms on their images
 and satisfy $A(f) >0$ on $[-1, 4]$.
Then, we extend these maps as diffeomorphisms of $\mathbb{R}/10\mathbb{Z}$
 so that the resulted maps belong to $\mathcal{A}$.
We take a smooth family $\{f_x( \cdot )\}$
 so that $f_x(y) = f_i$ in $R_i$ and $f_x( \cdot ) \in \mathcal{A}$
 for every $x \in \mathbb{T}^2$.
The construction of such an extension is not difficult
 and the details are left to the reader. 

By replacing the base map with higher power if necessary, 
we see that the corresponding skew-product map $F$
satisfies the assumption of Lemma~\ref{lemma:2-super}.
Thus every map sufficiently $C^2$-close to $F$ has
at most exponential growth of number of periodic points. 

Meanwhile, one can check that $F\in\cW^1$.
Indeed, the maps $f_1$ and $f_2$ have a $C^1$-robust blender
 in an interval contained in $(0, 1)$ 
by a similar reasoning as in Section~\ref{ss:non-emp}.
The map $f_3$ has an attractor-repeller pair near $y= 1/3, 2/3$.
Thus, the map has a heteroclinic pair which is connected to the blender.
By Main Theorem for $r=1$, 
 we see that in a $C^1$-neighborhood of $F$,
 $C^1$-generically we have super-exponential growth
 of the number of periodic points. 


\subsection{An example of $C^2$-generic super-exponential growth which is not $C^3$-generic}

Below we use the convention $\mathbb{T}^1 = \mathbb{R}/6\mathbb{Z}$.
Let us consider the set $\mathcal{B}$ of $C^3$-diffeomorphism $f$ of the circle
$\mathbb{T}^1$ which satisfy the following conditions:
\begin{itemize}
\item $f$ preserves the orientation of $\mathbb{T}^1$;
\item there exists $k >0$ such that $f^{-k}([2, 5]) \subset (3, 4)$;
\item $f'(y) <1$ for $y \in (3, 4)$;
\item $f([-1, 2]) \subset (-1, 2)$;
\item $S(f)(y) <0$ for $y \in (-1, 2)$.
\end{itemize}
As in the previous subsection, we have the following:
\begin{lemma}
\label{lem:last}
Let $\{f_x\}_{x \in \mathbb{T}^2}$ be  a smooth family
 of diffeomorphisms of $\mathbb{T}^1$ such that
 $f_x \in \mathcal{B}$ for every $x \in \mathbb{T}^2$. 
Let $F$ be the corresponding skew product diffeomorphism
 such that the center foliation $\cF^c(F)$ is $3$-normally hyperbolic.
Then, there exists a $C^3$-neighborhood $\mathcal{W}$ of $F$
 in $\mathrm{Diff}^3(\mathbb{T}^3)$ such that 
 every $G \in \mathcal{W}$ has at most exponential growth
 of the number of periodic points.
\end{lemma}
\begin{proof}
Let $G$ be a diffeomorphism sufficiently close to $F$ in $C^3$.
 Since the region $\mathbb{T}^2\times [2, 5]$ is forward-invariant
 and the map is contracting on $\mathbb{T}^2\times [3, 4]$ in the 
 center direction,
 the number of periodic points whose orbits
 touches the set $\mathbb{T}^2\times [2, 5]$
 grow at most exponentially with period
 (see a similar argument in Lemma \ref{lemma:2-super}).  
So we only need to count periodic orbits
 that lie inside the region $\mathbb{T}^2\times [-1, 2]$.
In this region, the map $F$ restricted to any center fiber has
negative Schwarzian derivative.
Thus, if $G$ is sufficiently close to $F$ then, for every periodic leaf,
the corresponding first-return map has a negative Schwarzian derivative too.
Such maps can have no more than 3 fixed points.
Accordingly, each periodic leaf of the center lamination for $G$
 contains at most three periodic points.
Hence, the number of periodic points grows
 at the same (or lower) exponential rate
 as the number of periodic points
 of the hyperbolic toral automorphism $x\mapsto Bx$.
\end{proof}

Meanwhile, it is possible to take the map $F$ in Lemma \ref{lem:last}
 such that
 it belongs to $\cW^2$. Then Theorem~\ref{thm:precise} tells us that
 a $C^2$-generic map in a small $C^2$-neighborhood of $F$ exhibits
 a super-exponential growth of the number of periodic points.

Again, like in Section \ref{c123},
 we start with the base map $B$
 with three Markovian parallelograms $(R_i)_{i=1,2,3}$.
 We take three maps that act on $[-1, 2]$ as follows:
\begin{align*}
f_1(y) &= (1-\delta)y - \varepsilon y^3,\\
f_2(y) &= (1-\delta)(y-1) + 1 - \varepsilon (y-1)^3,\\
f_3(y) &= y - \varepsilon(y-1/2)(y-1/2-\varepsilon)(y-1/2+\varepsilon),
\end{align*}
where $\varepsilon$ and $\delta$ are some positive real numbers.
Then, it is not hard to check that if $\delta$ is sufficiently small and
 $\varepsilon$ is much smaller than $\delta$,
 then these are diffeomorphisms on their images.
By extending them appropriately,
 we can construct a smooth family of diffeomorphisms $\{f_x\}$
 so that $f_x =f_i$ for $x \in R_i$
 and $f_x \in \mathcal{B}$ for every $x \in \mathbb{T}^2$. 

The maps $f_{1,2}$ create a blender
 (as in the examples discussed in the previous sections).
The map $f_3$ has two repeller-attractor pairs near $y=1/2$. 
Using the criterion in \cite[proposition 8.3]{AST},
 we can see that these attractor-repeller pairs have
 heteroclinic connections with the opposite values of $\tau_A$.
Also, one can see that these fixed points are linked by the blender in $\mathbb{T}^3$.
By replacing the base map if necessary, we can also obtain the $3$-normal hyperbolicity.
Thus, we have constructed $F \in \cW^2$ keeping the assumptions 
of Lemma~\ref{lem:last}.

%
%

\bigskip
\bigskip
\bigskip

\begin{itemize}
\item[]  \emph{Masayuki ASAOKA (asaoka-001@math.kyoto-u.ac.jp)}
\begin{itemize}
\item[] Department of Mathematics, Kyoto University, 
\item[] Kitashirakawa-Oiwakecho, Kyoto 606-8502, Japan
\end{itemize}
\item[] \emph{Katsutoshi SHINOHARA (ka.shinohara@r.hit-u.ac.jp)}
\begin{itemize}
\item[]  Graduate School of Commerce and Management, 
\item[]  Hitotsubashi University,
\item[]  2-1 Naka, Kunitachi, Tokyo 186-8601, Japan
\end{itemize}
\item[] \emph{Dmitry TURAEV (d.turaev@imperial.ac.uk)}
\begin{itemize}
\item[] Department of Mathematics, Imperial College London,
\item[] 180 Queen's Gate, London, UK 
\item[] 4 Lobachevsky University of Nizhny Novgorod, 
\item[] Nizhny Novgorod 603950, Russia
\end{itemize}
\end{itemize}

\end{document}